\documentclass[12pt]{article}

\pdfoutput=1

\usepackage[colorlinks,linkcolor=blue,citecolor=blue,bookmarks=false,pagebackref]{hyperref}
\usepackage{latexsym,amsfonts,bm,epsfig,amsmath,natbib,authblk,amsthm,thmtools,amssymb}
\usepackage{cleveref}
\usepackage{thm-restate}
\usepackage{IEEEtrantools}
\usepackage{dsfont}
\usepackage{graphicx}
\usepackage{xspace}
\usepackage{color}
\usepackage{setspace}
\usepackage{subcaption}
\usepackage{xcolor}

\usepackage{float}
\usepackage{booktabs}
\usepackage[margin=1cm]{caption}
\usepackage{xr}
\usepackage[linesnumbered,ruled,vlined]{algorithm2e}
\usepackage{algorithmic}
\usepackage{bigints}

\usepackage[toc,page,header]{appendix}
\usepackage{minitoc}
\noptcrule %no horizontal lines in toc

\floatstyle{plain}

\definecolor{softgreen}{rgb}{0.0, 0.6, 0.0}

\makeatletter
\newcommand{\myitem}[1]{%
\item[#1]\protected@edef\@currentlabel{#1}%
}
\makeatother

\newcommand{\argmin}{\operatornamewithlimits{argmin\,}}
\newcommand{\arginf}{\operatornamewithlimits{arginf\,}}

\newcommand{\argsup}{\operatornamewithlimits{argsup\,}}
\newcommand{\E}{\mathbb{E}}

\newcommand{\Var}{\mathrm{Var}}

\newcommand{\1}{\mathds{1}}

\newcommand{\MD}{\mathsf{L}}

\mathchardef\mhyphen="2D

\def \y{y_{1:n}}

\def \U{u_{1:m}}

\def \E{\mathbb{E}}

\def \dt {\mathrm{d}}

\newcommand{\KL}{\mathrm{KL}}

\newcommand{\MMD}{\operatorname{MMD}}

\newtheorem{remark}{Remark}
\newtheorem{assumption}{Assumption}
\newtheorem*{assumption*}{Standing Assumption}

\declaretheorem[name=Theorem]{theorem}
\declaretheorem[name=Corollary]{corollary}
\declaretheorem[name=Lemma]{lemma}

\newcommand\blfootnote[1]{%
\begingroup
\renewcommand\thefootnote{}\footnote{#1}%
\addtocounter{footnote}{-1}%
\endgroup
}

\crefname{assumption}{Assumption}{Assumptions}

\begin{document}

\def\spacingset#1{\renewcommand{\baselinestretch}%
{#1}\small\normalsize} \spacingset{1}

\title{Exact Sampling of Gibbs Measures with Estimated Losses}
\author{
\large
David T. Frazier$^{1*}$, Jeremias Knoblauch$^{2\dagger*}$, Jack Jewson$^{3}$, and Christopher Drovandi$^{4}$
}
\date{
\footnotesize
$^1$\textit{Monash University, Econometrics and Business Statistics
Clayton, 3800
Melbourne, VIC, AU 3800}
\newline    $^2$\textit{University College London, Statistical Science
1-19 Torrington Place
London, London, UK WC1E7HB}
\newline
$^3$ \textit{Monash University
Melbourne, VIC, AU}
\newline
$^4$\textit{Queensland University of Technology, School of Mathematical Sciences
GPO Box 2434
Brisbane 4001
Brisbane, QLD, AU 4001}
\\[2ex] 
%\today
\vspace*{-1cm}
}
\maketitle
\spacingset{1.7} % DON'T change the spacing!
\begin{abstract}

\blfootnote{$^\dagger$ Corresponding author (j.knoblauch@ucl.ac.uk). $^*$ Both authors contributed equally. We are thankful for the many helpful comments from the associate editor and referees that have greatly improved the paper. 
} 
%Room 144, 1-19 Torrington Pl, London WC1E 7HB
%}
%\blfootnote{$^*$ Both authors contributed equally.}
%In recent years, the shortcomings of Bayesian posteriors as inferential devices have received increased attention.
%
A popular strategy for ameliorating some of the shortcomings of Bayesian posterior inference is to instead target a Gibbs measure based on losses that connect a parameter of interest to observed data, and which are known to be robust to misspecification.
Existing theory for these procedures treats these losses as being analytically available, but in many situations these losses must be stochastically estimated using pseudo-observations. %and Monte Carlo approximations. 
In these settings, and even under strong assumptions, we show that posteriors based on standard Markov chain Monte Carlo (MCMC) algorithms exhibit strong dependence on the number of these pseudo-observations, and require utilizing a diverging number of pseudo-observations to ensure posterior concentration.
%: unless the number of pseudo-observations diverges sufficiently fast, the resulting posterior will concentrate very slowly. 
%
To remedy this issue, we introduce a modified piecewise deterministic Markov process (PDMP) sampler, and  formally show that its posterior draws have no dependence on the number of pseudo-observations used to estimate the loss within a Gibbs measure.
We verify the practical utility of this approach on five examples spanning intractable likelihoods and intractable losses.
\end{abstract}

\vspace*{-0.25cm}
\textbf{Keywords:} Generalised Bayes, Simulation-based inference, Bayesian asymptotics.

% these lines are needed to create an appendix-specific table of contents
\doparttoc % Tell to minitoc to generate a toc for the parts
\faketableofcontents % Run a fake tableofcontents command for the partocs
\part{} % Start the document part

\section{Introduction}
Bayesian 
inference has long been the gold standard for principled statistical methodology reliant on the inclusion of prior knowledge.
Beyond that, the Bayes posterior provides a natural way for quantifying uncertainty about the data-generating process under study.% \citep[e.g.][]{robert2007bayesian, bernardo2009bayesian}.
However, despite their broad success, standard Bayesian methods can be less well suited to some inferential settings, such as in data-rich environments, complex models, and for a range of machine learning tasks \citep[see][]{berger1994overview, bissiri:etal:2016, knoblauch2019generalized}.
For example, the Bayesian paradigm is severely challenged by intractable likelihoods, where
Markov chain Monte Carlo (MCMC) methods cannot be deployed without first approximating the likelihood function itself \citep[][]{ papamakarios2019sequential, durkan2020contrastive}, or resorting to more computationally demanding strategies like approximate Bayesian computation  \citep[][]{beaumont2010approximate, fearnhead2012constructing, frazier2018asymptotic, frazier2020model}.
Similarly, under model misspecification, the Bayes posterior concentrates towards sub-optimal parameter values \citep{draper1995assessment},   provides miscalibrated parameter inferences \citep{bunke1998asymptotic, walker2013bayesian}, and
produces brittle inferences  \citep{owhadi2015brittleness, jewson2018principles}. 

These issues have led many to resort to a direct change of the inferential target:
rather than focusing on the Bayes posterior, a host of contemporary proposals instead base parameter inferences on 
alternative Gibbs measures that replace likelihood functions with various loss functions. 
The loss function is often chosen to ensure that inferences are robust to certain forms of model misspecification, or can cope with intractable likelihoods \citep[][]{ matsubara2022robust,matsubara2023generalisedDFD,legramanti2022concentration}. Specifically, several popular approaches to inference with Gibbs measures rely on widely celebrated loss functions such as kernel scoring rules (\citealp{pacchiardi2024generalized}), the maximum mean discrepancy (\citealp{park2016k2}),  $\beta$-divergences (\citealp{jewson2023differentially}), or integral probability metrics (\citealp{legramanti2022concentration}). 
However, in each of the above cases, even for simple statistical models, the associated losses are generally not analytically tractable but defined in terms of intractable integrals. 
%
%This intractability is of course compounded in cases where the underlying model is itself analytically intractable. 
%
Computationally, the solution currently favored for this challenge is simple: %when using Gibbs measured based on either an intractable loss or an intractable model, 
the intractable integral which the loss depends on is estimated with a Monte Carlo average using   samples from the posited statistical model, and then used within the resulting posterior sampling algorithm as though the estimate were exact.  

%In this way, the computational challenge faced when conducting inference from Gibbs measure based on intractable losses, or Gibbs measures based on intractable models is precisely the same. 
With this in mind, one can immediately identify two  prevalent approaches for sampling from Gibbs measures based on intractable losses.
The first implements computation via Approximate Bayesian Computation (ABC) \cite[see e.g.][]{bernton2019approximate,legramanti2022concentration}.
As is well-known, ABC methods require specification of a hyperparameter often called the \textit{tolerance} which implements a direct trade-off between computational complexity and accuracy of the resulting posterior approximation. 
In practice, this  means that ABC often fails to  deliver reliable posterior approximations within a reasonable computational budget. %(\citealp{blum:2010}). 
Given this, several authors have argued that a more `naive' approach which replaces the intractable loss  with a Monte Carlo based  approximation using pseudo-observations  simulated from the assumed model may be preferable \citep[see e.g.][]{park2016k2, alquier2023estimation,pacchiardi2024generalized}.
In contrast to ABC, this naive approach poses a tradeoff between accuracy and computation that has not been rigorously investigated for Gibbs posteriors.
In particular, 
it is unclear in general how many simulated pseudo-observations should be used to ensure the resulting approximation is relatively close  to the posterior of interest. 

This latter naive approach, based on using simulated pseudo-observations, can be viewed as a type of pseudo-marginal Markov chain Monte Carlo (P-MCMC) posterior \citep[see e.g.][]{andrieu:roberts:2009}, or more specifically as a posterior obtained under a noisy Monte Carlo approach (\citealp{alquier2016noisy}).  Such an approach will result in a biased approximation of the Gibbs posterior of interest, even when the loss estimator itself is unbiased. For posteriors obtained via these algorithms, we make a useful contribution that furthers the understanding and application of such methods.
We analyze the fixed sample and asymptotic behavior of Gibbs measures based on estimated losses %obtained when applying P-MCMC like algorithms to estimated losses within the Gibbs framework: .
and show that even under favourable assumptions these posteriors are ill-suited for statitsical inference: the approximation error {of the posterior} induced by estimating the loss decays proportionally to the number of pseudo-observations used \textit{within each MCMC step}. Consequently, controlling the resulting approximation error can be  prohibitively expensive in practice.
Our results not only cover `naive' implementations of P-MCMC-type samplers, but more computationally efficient versions, such as correlated or block P-MCMC (\citealp{deligiannidis2018correlated}; \citealp{tran2016block}): such algorithms draw samples from the same target distribution as their naive counterparts, but have improved computational properties and better mixing in general (\citealp{deligiannidis2018correlated}). %Furthermore, this result is valid for Gibbs measured based intractable models, see, e.g., \citet{park2016k2, alquier2023estimation,pacchiardi2024generalized}, as well as intractable loss functions.

This strong dependence on the number of pseudo-observations motivates the necessity of our second contribution: we develop an algorithm that can produce  samples \textit{from the exact Gibbs posterior of interest}, even though the loss on which this Gibbs measure is based is intractable.
The key condition for the validity of this algorithm is that the gradient associated to the Gibbs measure's intractable loss can be estimated without bias.
When this is possible, we develop a novel modification of the piecewise deterministic Markov process (PDMP) sampler in  \citet{bierkens2019zig} and demonstrate that this algorithm produces samples from the Gibbs measure based on the intractable loss.
Unlike P-MCMC-type methods---which must use many pseudo-observations at each iteration to accurately approximate the Gibbs posterior of interest---our PDMP-based algorithm delivers samples from the exact Gibbs posterior so long as at least two pseudo-observations are used at each iteration.
Critically, %Despite the fact that Bayes posteriors with intractable likelihoods and Gibbs posteriors based on intractable loss functions emerge from entirely different statistical tasks, 
{for appropriately chosen losses, }
this proposed algorithm can be used to deliver exact Gibbs posterior inferences in models with intractable likelihoods, which has the interesting side effect of transforming Gibbs measures into a viable and attractive computational tool for intractable likelihood models.

The remainder of the paper proceeds as follows. In \Cref{sec:setup}, we elaborate upon the nature of the problem under study, and introduce two running examples that serve as motivation throughout the remainder of the paper. 
Next, we characterise the rates of convergence when using P-MCMC-type algorithms for Gibbs measures based on estimated losses in  \Cref{sec:theory-new}.
The results suggest that such approaches will be too computationally demanding for this task, and motivate our development of a new zig-zag sampler in \Cref{sec:PDMP}.
We end the paper with \Cref{sec:examples}, where we apply the new sampler to five examples drawn from the literature--two involving intractable models and three involving intractable losses--and show its clear dominance over MCMC-type methods for inference with Gibbs measures constructed using estimated losses\footnote{Code to reproduce our experiments is available in the supplementary material and at 
\url{https://github.com/jejewson/LossEstimationGibbsMeasure}.}.

\section{Motivation and Setup}\label{sec:setup}

Throughout, we consider the finite-dimensional data set $\y=(y_1,\dots,y_n)^\top$ with $y_i\in\mathcal{Y}$ for all $i=1,\dots,n$, and where we assume $y_{1:n} \sim P_0$.
While the true data-generating measure $P_0$ will generally depend on $n$, we suppress this for notational convenience.
Using our data, we model $P_0$ through the class of parametric models 
$
\{P_\theta : \theta \in \Theta\} \subset
\mathcal{P}(\mathcal{Y}^n)\text{ for }\Theta\subseteq \mathbb{R}^{d}
$,
where we write $p_{\theta}$ to denote the  density/mass of $P_{\theta}$, and  once again suppress dependence of $P_\theta$ and $p_{\theta}$  on $n$.
In a Bayesian approach to statistical modelling, we express our prior beliefs about the true state of the world probabilistically through the prior with density $\pi$. 
Given $p_{\theta}$ and $\pi$, the Bayesian approach for quantifying uncertainty about the value of $\theta$ that best describes the observed data proceeds via Bayes' Rule, which yields the Bayes posterior measure with density
\begin{IEEEeqnarray}{rCl}
\pi(\theta\mid \y)=\frac{p_\theta(\y)\pi(\theta)}{\int_\Theta p_\theta(\y)\pi(\theta)\dt\theta}.
\label{eq:bayes-posterior}
\end{IEEEeqnarray}
Conceptually, the Bayes posterior updates our prior beliefs $\pi$ about a parameter $\theta$ with data, and is internally coherent. 

\subsection{Posterior Belief Distributions}

It is easy to generalise the Bayes posterior: letting $\omega>0$ be a scalar and $\MD_n:\Theta \times \mathcal{Y}^n \to \mathbb{R}$ a loss function expressing the information about the parameter $\theta$ contained in $\y$, then so long as $\int_\Theta \exp\{-\omega\cdot \MD_n(\theta)\}\pi(\theta)\dt\theta<\infty$, we can define the Gibbs measure
\begin{IEEEeqnarray}{rCl}
\pi(\theta \mid \MD_n) := \frac{\exp\{-\omega\cdot \MD_n(\theta)\}\pi(\theta)}{\int_\Theta \exp\{-\omega\cdot \MD_n(\theta)\}\pi(\theta)\dt\theta},
\label{eq:gibbs-msr}
\end{IEEEeqnarray}
which recovers the Bayes posterior $\pi(\theta\mid\y)$ when $\MD_n(\theta) = - n^{-1}\log p_{\theta}(\y)$ and $\omega=n$.
Gibbs measures make appearances in many areas of Bayesian methodology, including ABC  \citep{beaumont2010approximate, frazier2018asymptotic}, Probably Approximately Correct (PAC) Bayes (see\citealp{alquier2024user} for a textbook treatment),  and generalised Bayesian methods \citep{bissiri:etal:2016, knoblauch2019generalized}. %, and we refer to Appendix \ref{app:Gibbsexamples} for a discussion of how these Bayesian methods fit within the Gibbs measure framework. 

While some Gibbs measures are straightforward to compute, others expose  computational hurdles whose impact has been mostly unaddressed in existing contributions  \citep[see e.g.][]{hooker2014bayesian, ghosh2016robust,cherief2020mmd, miller2021asymptotic, pacchiardi2024generalized}.
More particularly, many methods in the literature rely on losses $\MD_n(\theta)$ that are \textit{intractable}: $\MD_n(\theta)$ is defined as an integral over the class of models $P_\theta$, which has no closed-form solution. In such cases, $\MD_n(\theta)$ must be estimated through a proxy $\MD_{m,n}(\theta)$ that relies on $m$ samples of random variables $u_{1:m}$ with $u_j \sim P_{\theta}$ for $j=1,2,\dots m$. 
Though a proxy of this form is an explicit function of the simulated data $u_{1:m}$ so that $\MD_{m,n}(\theta) = \MD_{m,n}(\theta, u_{1:m})$, we usually suppress this dependence in the remainder for notational convenience.

\subsection{Motivating Examples}
\label{subsec:applicability-of-results}
Gibbs measures based on intractable likelihoods are popular tools for producing robust Bayesian inferences.
Two of the most popular instantiations are losses based on kernel scoring rules \citep{gneiting2007strictly} and $\beta$-divergences \citep{basu1998robust}, both of which generally necessitate estimating the loss using simulation from  $P_{\theta}$.

\subsubsection{Gibbs measures based on $\beta$-divergences}
\label{sec:beta-div-definition}

% \begin{example}[$\beta$-divergence]
%     \normalfont
Suggested for parameter estimation by \citet{basu1998robust},  first used to define a Gibbs posterior in \cite{ghosh2016robust} and subsequently extended in several works \citep{knoblauch2018doubly, futami2018variational,jewson2018principles, knoblauch2019generalized, jewson2023differentially, jewson2024stability}, the $\beta$-divergence is also often called the \textit{density power divergence}, and for some $\beta>0$ is given by
\begin{IEEEeqnarray}{rCl}
\dt_\beta(p_0,p_{\theta})=\int\left\{p_{\theta}^{1+\beta}(y)-\left(1+\frac{1}{\beta}\right) p_0(y) p_{\theta}(y)^\beta+\frac{1}{\beta} p_0^{1+\beta}(y)\right\} \dt y.
\label{eq:DPD}
\end{IEEEeqnarray}
Assuming that the true data-generating mechanism $P_0$ admits a density $p_0$ and that $y_i \overset{\text{iid}}{\sim} p_0$, if we wish to minimise $\dt_\beta(p_0,p_{\theta})$ over $\theta$, we can ignore the last term and conduct inference on $\theta$ using the loss
\begin{IEEEeqnarray}{rCl}
\MD^{\beta}_{n}(\theta)= \int p_{\theta}^{1+\beta}(u)\dt u
-\left(1+\frac{1}{\beta}\right)\frac{1}{n} \sum_{i=1}^n p_{\theta}(y_i)^\beta.
\nonumber
\end{IEEEeqnarray}
Minimising $\MD_n^{\beta}(
\theta)$ yields an estimator for $\theta$ that converges to the true data-generating parameter as $n\to\infty$ if the model is well-specified, and which exhibits robustness  
under model misspecification \citep{basu1998robust, jewson2024stability}.
Here, $\beta$ determines the degree of robustness: $\MD_n^{\beta}$ is more robust as $\beta \to \infty$, but recovers the non-robust negative log likelihood $-\log p_{\theta}(y_{1:n})$ as $\beta\to 0$.
%
% As $\beta \to 0$, its limit eventually recovers the non-robust negative log likelihood $-\log p_{\theta}(y_{1:n})$.
%
This means that Gibbs measures based on $\MD_n^{\beta}$ can continuously interpolate between a family of robust posteriors indexed by 
$\beta$ and the orthodox Bayes update in
\eqref{eq:bayes-posterior}, making them 
ideally suited for robust Bayes-like inferences.
%as it  recovers the standard Bayes posterior for $\beta \to 0$.
%
However, outside a small subset of exponential families, the integral $\int p_{\theta}^{1+\beta}(u)\dt u$ has no analytically available form.
As a result, to construct Gibbs measures of this form in practice, one generally has to draw samples $u_j \overset{\text{iid}}{\sim} p_{\theta}$ for $j=1,2,\dots m$, and then use the approximation
\begin{IEEEeqnarray}{rCl}
{\MD}^{\beta}_{m,n}(\theta)=  B_m(\theta)-A_n(\theta),\; B_m(\theta):=
\frac{1}{m}\sum_{j=1}^mp_{\theta}(u_j)^{\beta},
\; A_n(\theta):=\left(1+\frac{1}{\beta}\right) \frac{1}{n}\sum_{i=1}^n p_{\theta}(y_i)^\beta
\label{eq:loss-estimated-beta-div}
\end{IEEEeqnarray}
meaning that the target posterior depends on ${\MD}^{\beta}_{m,n}$ rather than $\MD^\beta_n$.

\subsubsection{Kernel Scoring rules}
\label{sec:mmd-introduction}

For a positive-definite and symmetric function $k:\mathcal{Y}\times\mathcal{Y} \to \mathbb{R}_{+}$, a kernel score is a function $S_k(P,y)=k(y,y)+\E_{Y,Y'\sim P}[k(Y,Y')]-2\E_{Y\sim P}[k(Y,y)]$, and can be used to measure the accuracy of the model $P$ at observation $y$. The use of kernel scores is commonplace in the literature on Gibbs posteriors for two  reasons: they provide strong robustness guarantees to model misspecification  \citep[see e.g.][]{briol2019statistical, cherief2020mmd}, and are directly applicable to simulation-based inference, i.e, intractable likelihood models, as their evaluation does \textit{not} require an analytically tractable likelihood function. The usefulness of such a class for intractable likelihood models is noted \cite{park2016k2}, \cite{bernton2019approximate},   \cite{pacchiardi2024generalized}, and \cite{legramanti2022concentration}.

A particularly important sub-class of kernel scores are those whose expectation with respect to $y \sim Q$ has the interpretation of a squared maximum mean discrepancy (MMD)---a distance measure between distributions $Q$ and $P$ given by 
%
%Under mild conditions, the squared MMD between probability measures $Q$ and $P$ is %the expected score under $Q$ so that 
%$\MMD^2(Q,P)=\E_{X,X'\sim Q}[S_k(P,X)]$, which gives
$$
\MMD^2(Q,P)=\E_{X,X'\sim Q}[k(X,X')]+\E_{Y,Y'\sim P}[k(Y,Y')]-2\E_{Y\sim P,X\sim Q}[k(Y,X)].
$$
The MMD has been a key idea in modern statistical analyses for more than a decade, including for testing theory \citep{gretton2012kernel}, generative models \citep{dziugaite2015training, sutherland2018demystifying}, simulation-based inference \citep{briol2019statistical, dellaporta2022robust}, as well as uncertainty quantification and parameter estimation \citep{cherief2020mmd, alquier2023estimation, alquier2024universal}.
The MMD is also related to the energy distance, 
which takes $k(x,x')=\|x-x'\|^\alpha$ and some $\alpha\in(0,2)$.

Kernel scoring rules are useful in parameter inference since they measure the difference between observed data and the assumed model: taking  $Q$ as the empirical measure on the data, $P=P_{\theta}$, and dropping terms that do not depend on $\theta$, delivers the loss
\begin{IEEEeqnarray}{rCl}
\MD^k_{n}(\theta)
& = & 
-
\frac{2}{n}\sum_{i=1}^n\mathbb{E}_{U\sim  P_{\theta}}\left[ k(U, y_i) \right]
+
\mathbb{E}_{U \sim P_{\theta}, U' \sim P_{\theta}}\left[ 
k(U, U')
\right],
\nonumber
\end{IEEEeqnarray}
which contains two intractable expectations that are estimated via   simulations from $P_{\theta}$.
More precisely, for  $u_j \overset{\text{iid}}{\sim} P_{\theta}$ with $j=1,\dots,m$, $\MD^k_{n}(\theta)$ can be estimated as
\begin{IEEEeqnarray}{rCl}
\MD^k_{m,n}(\theta)
& = &
-
2\frac{1}{nm}\sum_{i=1}^n\sum_{j=1}^m  k(u_j, y_i) 
+
\frac{1}{m^2}\sum_{j=1}^m\sum_{j'=1}^m
k(u_j, u_{j'}).
\label{eq:mmd-loss-mn}
\end{IEEEeqnarray}

Gibbs posteriors  based on $\MD^k_{m,n}(\theta)$ are popular in the literature for two reasons. 
The first is robustness of the MMD to Huber-type model misspecification \citep[see for instance][]{cherief2020mmd, cherief2022finite, alquier2023estimation, alquier2024universal}. 
The second reason is that $\MD^k_{m,n}(\theta)$ can be computed without access to a likelihood function, which has been used by several authors to deliver posterior inferences in the presence of intractable likelihoods \citep[see e.g.][]{park2016k2, briol2019probabilistic,bernton2019approximate,
pacchiardi2024generalized,
legramanti2022concentration}.

\subsection{Research Gap}{\label{Sec:research_gap}}
Sampling from Gibbs posteriors is generally implemented via MCMC. Within such algorithms, replacing the intractable loss $\MD_n(\theta)$ with its tractable counterpart $\MD_{m,n}(\theta)$ 
results in an algorithm whose draws are from a biased estimator for the de-facto target $\pi(\theta\mid\MD_n)$, which we refer to as the \textit{implicit} posterior target
\begin{IEEEeqnarray}{rCl}
\overline\pi(\theta \mid \MD_{m,n})  \; &  \propto & \; \pi(\theta) \cdot 
\E_{u_{1:m} \overset{\text{iid}}{\sim} P_{\theta}}\left[
\exp \{-\omega \cdot \MD_{m,n}(\theta)\}\right]. 
\nonumber
\end{IEEEeqnarray}
The quantity $\overline\pi(\theta \mid \MD_{m,n})$ distinguishes itself from the target of interest $\pi(\theta\mid\MD_n)$ by explicitly depending on the choice of $m$ via $\E_{u_{1:m} \overset{\text{iid}}{\sim} P_{\theta}}\left[ \exp \{-\omega \cdot \MD_{m,n}(\theta)\} \right]$.
This expectation reflects the fact that using $\MD_{m,n}(\theta)$ within sampling methods like MCMC requires drawing a \textit{new} sample $u_{1:m} \overset{\text{iid}}{\sim} P_{\theta}$ at each proposed value of $\theta$ within the algorithm---rather than re-using the same \textit{fixed} sample across different values of $\theta$.
Clearly, this expectation and its dependence on $m$ is crucial in defining the implicit target.

As is well-understood in the literature on P-MCMC and noisy Monte Carlo (\citealp{alquier2016noisy}), 
 $\overline\pi(\theta \mid \MD_{m,n})$ resulting from this  procedure is biased  for $\pi(\theta \mid \MD_n)$ even if the loss estimator itself is unbiased so that $\E_{u_{1:m} \overset{\text{iid}}{\sim} P_{\theta}}\left[  \MD_{m,n}(\theta) \right] = \MD_n(\theta)$.
In particular, when the tractable loss is of the form $\MD_{m,n}(\theta) = \MD_{m,n}(\theta, u_{1:m})$, where $u_{1:m}\stackrel{iid}{\sim}P_\theta$, Jensen's inequality implies: $\E_{u_{1:m} \overset{\text{iid}}{\sim} P_{\theta}}\left[
\exp\{-\omega\cdot \MD_{m,n}(\theta)\}
\right]
\ne 
%& \ge &
\exp\{-\omega \cdot \E_{u_{1:m} \overset{\text{iid}}{\sim} P_{\theta}}[ \MD_{m,n}(\theta) 
]\}.$ 
%
%\begin{IEEEeqnarray}{rCl}
%\E_{u_{1:m} \overset{\text{iid}}{\sim} P_{\theta}}\left[
%\exp\{-\omega \MD_{m,n}(\theta)\}
%\right]
%& \neq &
%%& \ge &
%\exp\left\{-\omega \cdot \E_{u_{1:m} \overset{\text{iid}}{\sim} %P_{\theta}}\left[ \MD_{m,n}(\theta) 
%\right]\right\}.
%\nonumber
%\end{IEEEeqnarray}
%
This inequality compels us to study two inter-linked problems for Gibbs measures based on estimated losses.

\begin{itemize}
\myitem{\textbf{C.1}}
We first study how the difference between $\overline\pi(\theta\mid \MD_{m,n})$ and $\pi(\theta\mid \MD_n)$ decays as a function of $m$ for any \textit{fixed} dataset.
Next, we show that as the data set size $n$ increases, \textit{posterior concentration} is impacted by the choice of $m$ even under favorable conditions:  for $\overline\pi(\theta\mid \MD_{m,n})$ to be a reliable approximation of $\pi(\theta\mid \MD_n)$, the number of pseudo-observations $m$ generated within each MCMC iteration must generally be of the same order as $n$. Such a requirement renders accurate approximation of $\pi(\theta\mid \MD_n)$ via $\overline\pi(\theta\mid \MD_{m,n})$ infeasible in many situations.
\label{item:A1}

\myitem{\textbf{C.2}}
    Given the drawbacks associated to using estimated losses within MCMC, we design an alternative algorithm that is capable of sampling from 
    $\pi(\theta\mid \MD_n)$ directly.
    This algorithm is a novel extension of existing PDMP samplers that can be shown to deliver samples from $\pi(\theta\mid\MD_n)$ so long as gradients of $\MD_{n}$ in $\theta$, can be estimated without bias.
    Besides solving this outstanding challenge in the computation of Gibbs measures with estimated losses, this new approach is also attractive for inference in intractable likelihood models. So long as we can generate samples from the  model and unbiasedly estimate the gradient of the loss, \textit{we can produce exact {generalised posterior inferences}} even for models with intractable likelihoods. 
    \label{item:A2}
\end{itemize}

Our findings reveal that from a theoretical and practical standpoint, constructing Gibbs measures based on losses whose gradient can be estimated without bias is clearly preferable. When the gradient of $\MD_n(\theta)$ can be unbiasedly estimated, \textit{we can produce exact {generalised posterior inferences}} even when the likelihood is intractable, while avoiding the usual fine-tuning of tolerance hyper-parameters required to implement ABC methods. %Of course, such an approaches cannot be implemented in all such situations, however, when feasible such an approach allows one to bypassing the computational demands often associated with methods such as ABC.
%
%Indeed, a key practical ramification of our results is that even with intractable  likelihoods, we can still draw exact samples from $\pi(\theta\mid \MD_n)$.
%
%In other words, so long as the gradient of $\MD_n(\theta)$ can be unbiasedly estimated, \textit{we can produce exact {generalised posterior inferences} for likelihood-free inference problems}, and bypass the computational demands often associated with methods such as ABC.
%
Further, the algorithmic approach we introduce is equally well-suited to  robust losses such as $\beta$-divergences, which are generally intractable for even very simple models such as  Poisson regression models or Gaussian mixture models.
Irrespective of the posited model, our proposed algorithm can sample exactly from the associated Gibbs posteriors under mild regularity conditions.
%
%Conversely, the results presented in this paper apply even to Gibbs posteriors defined for parameters of interest such as quantiles or treatment effects that do not correspond to any likelihood model. 
%

%in situations where the model or loss is intractable, all else equal, it is highly-preferential -- from both a theoretical and practical perspective -- to choose as a loss function one that can be unbiasedly estimated. 
%
%If this is the case, then \ref{item:A1} and \ref{item:A2} imply that even for intractable models we can target the \textit{intractable posterior} $\pi(\theta\mid \MD_n)$ directly, completely bypassing the need to conduct inference using approximate Bayesian methods like ABC.

Results related to \ref{item:A1} have  been  obtained for certain types of posteriors based on estimated or approximated likelihoods, but we are unaware of similar results to ours for Gibbs posteriors.
Furthermore, we note  that since estimating likelihoods is in generally more difficult than estimating the types of integrals that appear in loss functions, the associated analyses are quite distinct from our own. 

The case in which we have access to an unbiased likelihood estimator is especially well-understood following the foundational work of \cite{andrieu:roberts:2009}. Limited results in certain directions have also been derived when biased likelihood estimators are used within Metropolis-Hastings, which is referred to as noisy Monte Carlo by \cite{alquier2016noisy}. As discussed in \cite{alquier2016noisy}, such algorithms do not draw samples from the correct target distribution but an approximation that depends on the number of pseudo-observations---the number $m$ in our notation---used in implementing the algorithm. For specific examples and estimated likelihoods only, \cite{alquier2016noisy} demonstrate that noisy Monte Carlo can converge to the appropriate target as $m$ diverges, however, general results are not provided. Since the posterior $\overline\pi(\theta\mid\MD_{m,n})$ meets the definition of a noisy Monte Carlo posterior, our results extend the analysis in \cite{alquier2016noisy} to general Gibbs measures, and we provide a general characterization for the rate, in $m$, at which $\overline\pi(\theta\mid\MD_{m,n})$ converges to $\pi(\theta\mid\MD_{n})$ when the sample size $n$ is fixed. %

Departing from classical likelihood-based Metropolis-Hastings MCMC algorithms, 
\cite{kaji2023metropolis} replace an intractable likelihood ratio with an estimated likelihood ratio -- obtained by fitting a classifier that contrasts the $n$ observed data points against $m$ simulated draws from the assumed model -- to produce samples from an approximation to the standard Bayesian posterior; related  classification-based approaches can also be found in \cite{wang2022approximate} and \cite{wang2026generative}. 
\cite{kaji2023metropolis} impose the condition that $m$ must increase in proportion to $n$ to ensure that the classifier is trained on a sufficiently large set of fake data to approximate the likelihood ratio reliably. 
Under regularity conditions that include convergence of $n/m$, \cite{kaji2023metropolis} demonstrate that the resulting approximation of the Bayes posterior concentrates at a rate that matches the standard rate of posterior concentration,
up to an inflation factor that depends on the classifier's convergence rate and which will generally slow convergence unless it is of order $O(1)$. 

While our contributions include  posterior concentration results,  we derive them in a methodologically and theoretically distinct regime from what is done in \cite{kaji2023metropolis}.
Our target of interest is a Gibbs measure $\pi(\theta \mid \MD_n)$ based on a general intractable loss $\MD_n(\theta)$---rather than an approximation of a Bayes posterior in the case where the likelihood is intractable.
Since the intractable losses, $\MD_n$, used in generalized Bayesian methods generally take the form of an integral, we focus exclusively on the case where  $\MD_n$ is estimated via Monte Carlo averages as $\MD_{m,n}$.
In contrast,  the work of \cite{kaji2023metropolis} considers intractable likelihoods---which are generally not approximable through Monte Carlo averages---and investigates the effect of 
approximating them through an estimated likelihood ratio approach based on classifiers.
%
%As a result, the impact of the number of pseudo-observations $m$ in our analysis is different to that considered in \cite{kaji2023metropolis}. 
%
The fundamental difference in the type of intractability considered between the two approaches, and the differences in the approximations themselves, necessitate distinctly different mathematical analysis for the resulting posteriors.
To summarize, the methodology, assumptions and theoretical analysis of \cite{kaji2023metropolis}, while insightful and novel, is not directly comparable or applicable in our setting: not only are the posteriors we target different, Bayes versus Gibbs, the methods used to target these posteriors, estimating likelihood ratios versus Monte Carlo estimation of an integral loss function, are also different.

In Section \ref{sec:theory-new} below, we summarise our theoretical results and show that the posterior concentration rate of $\overline\pi(\theta\mid\MD_{m,n})$ is explicitly determined by the minimum of (a polynomial function of) $m$ and $n$. Consequently, in order for $\overline\pi(\theta\mid\MD_{m,n})$  to achieve the standard posterior concentration rate we will often require $m$ to be in proportion to $n$. 
Motivated by this fact, Section \ref{sec:PDMP} provides a novel MCMC algorithm that can sample from $\pi(\theta \mid \MD_n)$ without bias and which does not depend on the number of pseudo-samples used in its implementation.

\section{Noisy MCMC with estimated losses}
\label{sec:theory-new}

In this section, we demonstrate that approximating $\pi(\theta\mid\MD_n)$ using noisy P-MCMC-like algorithms that target $\overline\pi(\theta\mid\MD_{m,n})$ can be practically infeasible.
To derive a negative result like this, it is not necessary to work with a set of regularity conditions that are as broadly applicable as possible.
Rather, it suffices to identify reasonable conditions that should make $\overline\pi(\theta\mid\MD_{m,n})$ a more attractive approximation and to demonstrate that, even under such favorable conditions, its use for approximating $\pi(\theta\mid\MD_n)$ poses serious drawbacks.
\subsection{Fixed $n$ and Diverging $m$ Results}
First, for positive loss functions, such as the class of kernel scores, we show that it is possible to directly quantify the gap between the feasible target $\overline\pi(\theta\mid\MD_{m,n})$ and the infeasible posterior of interest $\pi(\theta\mid\MD_n)$ \textit{for any fixed dataset $y_{1:n}$}. We impose the following mild regularity conditions.

\begin{assumption}\label{ass:positive}
For fixed $n$, the following are satisfied. 
\newline\noindent(i) $\MD_{m,n}(\theta)\ge0$ for all $\theta\in\Theta$. 
\newline\noindent(ii) There exists $c_1:\Theta\rightarrow\mathbb{R}_+$ such that $\E_{u_{1:m}}[ \MD_{m,n}(\theta)]=\MD_n(\theta)+c_1(\theta)/m^\kappa$, where $\kappa>0$. \newline\noindent(iii) There exists $c_2:\Theta\rightarrow\mathbb{R}_+$ such that $\Var_{\U}\{\MD_{m,n}(\theta)\}=c_2(\theta)/m$. 
\newline\noindent(iv) $\int c_j^2(\theta)\pi(\theta\mid \MD_n)\dt\theta\le C<\infty$ and $\int c_j(\theta)\pi(\theta)\dt\theta\le C<\infty$, for $j=1,2$.
\end{assumption}

Assumption \ref{ass:positive} requires a weak set of moment conditions on the loss $\MD_{m,n}(\theta)$, and some minor integrability conditions on certain functions of $\theta$. This condition can be shown to be satisfied for several classes of loss functions used in practice, including distance functions between observed and simulated samples, MMD and energy scores, Bayesian synthetic likelihood (\citealp{wood:2010}; \citealp{frazier2022bayesian}) and certain positive surrogate loss functions. 

\begin{restatable}{lemma}{lemtvbound}\label{lem:tvbound}
If Assumption \ref{ass:positive} is satisfied, then
$
\int |\overline\pi(\theta\mid \MD_{m,n})-\pi(\theta\mid \MD_{n})|\dt\theta = O(1/\min\{m^\kappa,m\})
$  for $n$ fixed and $m$ large. 
\end{restatable}

Lemma \ref{lem:tvbound} clarifies that, for a fixed dataset, $\overline\pi(\theta\mid\MD_{m,n})$ approximates $\pi(\theta\mid\MD_n)$ at the usual Monte Carlo rate, and thus requires taking $m$ large in general. %Consequently, for many loss functions used in practice, $\overline\pi(\theta\mid\MD_{m,n})$ will only accurately approximate $\pi(\theta\mid\MD_n)$ if we let the number of simulated samples $m$ used in each MCMC iteration diverge. 
We remind the reader that while more general sampling algorithms like the correlated P-MCMC of \cite{deligiannidis2018correlated} or the block P-MCMC of \cite{tran2016block}, are more computationally efficient, they still sample from the $m$-dependent target distribution $\overline\pi(\theta\mid \MD_{m,n})$, and thus suffer from the same dependence on  $m$. 

Lemma \ref{lem:tvbound} does not directly apply to certain losses used in the literature, most notably the $\beta$-divergence. However, a corollary to Lemma \ref{lem:tvbound} that covers the $\beta$-divergence, and similarly defined robust losses, can be readily obtained. The key to obtain such a result is that losses like the $\beta$-divergence can be written as the summation of two components: one that depends on data, the other that depends purely on simulations from the model. %, and where this latter component can be unbiasedly estimated.% i.e., $\MD_{m,n}(\theta)=A_n(\theta)+B_m(\theta)$, $B_m(\theta)\ge0$, and $\E[B_m(\theta)]=B(\theta)$. 

\begin{restatable}{corollary}{cortvbound}\label{cor:beta_div}
If a version of Assumption \ref{ass:positive} is satisfied for the $\beta$-divergence, see \Cref{ass:beta_positive} in \Cref{sec:main_proofs} for specific details,   then $\int |\overline\pi(\theta\mid\MD_{m,n})-\pi(\theta\mid\MD_n)|\dt\theta=O(1/m)$ for $n$ fixed and $m$ large.
\end{restatable}

\subsection{Both $m$ and $n$  Diverging}

\Cref{lem:tvbound} and \Cref{cor:beta_div} characterize the behavior of $\overline\pi(\theta\mid\MD_{m,n})$ as a function of $m$, with $n$ fixed, but do not explain how this posterior behaves when $m$ and $n$ both diverge. In the following section, we show that \textit{even under somewhat optimistic assumptions on the loss function}, we must generally use $m\gg n$ samples at each MCMC iteration to ensure that $\overline\pi(\theta\mid\MD_{m,n})$ achieves the standard rate of posterior concentration.
\subsubsection{Maintained assumptions}

Throughout, for given sequences of random variables $\{X_n\}_{n\in\mathbb{N}}$ and scalars $\{a_n\}_{n\in\mathbb{N}}$, we will take $X_n \lesssim a_n$ to mean that the sequence $\{X_n/a_n\}_{n\in\mathbb{N}}$ is  stochastically bounded by a constant $C >0$ almost surely, so that $\mathbb{P}\left( \lim_{n\to\infty} X_n/a_n \leq C \right) = 1$. 
Similarly, we write $X_n\asymp a_n$ if $a_n\lesssim X_n$ and $X_n\lesssim a_n$.
Additionally, $\|x\|$ denotes the usual Euclidean norm for any vector $x \in \mathbb{R}^d$. The expectation with respect to the data-generating distribution $P_0$ is denoted as $\E_0$.

We first impose two mild regularity conditions: the existence of a well-defined point in $\Theta$ onto which the posterior should concentrate, and an assumption that the prior places sufficient mass around this point. Such conditions are standard in the study of posterior concentration in Gibbs measures; we refer to \cite{syring2020gibbs},  \cite{alquier2020concentration}, and \cite{miller2021asymptotic} for related assumptions and discussion.

\begin{assumption}\label{ass:pm_cond}
    There exists $\MD:\Theta\rightarrow\mathbb{R}$ such that $\MD(\theta):=\E_0\left[\MD_{n}(\theta)\right]$, and there exist  $\theta_0\in\Theta$ such that for any $\delta>0$, $\MD(\theta_0)\le \inf_{\theta:\|\theta-\theta_0\|>\delta}\MD(\theta)$. 
    Further, for 
    $\epsilon>0$ and $\mathcal{A}_{\epsilon} := \{\theta\in\Theta: |\MD(\theta)-\MD(\theta_0)|\le\epsilon\}$,  the prior satisfies $\int_{\Theta}\mathds{1}\{\theta\in \mathcal{A}_{\epsilon}\} \pi(\theta) \dt\theta \ge e^{-\lambda\epsilon}$ for some $\lambda>0$.
\end{assumption}

\noindent
To ensure that our estimated loss is of sufficient quality for posterior inferences,  we impose two  conditions to control its moments.
% %

\begin{assumption}\label{ass:mean_bias}
One of the following holds: (i) the loss estimator satisfies $\E_{\U\sim P_\theta}\MD_{m,n}(\theta)=\MD_n(\theta)$ for all $\theta\in\Theta$; (ii) there exists a continuous function $\gamma:\Theta\rightarrow\mathbb{R}_+$ and a $\kappa>0$ such that, 
$
\E_0\E_{\U\sim P_\theta}\left[\MD_{m,n}(\theta)\right]\lesssim\MD(\theta)+\gamma(\theta)/m^\kappa
$	for all $\theta\in\Theta$.
\end{assumption}

\begin{assumption}\label{ass:tails-new}There exists $\eta:\Theta\rightarrow \mathbb{R}_+$ and $g: 
\mathbb{N} \times \mathbb{N} \times \mathcal{J} \to \mathbb{R}_{+}$ with $\mathcal{J} \subset \mathbb{R}_{+}$ such that 
for any $\lambda\in \mathcal{J}$,
%\begin{IEEEeqnarray}{rCl}
 $   \int_\Theta\E_0\E_{\U \sim P_{\theta}} \left[e^{-\lambda\left[\MD_{m,n}(\theta)-\MD_{}(\theta)\right]}\right]\pi(\theta)\dt\theta\le\int_\Theta e^{g(m,n,\lambda)\eta(\theta)} \pi(\theta)\dt\theta,$	
 %   \nonumber
%\end{IEEEeqnarray}
%
and the prior $\pi(\theta)$ is such that $\int e^{\eta(\theta)}\pi(\theta)\dt\theta< \infty$, while $g(m,n,\lambda)\lesssim\lambda^2/\min\{m,m^\kappa,n\}$.

\end{assumption}

\Cref{ass:mean_bias} allows  $\MD_{m,n}(\theta)$ to be a  biased estimator  for the infeasible loss $\MD_n(\theta)$ as long as the bias decays (pointwise) to zero at rate $1/m^\kappa$, which is satisfied for both losses introduced in our running examples (\eqref{eq:loss-estimated-beta-div} and \eqref{eq:mmd-loss-mn} of \Cref{subsec:applicability-of-results}).
This condition is the $n$-dependent version of the bias condition maintained in \Cref{ass:positive} (ii), which only treated the fixed $n$ setting.
%
%Assumption \ref{ass:tails-new} is a modified version of conditions used in PAC-Bayes and ensures that the tails of the posterior $\overline{\pi}(\theta\mid\MD_n)$ decay sufficiently fast so that the posterior concentrates onto the loss minimizer.
%
%Though establishing it can sometimes be more challenging to verify, \Cref{ass:tails-new} is related to two well-understood conditions.
%
\Cref{ass:tails-new} generalises the  tail condition typically used in Probably Approximately Correct (PAC) Bayesian analysis to estimated losses $\MD_{m,n}$
(see, e.g., Definitions 2.3 and 2.4 in \citealp{alquier2016properties}),
and holds whenever the loss can be written as the sum of sub-exponential random variables.
Conditions similar to \Cref{ass:tails-new}  are generally sufficient to establish concentration of Gibbs measures (see, e.g., \citealp{syring2020gibbs}), and so it is not surprising that similar conditions must also be maintained when the loss is estimated. \Cref{simple_lemma} in \Cref{app:lemmas} gives a set of more interpretable sufficient conditions that we can use to verify \Cref{ass:tails-new}. In the following section, we show that \Cref{ass:tails-new} is satisfied in our running examples under primitive conditions.
%
%This is perhaps not surprising: Gibbs measures are themselves defined through an exponentiated negative loss which is sometimes called a `pseudo-likelihood'.
%
%Put differently, our assumption requires that large values of the resulting loss, under the chosen prior, occur with sufficiently small probability so that the corresponding Gibbs posterior concentrates in a standard manner.

%Since we operate with estimated losses, it should not be surprising that we require a similar condition on the estimated loss.

\subsubsection{Applicability of Assumptions}
The conditions we impose are widely applicable. To demonstrate this, we verify that they hold for the intractable losses and estimators corresponding to the $\beta$-divergences and the MMD  defined in \eqref{eq:loss-estimated-beta-div} and \eqref{eq:mmd-loss-mn}.
%
%To prove this, we us the key \Cref{simple_lemma} in Appendix \ref{app:results}.
%as they are by far the most popular in the literature on generalised Bayesian methods. 
%

\begin{restatable}{lemma}{lemmabeta}\label{lem:beta_div}
For the proxy loss in  \eqref{eq:loss-estimated-beta-div}, Assumption \ref{ass:mean_bias}(i) is satisfied. 
If there exists  $M_{\theta}$ such that $\sup_{y\in\mathcal{Y}}p_\theta(y) \le M_\theta $ and $\int e^{M_\theta^{2\beta}}\pi(\theta)\dt\theta<\infty$, then Assumption \ref{ass:tails-new} is satisfied with $g(m,n,\lambda)\asymp \lambda^2/\min\{m,n\}$  for any $\beta \in (0,1)$.
\end{restatable}
From \Cref{lem:beta_div}, Assumptions \ref{ass:mean_bias} and \ref{ass:tails-new} are satisfied for the $\beta$-divergence when the data is discrete and for many density functions under reasonable prior choices.
For example, for
$p_\nu(y;\mu)$ denoting the density of a Student's $t$ distribution with mean $\mu$, scale $\sigma$, and degree $\nu$,  we have that $\sup_{y\in\mathcal{Y}}p_\nu(y;\mu,\sigma)\le M_{\nu,\sigma}:=\frac{\Gamma\left(\frac{\nu+1}{2}\right)}{\sqrt{\nu\sigma^2 \pi} \Gamma\left(\frac{\nu}{2}\right)}$.
So long as we choose priors on $(\mu, \nu, \sigma^2)$ that have sufficiently thin tails, the integrability conditions of the lemma are satisfied.
This condition is not restrictive: with factorized priors, this condition is satisfied, for instance, by placing a Gamma prior on $\nu$ and an inverse Gamma prior on $\sigma^2$.

% The condition on $p_\theta(y)$ in \Cref{lem:beta_div} is satisfied for any mass function, since $p_\theta(y)^\beta \le 1$ for all $\theta\in\Theta$, and for most density functions that do not have poles, including densities with heavy tails like the student-t (at least under standard priors). To see this, take $p_\nu(y;\mu)$ as the student-t density with mean $\mu$, scale $\sigma$ and degrees of freedom $\nu$. It can then be shown that $\sup_{y\in\mathcal{Y}}p_\nu(y;\mu,\sigma)\le M_{\nu,\sigma}:=\frac{\Gamma\left(\frac{\nu+1}{2}\right)}{\sqrt{\nu\sigma^2 \pi} \Gamma\left(\frac{\nu}{2}\right)}$; taking a standard prior on $\nu$ (resp., $\sigma^2$), such as the Gamma distribution (resp., inverse Gamma), ensures that $\int e^{M_{\nu,\sigma}}\pi(\nu,\sigma)\dt\nu\dt\sigma<\infty$.

For bounded kernel scores, such as the MMD, the situation is even simpler: provided that the kernel is bounded, Assumptions \ref{ass:mean_bias} and \ref{ass:tails-new} are automatically satisfied.

\begin{restatable}{lemma}{lemmakenel}\label{lem:kernel}
If $0\le k(u,u')\le\mathsf{K}<\infty$,  Assumption \ref{ass:mean_bias} is satisfied with $\kappa=1$, and Assumption \ref{ass:tails-new} is satisfied with $g(m,n,\lambda)\asymp\lambda^2/\min\{m,n\}$ for the proxy loss in \eqref{eq:mmd-loss-mn}.
\end{restatable}

%The assumption of a bounded kernel is a weak condition which is satisfied for a large subset of kernel scoring rules.
%
If we want a kernel score to also be interpretable as a MMD, this assumption is not restrictive: for the MMD to be a proper distance metric, the   mean embeddings corresponding to the associated kernel must be finite for any probability distribution.
In general, the only way to ensure this is to require these kernels to be bounded.
%
%Indeed, to the best of our knowledge, no unbounded kernel is known to possess this property.
%
Furthermore, while the energy distance is unbounded, it still satisfies Assumptions \ref{ass:mean_bias}-\ref{ass:tails-new} if the underlying data is sub-Gaussian (see Lemma \ref{lem:energy} in Appendix \ref{app:lemmas} for details).

Together, Lemmas \ref{lem:beta_div}  and \ref{lem:kernel} demonstrate that  our assumptions are satisfied under realistic conditions of interest, and for classes of loss functions that are popular in generalized Bayesian inference.
%
%The form taken by $g(m,n,\lambda)$ in these examples also teases a key result: for the implicit P-MCMC target $\overline{\pi}(\theta\mid\MD_{m,n})$ to achieve the same posterior contraction rate as $\pi(\theta\mid\MD_n)$, one needs to choose $m=m(n) \asymp n$ to avoid the proxy loss from injecting a bias into the posterior approximation that cannot be recovered even as the sample size increases. 
%
% As we shall see, all else equal, the presence of this bias directly translates to a slower rate of convergence in the posterior unless $m$ is chosen appropriately. For the case of the MMD, this requirement on $m$ is exactly the same ones as those previously derived for point estimation by \citep[see][]{briol2019statistical}.
% %
% To the best of our knowledge, this stands in contrast to the  $\beta$-divergence: while the proxy loss we study was used in previous work \citep[see e.g.][]{futami2018variational}, the effect of $m$ on said proxy loss---or the Gibbs posterior this proxy induces---has thus far not been studied. 

%

\subsubsection{Main Result}

We now analyze how $m$ has to grow in $n$ for posterior concentration of $\overline{\pi}(\theta\mid \MD_{m,n})$ to not be adversely impacted by estimation of $\MD_n$.
%and use this result as motivation to develop more effective sampling algorithms in this context. 
While we treat the number of simulated samples  $m = m(n)$ as a function of $n$ in this analysis,  we suppress this for notational simplicity. 
To state our main result, define the rate $r_{m,n} = \min\{m,m^\kappa, n\}$ as well as
$\overline\Pi(A\mid \MD_{m,n}) = \int_A \overline{\pi}(\theta\mid \MD_{m,n})\dt \theta$, for any $A \subseteq \Theta$.

\begin{restatable}{theorem}{thmnew}\label{thm:new}
Suppose Assumptions \ref{ass:pm_cond}-\ref{ass:tails-new} hold. 
%
%For $n$ large enough but finite,
%$$
%\E_0 \left[\int_\Theta\{\MD(\theta)-\MD(\theta_0)\}\overline{\pi}(\theta\mid \MD_{m,n})
%\dt\theta\right]\lesssim \frac{1}{n}+\frac{1}{m^\kappa}+\frac{\lambda}{r_{m,n}}+\frac{\log \tau}{\lambda}.
%$$
For $\lambda\asymp r_{m,n}^{1/2}$
%$ 
and $\epsilon_n r_{m,n}^{1/2}\rightarrow\infty$, and $M>0$, possibly large, we have 
$
\E_0 \overline{\Pi}\left[|\MD(\theta)-\MD(\theta_0)|> M\epsilon_n \mid\MD_{m,n}\right]\lesssim \frac{1}{M}.
$
In addition, if $\{\MD(\theta)-\MD(\theta_0)\}^\alpha\ge \|\theta-\theta_0\|$ for some $\alpha>0$, then 
$
\E_0 \overline{\Pi}\left[\|\theta-\theta_0\|> (M\epsilon_n)^{1/\alpha}\mid\MD_{m,n}\right]\lesssim \frac{1}{M}.
$
\end{restatable}

Theorem \ref{thm:new} directly implies that in order for $\overline\pi(\theta\mid\MD_{m,n})$ to concentrate at the standard rate, we must use  $m^{\kappa}\gtrsim n$ pseudo-observations within each MCMC iteration. That is, unless $m$ increases at least as fast as $n^{1/\kappa}$, concentration of $\overline\pi(\theta\mid\MD_{m,n})$ will occur more slowly than under the posterior $\pi(\theta\mid\MD_n)$, meaning that the implicit target $\overline{\pi}(\theta\mid \MD_{m,n})$ will not learn from the data with the same efficiency as the de-facto target ${\pi}(\theta\mid \MD_{n})$. 
Lastly, we note that the H\"older-type condition imposed in the second result of Theorem \ref{thm:new} is similar to regularity conditions employed in the study of ABC posteriors \cite[see e.g. Assumption 3.5 in][]{bernton2019approximate}. 

\subsubsection{Ramifications}

\Cref{lem:tvbound} and \Cref{thm:new} demonstrate the fundamental limitations of using the implicit target $\overline\pi(\theta\mid\MD_{m,n})$: reliable inferences are guaranteed only if $m$ increases at rate  $n^{1/\kappa}$ or faster. 
%
%For Gibbs posteriors based on kernel scores and $\beta$-divergences, \Cref{lem:tvbound} and \Cref{cor:beta_div} show that $m$ will generally need to be large in order for $\overline\pi(\theta\mid\MD_{m,n})$ to be an accurate approximation of $\pi(\theta\mid\MD_n)$. 
%
These results show that
we must usually draw $m \gg n$ samples from the model $P_{\theta}$ \textit{at each MCMC iteration} to ensure that $\overline\pi(\theta\mid\MD_{m,n})$ provides a reasonable approximation of $\pi(\theta\mid\MD_{n})$.
This imposes a substantive computational burden: using Landau notation, running such a sampling algorithm for $\mathcal{I}$ iterations would  lead to a computational overhead of $\mathcal{O}(\mathcal{I} \cdot n^2)$.
%
%On top of scaling quadratically, we also find that empirically, such P-MCMC methods can mix poorly. 
%
As our experiments in \Cref{sec:examples} also reveal that using P-MCMC-type methods can mix poorly for the settings we consider, $\mathcal{I}$ generally also has to be  very large to obtain  high-fidelity approximations of $\pi(\theta\mid\MD_n)$.
In conclusion, using estimated losses $\MD_{m,n}$ and MCMC-type algorithms to accurately approximate the Gibbs posterior $\pi(\theta\mid\MD_n)$ will generally be computationally prohibitive in practice.

To close, we remind the reader that 
while the assumptions we have imposed are appropriate for the problem class under study (see Lemmas  \ref{lem:beta_div} and \ref{lem:kernel}), they should also be expected to favour $\overline\pi(\theta\mid\MD_{m,n})$.
Consequently, it stands to reason  that for proxy losses that are less well-behaved, the accuracy with which $\overline\pi(\theta\mid\MD_{m,n})$ approximates $\pi(\theta\mid\MD_{n})$ can be even worse than what is implied by our results.

\section{Zig-zag sampling with estimated losses}
\label{sec:PDMP}

To provide a solution for the challenge outlined in the previous section, we  advocate for the use of a modified version of the zig-zag sampler proposed by \citet{bierkens2019zig} to 
\textit{directly} draw samples from $\pi(\theta\mid\MD_n)$.
The only requirements for this are that (i) a mild unbiasdness condition on the gradients of $\MD_{m,n}$ can be verified and (ii) a suitable computational upper bound on the switching rates can be found.
Unlike for P-MCMC algorithms, drawing  samples directly means that there is no distinction between implicit and de-facto target.
Thus, the number of draws from $P_{\theta}$  {neither} influences the concentration of the target, {nor} does it have to increase as a function of $n$ like $m=m(n)$ did.
To highlight these important differences, we will use $b$ instead of $m$ to denote the number of draws from $P_{\theta}$ within each iteration of our proposed zig-zag sampler. %
Given that we do not require $b$ to increase with $n$, our zig-zag algorithm's  computational complexity only scales linearly in $n$. In our applications, we have seen that a choice of $b\le5$ is sufficient for reliable performance, and so our proposed approach to sample the exact posterior will often be much faster to implement than using P-MCMC to sample from the approximate posterior target. 
While we defer this discussion to Section \ref{sec:subsampling}, the procedure can also be further sped up through data sub-sampling.

\subsection{The zig-zag sampler}

While Markov chain Monte Carlo (MCMC) algorithms follow a discrete time reversible stochastic process, piecewise deterministic Markov processes (PDMPs) \citep{fearnhead2018piecewise} power a family of sampling algorithms based on continuous time stochastic processes with non-reversible deterministic linear dynamics between random stopping times from an inhomogeneous Poisson process \citep[see][for a review]{corbella2022automatic}.
One particularly well-studied member of this class is the zig-zag sampler \citep[see][]{bierkens2019zig}, which we will adapt to the setting studied in this paper.

Provided that $\theta$ is a vector with $d$ components, the zig-zag sampler operates on the extended space $\Theta\times\mathcal{V}$, where $\mathcal{V}=\{-1,1\}^d$ denotes the space of \textit{velocity variables}.
Given $\nu =(\nu_1, \nu_2, \dots \nu_d)^{
\top}$, it is understood that $\nu_j$ indicates the velocity at which the  zig-zag process is  traversing the $j$-th dimension of $\Theta$.
The process is piecewise deterministic,  and velocities are constrained to only change for one component of $\nu$ at \textit{stopping times}  $\{\tau^k\}_{k=0}^{\infty}$.
Given these stopping times, the dynamics for evolving $\theta$ are chosen to be linear between consecutive stopping times, which  enables the representation of the full continuous-time process $\{(\theta^{(t)}, \nu^{(t)})\}_{t\geq 0}$  through its so-called \textit{skeleton}  $(\tau^k, \theta^{(\tau^k)}, \nu^{(\tau^k)})_{k=0}^{\infty}$.

To sample these skeleton points, one follows a  recursive procedure:
starting with the previous skeleton point $(\tau^k, \theta^{(\tau^k)}, \nu^{(\tau^k)})$ and
given the potential function $\Psi_n(\theta) = -\log(\pi(\theta)) + \omega\MD_{n}(\theta)$ for which $\pi(\theta\mid\MD_n) \propto \exp\{- \Psi_n(\theta)\}$, one constructs the next stopping time $\tau^{k+1}$ by first sampling $\tau_j$ according to $d$  non-homogenous one-dimensional Poisson processes with rate functions $\lambda_j(\theta^{(\tau^k)} + \nu^{(\tau^k)}\cdot t, \nu^{(\tau^k)})$ for $j=1, 2,\dots, d$ with
\begin{IEEEeqnarray}{rCl}
%    \lambda_j(t; \:\theta^{(\tau^k)}, \nu^{(\tau^k)}) & := & \max\left\{ v^{(\tau^k)}_j \cdot \frac{\partial}{\partial\theta_j}\Psi_n(\theta^{(\tau^k)}+t\cdot \nu^{(\tau^k)}),0\right\},
    \lambda_j(\theta, \nu) & := & \max\left\{ \nu_j \cdot \frac{\partial}{\partial\theta_j}\Psi_n(\theta),0\right\},
    \nonumber
\end{IEEEeqnarray}
and then obtains 
$\tau^{k+1} = \tau^{k} + \tau_{j^*}$, where $j^* = \argmin_{j}\{ \tau_j \}$.
From this, one can now construct the next skeleton point $(\tau^{k+1}, \theta^{(\tau^{k+1})}, \nu^{(\tau^{k+1})})$
by setting $\nu_{j^*}^{(\tau^{k+1})} = - \nu_{j^*}^{(\tau^{k})}$, $\nu_{j}^{(\tau^{k+1})} =  \nu_{j}^{(\tau^{k})}$ for $j \neq j^*$, and  taking $\theta^{(\tau^{k+1})} = \theta^{(\tau^{k})} + \tau_{j^*} \cdot \nu^{(\tau^k)}$.

While this process is elegant, the candidate stopping times $\tau_j$ can generally \textit{not} be sampled directly from the inhomogenous Poisson processes with rate functions $\lambda_j$.
Instead, one usually relies on \textit{Poisson thinning} \citep{lewis1979simulation} and \textit{Cinlar's algorithm} \citep{cinlar1975introduction}, which requires the existence of \textit{computational bounds} $\Lambda_j$ associated to $\Psi_n(\theta)$ so that 
%(i) $\lambda_j(t; \theta, \nu) \leq \Lambda_j(t; \theta, \nu)$ for all $t>0$ and $(\theta, \nu) \in \Theta \times \mathcal{V}$, and (ii) $\int_0^t \Lambda_j(s; \theta,\nu)\dt s $ is computationally tractable for Cinlar's algorithm. 
(i) $\lambda_j(\theta + \nu\cdot t, \nu) \leq \Lambda_j(\theta + \nu\cdot t, \nu)$ for all $t > 0$ and  $(\theta, \nu) \in \Theta \times \mathcal{V}$, and (ii) $\int_0^t \Lambda_j(\theta + \nu\cdot s,\nu)\dt s $ is computationally tractable for Cinlar's algorithm.
If both requirements are met, candidate stopping times $\tau_j$ can now be sampled from Poisson processes with rates $\Lambda_j$, and then thinned using Poisson thinning.
In particular, a  new stopping time $\tau^{k+1} = \tau^{k} + \tau_{j^*}$ for $j^* = \argmin_{j}\{ \tau_j \}$  results in a change of velocity only with probability 
$\lambda_{j^*}( \theta^{(\tau^k)} + \nu^{(\tau^k)}\cdot \tau_{j^*}, \nu^{(\tau^k)}) / \Lambda_{j^*}( \theta^{(\tau^k)} + \nu^{(\tau^k)}\cdot \tau_{j^*}, \nu^{(\tau^k)})$.

\subsection{The role of unbiased gradient estimation}

While Poisson thinning ameliorates the computational intractability of a naive implementation for the zig-zag sampler, the Poisson rate function $\lambda_j$ still has to be evaluated pointwise to compute acceptance probabilities.
This would seem to prevent us from using the strategy outlined above: 
while the entire thrust of our paper is the intractability of $\MD_n$,
the rates $\lambda_j$ depend on  the intractable partial derivatives $ \frac{\partial}{\partial\theta_j} \MD_n(\theta)$.

Fortunately, it turns out that we can still rely on the above approach whenever we can estimate the gradient $ \frac{\partial}{\partial\theta} \MD_n(\theta)$ without bias. 
More formally, the key condition is the existence of a vector function $\varphi_{b,n}(\theta)$ depending on a sample $u_{1:b} \sim P_{\theta}$ for which 
\begin{IEEEeqnarray}{rCl}
\E_{u_{1:b}\sim P_\theta}\left[\varphi_{b,n}(\theta)\right]=\frac{\partial }{\partial\theta}\MD_{n}(\theta).
\label{eq:needed}
\end{IEEEeqnarray}
%
% {\color{blue}Note that we use the notation $b$ here to denote the number of simulated variables used to construct \eqref{eq:needed}. This is deliberate and is done to make clear that this choice for implementing zig-zag sampling no longer depends on the sample size $n$.}
With this, we can now write $\Psi_{b,n}'(\theta) := - \frac{\partial}{\partial\theta}\log \pi(\theta) + \omega  \varphi_{b,n}(\theta) $ 
as well as the corresponding estimated rate function  %$\hat{\lambda}_{j}(t; \theta,\nu) := \max\left\{ \nu_j \cdot \left[\Psi_{b,n}'(\theta + \nu \cdot t)\right]_j, 0 \right\}$. 
$\hat{\lambda}_{j}(\theta,\nu) := \max\left\{ \nu_j \cdot \left[\Psi_{b,n}'(\theta)\right]_j, 0 \right\}$.
%\left[\omega \cdot \varphi_{m,n}(\theta + \nu \cdot t)_{j} - \frac{\partial}{\partial \theta_j}\log \pi(\theta + \nu \cdot t) \right], 0\right\}$.
%
Now, if we can once again find \textit{computational bounds} %$\hat{\Lambda}_{j}(t; \theta,\nu)$ 
$\hat{\Lambda}_{j}(\theta,\nu)$ 
that are \textit{not} dependent on $u_{1:b}$ so that
\begin{IEEEeqnarray}{rCl}
    %\hat{\lambda}_j(t; \theta, \nu) & \leq & \hat{\Lambda}_j(t; \theta, \nu) \text{ for all } t>0, \text{ all possible }u_{1:b}, \text{ and all } (\theta, \nu) \in \Theta \times \mathcal{V},
    \hat{\lambda}_j(\theta + \nu\cdot t, \nu) & \leq & \hat{\Lambda}_j(\theta + \nu\cdot t, \nu) \text{ for all } u_{1:b}, \text{ and all } t > 0, (\theta, \nu) \in \Theta \times \mathcal{V},\label{eq:computational-bound-estimated}
\end{IEEEeqnarray}
%(i) $\hat{\lambda}_j(t; \theta, \nu) \leq \hat{\Lambda}_j(t; \theta, \nu)$ for all $t>0$, $u_{1:m}$, and $(\theta, \nu) \in \Theta \times \mathcal{V}$,
and so that $\int_0^t \hat{\Lambda}_j(\theta + \nu\cdot s,\nu)\dt s $ is computationally tractable for Cinlar's algorithm, the zig-zag sampler  in Algorithm \ref{Alg:ZigZag_EstimatedLoss} samples directly from $\pi(\theta\mid\MD_n)$.
This result is stated below, proven in Section \ref{sec:ProofStationaryDist_ZigZag}, and inspired by the proof of Theorem 4.1 in \cite{bierkens2019zig}.
\begin{restatable}{theorem}{thmzigzag}
\label{thm:StationaryDist_ZigZag}
Algorithm \ref{Alg:ZigZag_EstimatedLoss} produces a  zig-zag process
    with invariant distribution  $\pi(\theta\mid\MD_n)$.
    %\frac{\pi(\theta) \exp(-w\cdot \MD_n(\theta))}{\int\exp(-\Psi_n(\theta))d\theta}$.
\end{restatable}
%\end{theorem}
%

Theorem \ref{thm:StationaryDist_ZigZag} implies that Algorithm \ref{Alg:ZigZag_EstimatedLoss} produces draws from the Gibbs measure of interest, $\pi(\theta\mid\MD_n)$, and so our posterior target will exhibit the same theoretical behavior as $\pi(\theta\mid\MD_n)$: under reasonable regularity conditions, \cite{miller2021asymptotic} proves that $\pi(\theta\mid \MD_n)$ concentrates onto an appropriately defined quantity, and exhibits standard asymptotic normality.

\begingroup
\singlespacing
\begin{algorithm}[H]
\caption{Zig-zag sampler for $\pi(\theta\mid\MD_n)$ using unbiasedness condition \eqref{eq:needed}}\label{Alg:ZigZag_EstimatedLoss}
\KwIn{Initial $(\theta^{(0)},\nu^{(0)})$, $\tau^0 = 0$,  $b \in \mathbb{N}$, $\varphi_{b, n}$ satisfying \eqref{eq:needed},   $\hat{\Lambda}_j$ satisfying \eqref{eq:computational-bound-estimated}} 
\vspace*{0.1cm}
\For{$k=0,1,\ldots$}{
    Sample $\tau_j \sim \mathbb{P}(\tau_j \geq t) = \exp\left\{-\int_0^t\hat{\Lambda}_j(\theta^{(\tau^{k})} + \nu^{(\tau^{k})}\cdot s, \nu^{(\tau^{k})}) ds\right\}$ for $j=1,2,\dots, d$ \\
    Set $j^* = \argmin_{j}\tau_j$; update $\tau^{k+1} = \tau^k + \tau_{j^*}$ and $\theta^{(\tau^{k+1})} = \theta^{(\tau^{k})} + \nu^{(\tau^{k})} \cdot  \tau_{j^*}$\\
    Sample $u_{1:b} \sim P_{\theta^{(\tau^{k+1})}}$; set $\hat{p}_{k+1} =  \dfrac{\hat{\lambda}_{j^*}(\theta^{(\tau^{k})} + \nu^{(\tau^{k})}\cdot \tau_{j^*}, \nu^{(\tau^{k})})}{ \hat{\Lambda}_{j^*}(\theta^{(\tau^{k})} + \nu^{(\tau^{k})}\cdot \tau_{j^*}, \nu^{(\tau^{k})})}$;
    draw 
    $Z_{k+1} \sim \operatorname{Ber}(\hat{p}_{k+1})$
    \\
    \eIf{$Z_{k+1} = 1$}{Set $\nu_{j^*}^{(\tau^{k+1})} = - \nu_{j^*}^{(\tau^{k})}$ and $\nu_{j}^{(\tau^{k+1})} =  \nu_{j}^{(\tau^{k})}$ for $j \neq j^*$}{Set $\nu^{(\tau^{k+1})} = \nu^{(\tau^{k})}$}
}
\Return $(\tau^k, \theta^{(\tau^k)}, \nu^{(\tau^k)})_{k=0}^{\infty}$
\end{algorithm}
\endgroup

\subsection{Relation to previous results}{\label{sec:previous}}

The closest that prior work has come to our proposed idea is in \citet{pacchiardi2024generalized}, which touches upon some of the underlying ideas, but provides no details,  formal derivations, or  results.
Further, while Theorem \ref{thm:StationaryDist_ZigZag} appears superficially related to the idea of using the zig-zag sampler after sub-sampling \citep[see][]{bierkens2019zig}, it is fundamentally different: unlike in the case of sub-sampling, the  unbiasedness requirement in \eqref{eq:needed} does \textit{not} relate to the data-generating distribution, but to additional randomness due to simulating from $P_{\theta}$.
In fact, it is easy to design zig-zag samplers that implement \textit{both} sub-sampling and unbiased estimation as per \eqref{eq:needed}.
To illustrate this, consider Gibbs measures based on additively decomposable loss functions $\MD_n(\theta) = \sum_{i=1}^n \ell^i_n(\theta)$ whose individual terms $\ell^i_n$ only depend on the datum $y_i$ and have to be approximated via proxy losses $\ell^i_{b,n}$ depending on simulated observations $u_{1:b}$ so that $\mathbb{E}_{u_{1:b} \sim P_{\theta}}[\ell^i_{b,n}(\theta)] = \ell^i_{n}(\theta)$.
For this setting, one could construct zig-zag samplers that combine sub-sampling over $\{\ell^i_n\}_{i=1}^n$ with a proxy loss approach using $\ell^i_{b,n}$ to approximate $\ell^i_n$. This is discussed further in Section \ref{sec:subsampling}

\subsection{Unbiased gradient estimators and computational upper bounds}{\label{sec:constructing_grad}}

While Theorem \ref{thm:StationaryDist_ZigZag} shows that a function $\varphi_{b,n}$ satisfying \eqref{eq:needed} is crucial, it does not tell us how to construct such estimators in practice.
Fortunately, this is especially simple when we have access to a proxy loss $\MD_{b,n}$ that is an unbiased estimator for $\MD_n$, so that we can write
\begin{IEEEeqnarray}{rCl}
    \MD_n(\theta)
    & = &
    \int \MD_{b,n}(\theta) p_{\theta}(u_{1:b}) \dt u_{1:b}.
    \label{eq:direct-sampling-representation}
\end{IEEEeqnarray}
differentiating both sides and applying the log-derivative trick and obtain 
\begin{IEEEeqnarray}{rCl}
    \varphi_{b,n}(\theta) & = & \frac{\partial}{\partial\theta} \MD_{b,n}(\theta)+\MD_{b,n}(\theta)\frac{\partial}{\partial\theta} \log p_\theta(u_{1:b})
    \label{eq:candidate-direct}
\end{IEEEeqnarray}
as a natural candidate.
For example, in the case of the $\beta$-divergence, this yields
 \begin{IEEEeqnarray}{rCl}
\varphi_{b, n}^{\beta}(\theta)& = & \frac{(\beta + 1)}{b}\sum_{i=1}^n\sum_{k=1}^b\frac{\partial}{\partial{\theta}} \log p_{\theta}(u_k)p_{\theta}(u_k)^{\beta} - (\beta + 1)\sum_{i=1}^n\frac{\partial}{\partial{\theta}} \log p_{\theta}(y_i)p_{\theta}(y_i)^{\beta},
\nonumber
\end{IEEEeqnarray}
which is what we use in our robust clustering experiments in \Cref{sec:examples}. \Cref{Sec:PoissonRegression} contains an additional $\beta$-divergence application to Poisson regression.

Under this representation, \Cref{corr:betaD_bounds} provides weak conditions for the existence of computational upper bounds $\hat{\Lambda}_j$.
\begin{corollary}\label{corr:betaD_bounds}
Assume that density or mass function $p_{\theta}(\cdot)$ satisfies    
%\begin{IEEEeqnarray}{rCl}
$
\left|\frac{\partial}{\partial{\theta_j}} \log p_\theta(U) p_\theta(U)^\beta\right|\leq R_j(\theta,\beta)$ 
%\end{IEEEeqnarray}
for $j = 1, \ldots, d$. Then, %$\hat{\Lambda}_j$ in \eqref{eq:computational-bound-estimated} exists, and 
Algorithm \ref{Alg:ZigZag_EstimatedLoss} delivers draws from the $\beta$D-Bayes posterior $\pi(\theta\mid\MD_{n})$, with $\MD_{n}(\theta)=\MD^{\beta}_{n}(\theta)$. 
\end{corollary}

While the function in \eqref{eq:candidate-direct} is often easy to compute, there are situations where we may not have access to the density $p_{\theta}$ and thus cannot evaluate the second term.
In this scenario, a second candidate function $\varphi_{m,n}$ can be derived whenever we have access to an easy-to-sample density $p_v$ and a generator $G_{\theta}$ so that $P_{\theta} \sim p_v \circ G_{\theta}^{-1}$.
More specifically, the \textit{direct} integration over $p_{\theta}$ in \eqref{eq:direct-sampling-representation} can then be replaced by the \textit{indirect} integration that amounts to first sampling $v_{1:b} \sim p_v$ and then setting $u_{1:b} = G_{\theta}(v_{1:b})$.
By differentiating both sides of this pushforward representation, we obtain the alternative candidate function
\begin{IEEEeqnarray}{rCl}
\varphi_{b,n}(\theta) 
%= \left[\frac{\partial}{\partial\theta}\left\{\frac{\partial}{\partial u_{1:b}} \MD_{b,n}(\theta)\right\}\right]_{u_{1:b} = G_\theta(v_{1:b})}\frac{\partial}{\partial{\theta}}G_\theta(v_{1:b}),
= \left[\frac{\partial}{\partial u_{1:b}} \MD_{b,n}(\theta)\right]_{u_{1:b} = G_\theta(v_{1:b})}\frac{\partial}{\partial{\theta}}G_\theta(v_{1:b}) + \frac{\partial}{\partial\theta}\MD_{b,n}(\theta),
\label{eq:candidate-indirect-sampling}
\end{IEEEeqnarray}
which can be used whenever $P_{\theta}$ can be written as a forward simulation model.
As one of the main advantages of using MMD-Bayes for parameter inference lies in its ability to deal with precisely this type of simulation model, our numerical demonstrations %on regression and copula models using MMD-Bayes
in \Cref{sec:examples} rely on this alternative representation. Under this representation, \Cref{corr:MMD_bounds} provides weak conditions for the existence of computational upper bounds $\hat{\Lambda}_j$.

\begin{corollary}\label{corr:MMD_bounds}
Assume the generator $G_{\theta}(\cdot) = (G_{\theta,1}(\cdot), \ldots, G_{\theta,p}(\cdot))$ satisfies  
%\begin{IEEEeqnarray}{rCl}
$    \left|\frac{\partial}{\partial \theta_j}G_{\theta,l}(v)\right| = O\left(\left|\frac{1}{G_{\theta,l}(v)}\right|\exp\left\{\frac{\partial}{\partial \theta_j}G_{\theta,l}(v)^2\right\}\right)$  as $G_{\theta,l}(v)\rightarrow\pm \infty$
%\end{IEEEeqnarray}
and $j = 1, \ldots, d$, $l = 1\,\ldots, p$. Then, 
Algorithm \ref{Alg:ZigZag_EstimatedLoss} delivers draws from $\pi(\theta\mid\MD_{n})$, with $\MD_{n}^k(\theta)=\MMD^2(P_\theta,\mathbb{P}_n)$ where $k$ is the radial basis (RBF) kernel. 
\end{corollary}

\section{Applications}\label{sec:examples}

To illustrate the practical implications of our theoretical findings and the usefulness of our zig-zag approach, we consider two sets of experiments over two different classes of examples: intractable likelihoods and intractable losses. In both cases, we compare our zig-zag sampler with $b \in \{2, 5, 20, 50\}$ against the block P-MCMC (bP-MCMC) proposed in \citet{tran2016block} with $m \in \{2, 5, 10,20, 50, 100, 1000\}$ (the specific values of  $b$ and $m$ we use for the different examples can be found in the relevant figures). 
Within the class of P-MCMC methods, bP-MCMC is the strongest possible competitor
we could identify.
However, we again point out that while bP-MCMC is a more efficient  sampling algorithm than vanilla P-MCMC methods (\citealp{deligiannidis2018correlated}),  the target distribution remains $\overline\pi(\theta\mid\MD_{m,n})$, which still exhibits explicit dependence on the number of simulations used at each iteration.

For all experiments, we additionally strengthen the bP-MCMC as a competitor of zig-zag sampling by setting the covariance matrix for the random walk proposal distribution of bP-MCMC equal to the posterior distribution's covariance matrix (as estimated based on the zig-zag sampler) and initialising it at the true parameter values.
For additional details on how we implement this sampler, see \Cref{app:bPMCMC}.

We compare the accuracy of the different methods relative to the true posterior $\pi(\theta\mid\MD_n)$. {Draws from $\pi(\theta\mid\MD_n)$ were obtained by running the zig-zag sampler for a very long time period using $b=500$.} 
As our theory predicts, the zig-zag sampler is insensitive to the number of simulations $b$. 
%
%While one may wish to tune $b$ to reduce the algorithm's runtime, it correctly targets $\pi(\theta\mid \MD_n)$ regardless of how small one chooses $b$. 
%
%These examples illustrate that 
 In contrast, the posterior target sampled from
bP-MCMC can strongly depend on the number of samples $m$ used at each sampling iteration. %---so that in practice, large values of $m$ may be required to obtain reasonable approximations of $\pi(\theta\mid \MD_n)$. 
While our theory clarifies at which rate $m = m(n)$ ought to grow to ensure that the approximation error is not unreasonable, such rates are of limited help in deciding how large $m$ should be in practice. This means that, 
%
%In particular, our experiments show that the scale of the data, the dimensionality of the parameter vector, and the shape of the posterior all have non-negligible influence on the approximation error for any finite $n$ and $m$.
%
to ensure good posterior approximations using bP-MCMC we must either re-run the algorithm with increasing choices of $m$, until the resulting posteriors stabilise, or choose $m$ sufficiently large and then run the algorithm once for this large choice.
The former is computationally burdensome, and the latter runs the risk of not choosing $m$ sufficiently large.
%
%In summary, unlike bP-MCMC whose approximation error is  severely affected by the choice of $m$, our proposed zig-zag sampler requires no tuning of  $b$ to ensure valid posterior approximations, and so is a safer bet to use when feasible.

\subsection{Intractable Likelihood Examples}
We first compare the accuracy of bP-MCMC to our proposed sampler across two benchmark likelihood-free inference examples. For these examples, we conduct inference using the MMD, which was proposed by \cite{park2016k2} and \cite{briol2019statistical} for use in intractable models. 
\subsubsection{g-and-k}
We observe data $y_1,\dots,y_n$ generated from the g-and-k model, which is most often specified according to the quantile function 
$$
\begin{aligned} Q(z_p ; \theta)&= a+\mu\left[1+0.8 \frac{1-\exp (-g \cdot z_p)}{1+\exp (-g\cdot z_p)}\right]\times\left(1+z_p^2\right)^k z_p,\end{aligned}$$where $a,\mu,g,k$ are unknown parameters, and $z_p$ denotes the $p$-th quantile of the standard normal distribution. Inference on the model parameter $\theta=(a, \mu, g, k)^\top$ is often carried out using ABC given that the likelihood is intractable. Following the numerical results in \cite{drovandi2011likelihood}, we compare bP-MCMC based on the MMD distance against our zig-zag-based approach. 

Data is generated from the g-and-k model using true values $a=\mu=g=1$ and $k=0.50$ and for a sample of size $n=1000$ observations. As is standard for this model, our priors are chosen to be uniform over $[0,5]$ for each parameter. We compare the accuracy of each method relative to the true posterior visually in Figure %\ref{Fig:gandk_PM_vs_ZZ_n1000}-
\ref{Fig:gandk_PM_vs_ZZ_sims}. 
\Cref{Fig:gandk_PM_vs_ZZ_sims} plots the posterior approximation error for the $k$-marginal between both approaches relative to the number of times the algorithms simulate from the g-and-k model, and across different choices of $b$ and $m$ for each algorithm.
The plot uses the difference between estimated and true posterior mean and posterior standard deviation as benchmarks; {the `true' mean and standard deviation are based on a zig-zag sampler with $b=500$.} 
The results illustrate the relative cost of reducing posterior approximation errors as a function of calls to the simulators: while the zig-zag's computational efficiency is virtually independent of $b$, the choice of $m$ interacts drastically with the performance of bP-MCMC.
Most obviously, small choices of $m$ lead to biased estimates, and this bias cannot be entirely overcome even by choosing $m$ large, which drastically increases the computational cost of the algorithm.

\begin{figure}[!htp]%[t!]
%\vskip -1cm %-0.1in
\begin{center}
\begin{subfigure}[t]{0.4\textwidth}
    \includegraphics[scale=0.35]{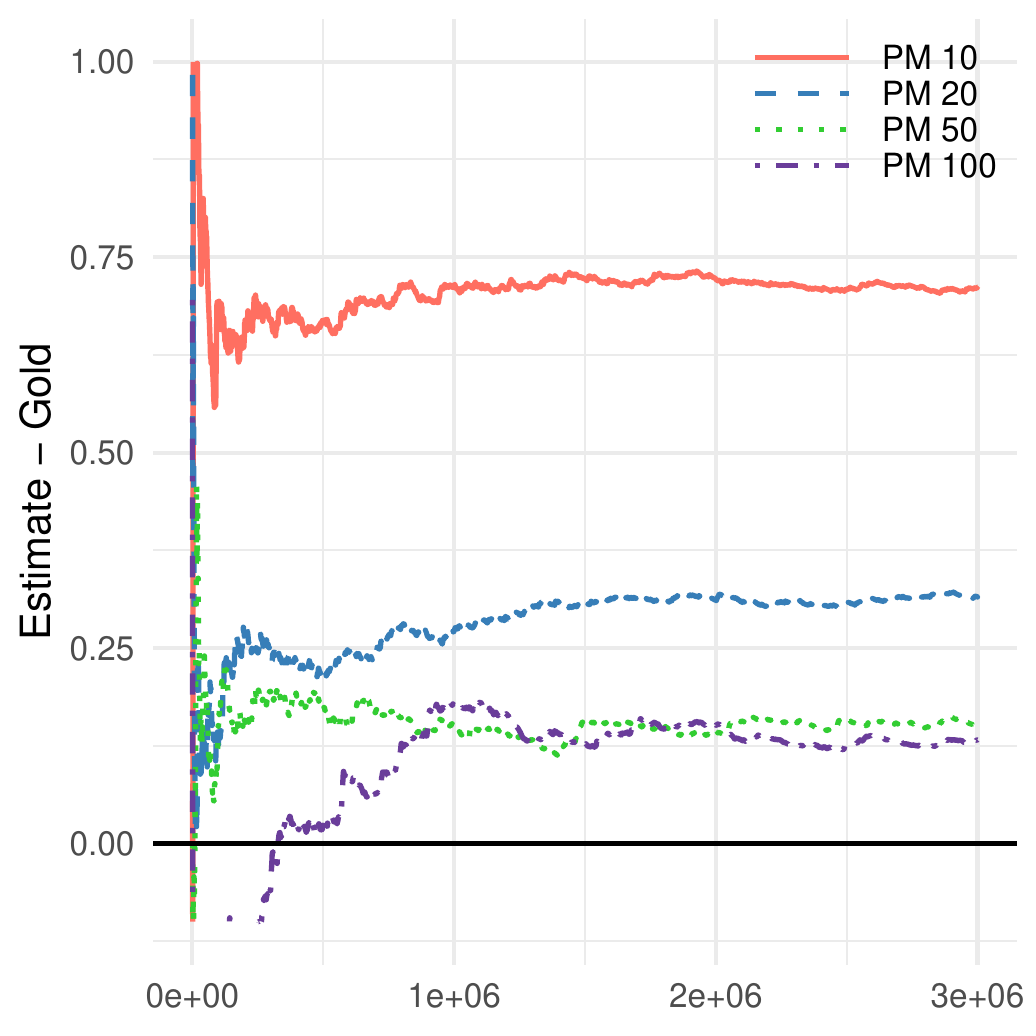}
    \caption{PM, mean}
\end{subfigure}
\begin{subfigure}[t]{0.4\textwidth}
    \includegraphics[scale=0.35]{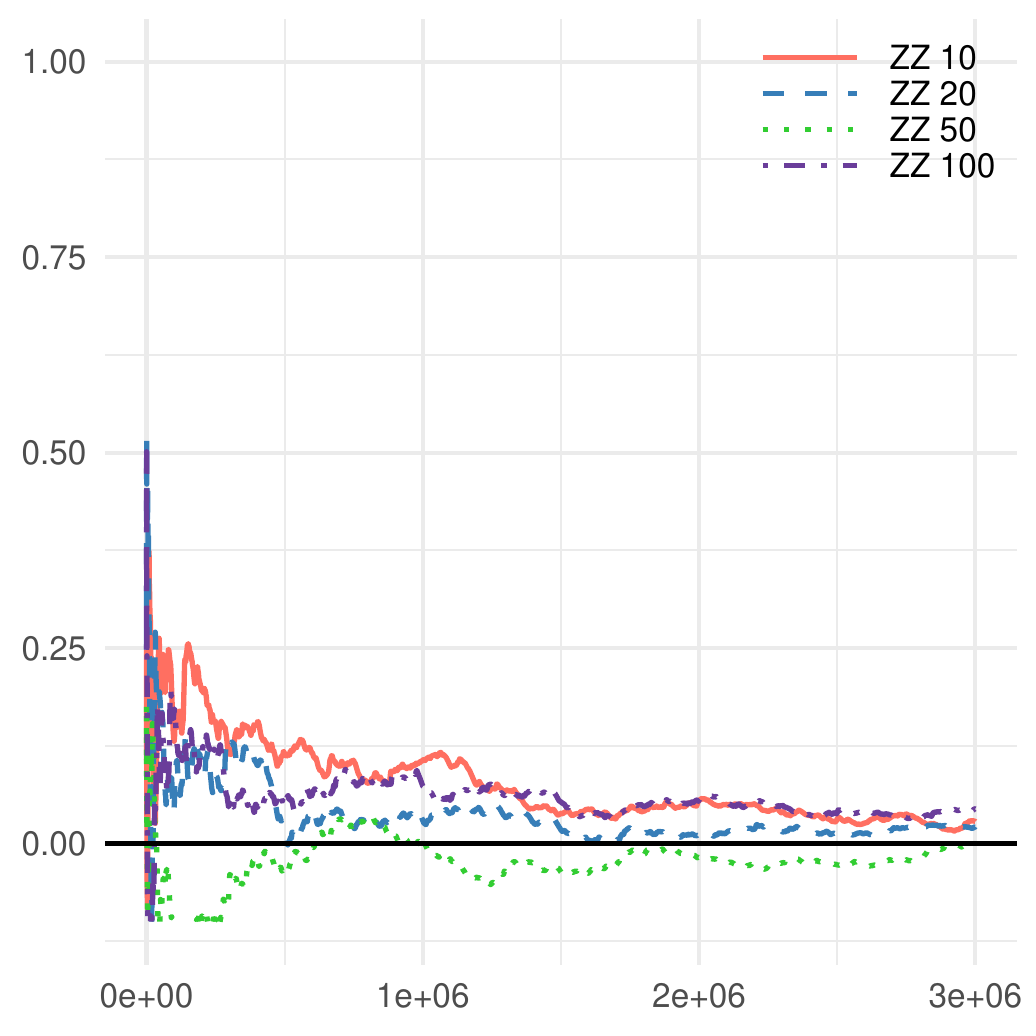}
    \caption{ZZ, mean}
\end{subfigure}
\begin{subfigure}[t]{0.4\textwidth}
    \includegraphics[scale=0.35]{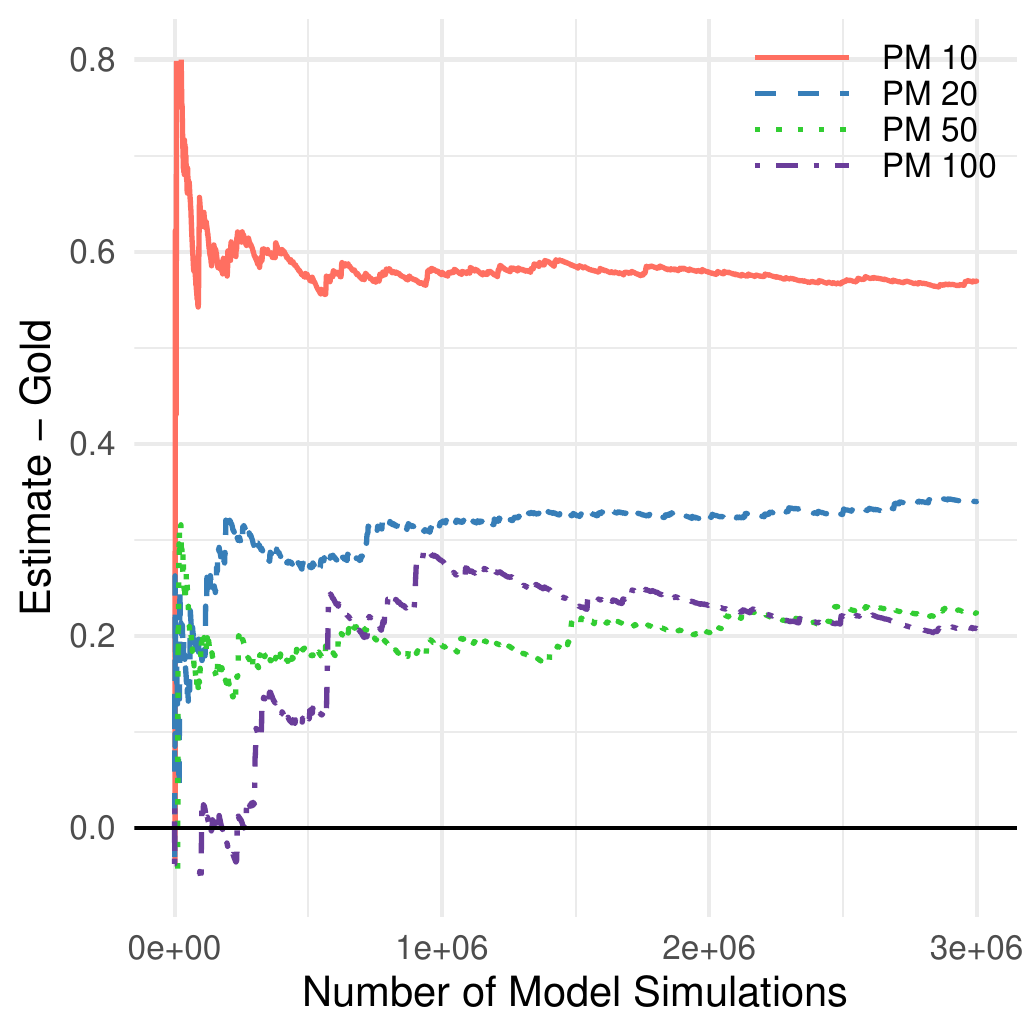}
    \caption{PM, standard deviation}
\end{subfigure}
\begin{subfigure}[t]{0.4\textwidth}
    \includegraphics[scale=0.35]{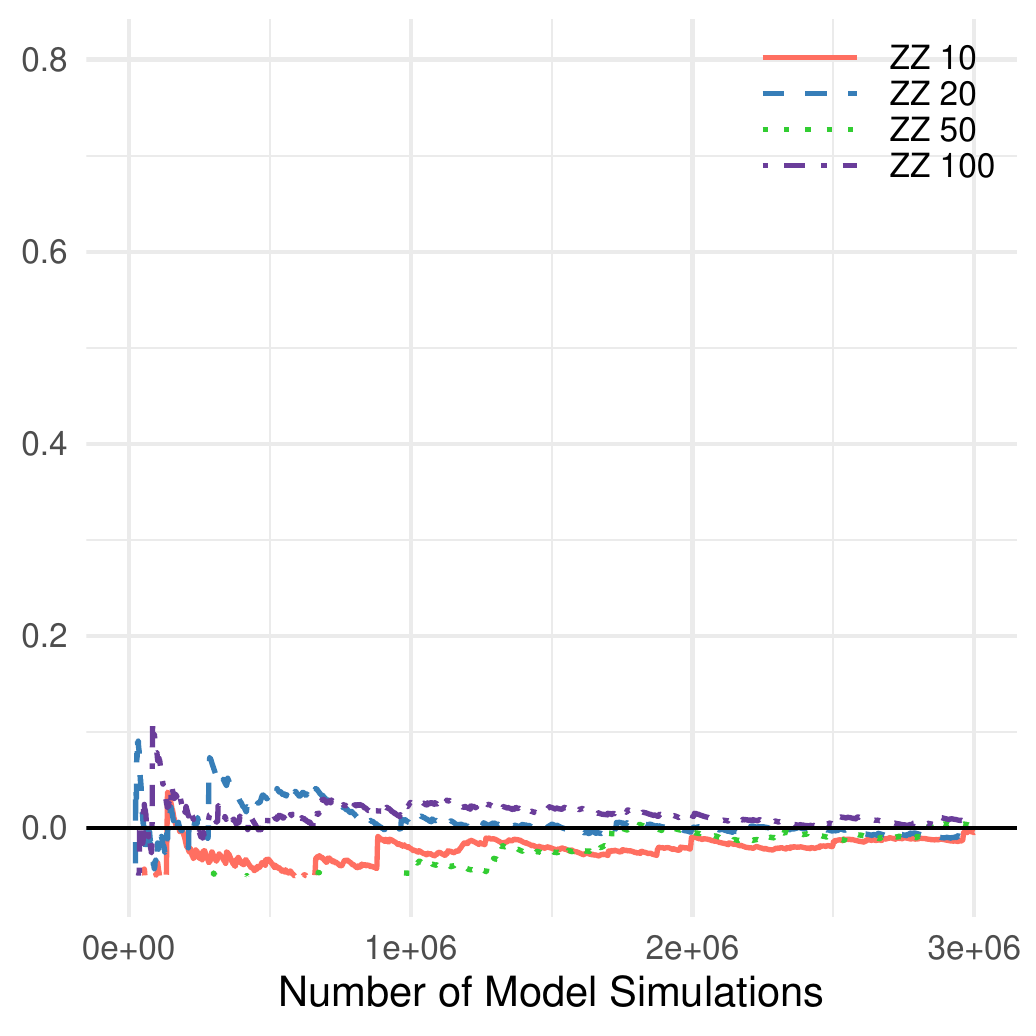}
    \caption{ZZ, standard deviation}
\end{subfigure}
\captionsetup{width=1.1\linewidth}
\caption{
We investigate computational efficiency by plotting
approximation accuracy as a function of the number of times the algorithms draw a simulation from the model. 
Using a long run of zig-zag sampling as the gold standard and an essentially exact approximation of $\pi(\theta\mid\MD_n)$, the plots chart the accuracy of the bP-MCMC (\textit{PM}) and zig-zag (\textit{ZZ}) samplers in the g-and-k model for $n=100$ along the $y$-axis as the number of model simulations increases along the $x$-axis. 
In the top row, this is done by plotting 
differences between true and estimated posterior means for bP-MCMC with $m\in\{10,20,50,100\}$ (\textbf{top left}) and zig-zag sampling with $b\in\{10,20,50,100\}$ (\textbf{top right}).
Meanwhile, the bottom row  compares the differences between true and estimated posterior standard deviations for bP-MCMC (\textbf{bottom left}) and zig-zag sampling (\textbf{bottom right}).
While the target sampled from the zig-zag is nearly independent of the choice of $b$, the bP-MCMC sampler faces a clear a trade-off between convergence speed and approximation accuracy that is determined by the magnitude of $m$. 
}
\label{Fig:gandk_PM_vs_ZZ_sims}
\end{center}
\end{figure}

\subsubsection{M/G/1 Queue}
The M/G/1 queuing model sees customers arrive at a single server with independent inter-arrival times exponentially distributed with arrival rate $\theta_3\ge0$. They are then served with independent service times that are distributed uniformly between $\theta_1$ and $\theta_2$, where $\theta_1,\theta_2\ge0$. The only observed data from the system, $y_i$, are the inter-departure times, and the goal is inference on the model parameters $\theta=(\theta_1,\theta_2,\theta_3)^\top$.

While simple to state, the likelihood for the M/G/1 queuing model is complicated, and many authors have instead proposed to conduct inference using simulation-based Bayesian methods (see, e.g., \citealp{bernton2019approximate}, and \citealp{drovandi2021comparison} for two examples). To this end, we now compare the accuracy and computational cost of the bP-MCMC-based approach to our zig-zag sampler in the M/G/1 model {using the MMD distance}. We generate $n=100$ observed inter-departure times with $\theta_0=(1,5,0.2)^\top$, and consider independent and uniform priors $\mathcal{U}[0,1]^3$.

\Cref{Fig:MG1_PM_vs_ZZ_n1000} plots the posteriors for $\theta_1$ and $\theta_2$ across the two different algorithms. The results again demonstrate that $\overline{\pi}(\theta\mid\MD_{m,n})$ varies with the choice of $m$, but less than so than in the g-and-k model. In contrast, the behavior of our PDMP algorithm  is independent of $b$ up to Monte Carlo error. However, in contrast to the g-and-k example, the zig-zag sampler and the bP-MCMC approaches behave similarly for moderate $m$. This highlights that there are indeed situations where P-MCMC algorithms can reliably target the exact posterior. However, we note that, in contrast to our proposed zig-zag algorithm, in any given example, it is infeasible to know a priori how large $m$ must be taken to ensure that the bP-MCMC target $\overline{\pi}(\theta\mid\MD_{m,n})$ matches the infeasible target of interest, ${\pi}(\theta\mid\MD_{n})$.

The results for $\theta_3$ are similar for both bP-MCMC and the zig-zag sampler; also, for $m\ge10$, the computational efficiency of the bP-MCMC and the zig-zag samplers are similar. Therefore, we do not give these results in the main paper for brevity.

%PDMPs_MMD_regression2.Rmd
\begin{figure}[!ht]%[t!]
%\vskip -1cm %-0.1in
\begin{center}
\includegraphics[trim= {0.0cm 0.00cm 0.0cm 0.0cm}, clip,  width=0.35\columnwidth]{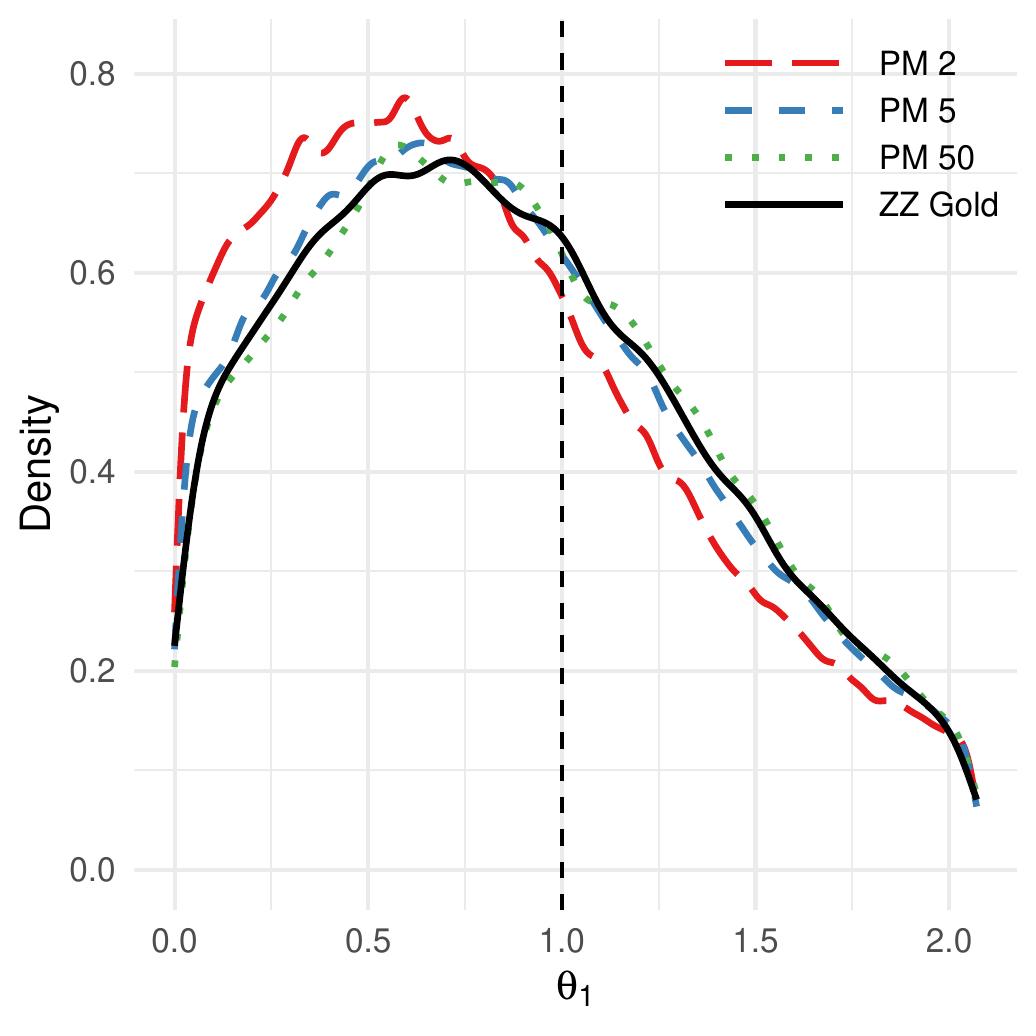}
\includegraphics[trim= {0.0cm 0.00cm 0.0cm 0.0cm}, clip,  width=0.35\columnwidth]{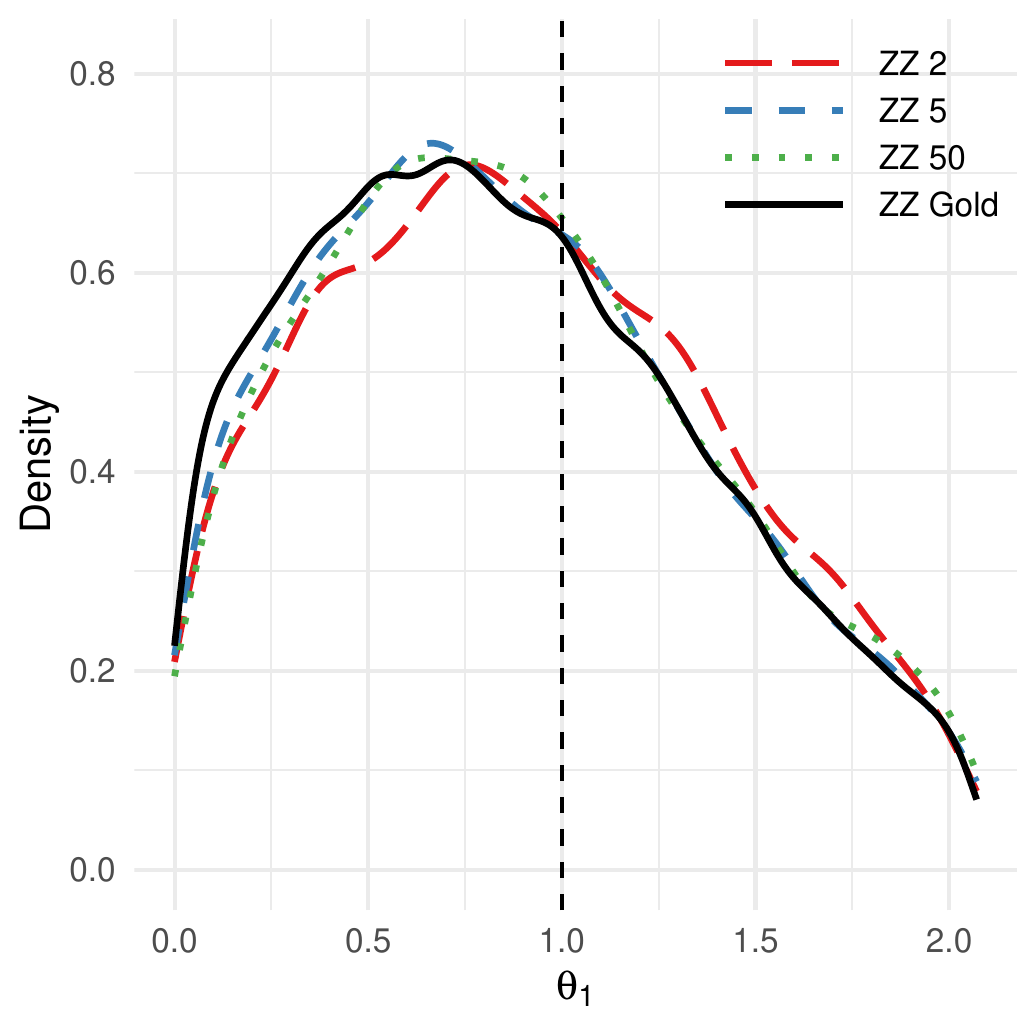}\\
\includegraphics[trim= {0.0cm 0.00cm 0.0cm 0.0cm}, clip,  width=0.35\columnwidth]{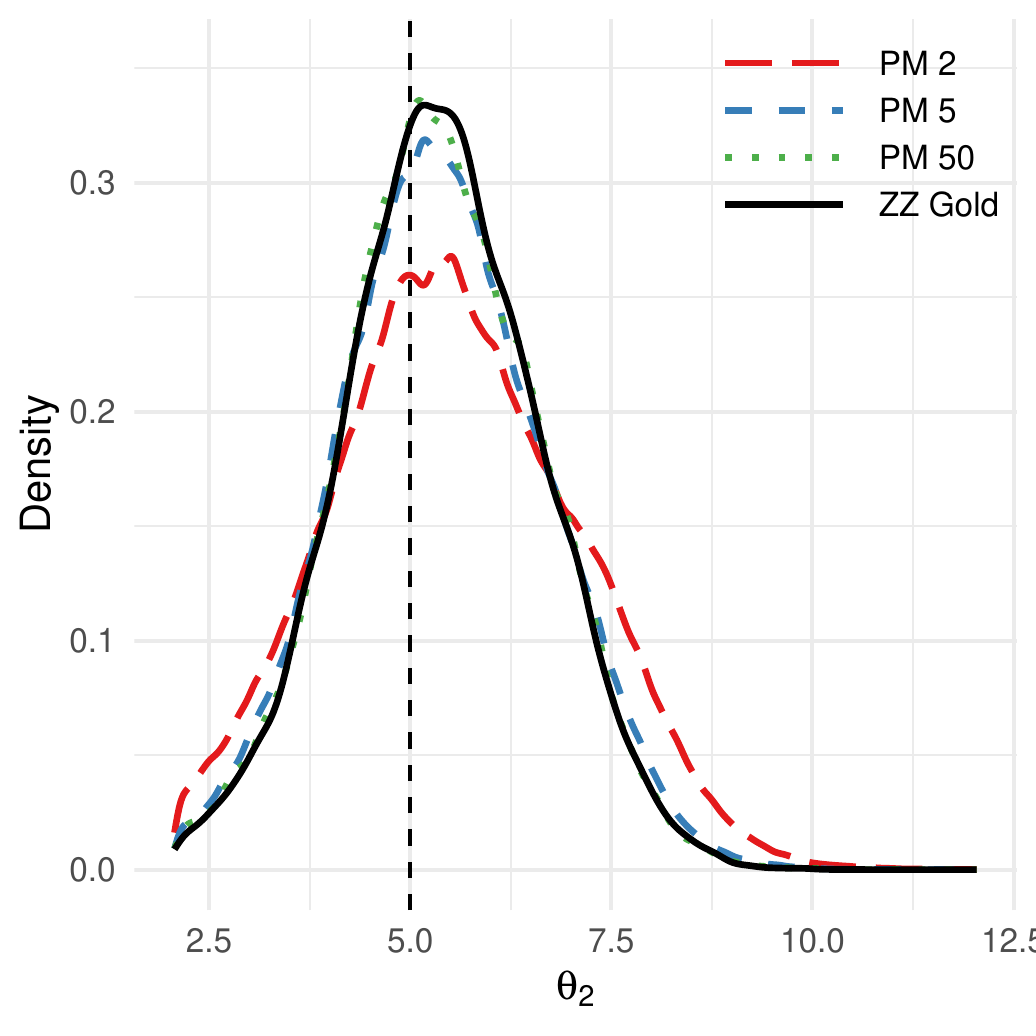}
\includegraphics[trim= {0.0cm 0.00cm 0.0cm 0.0cm}, clip,  width=0.35\columnwidth]{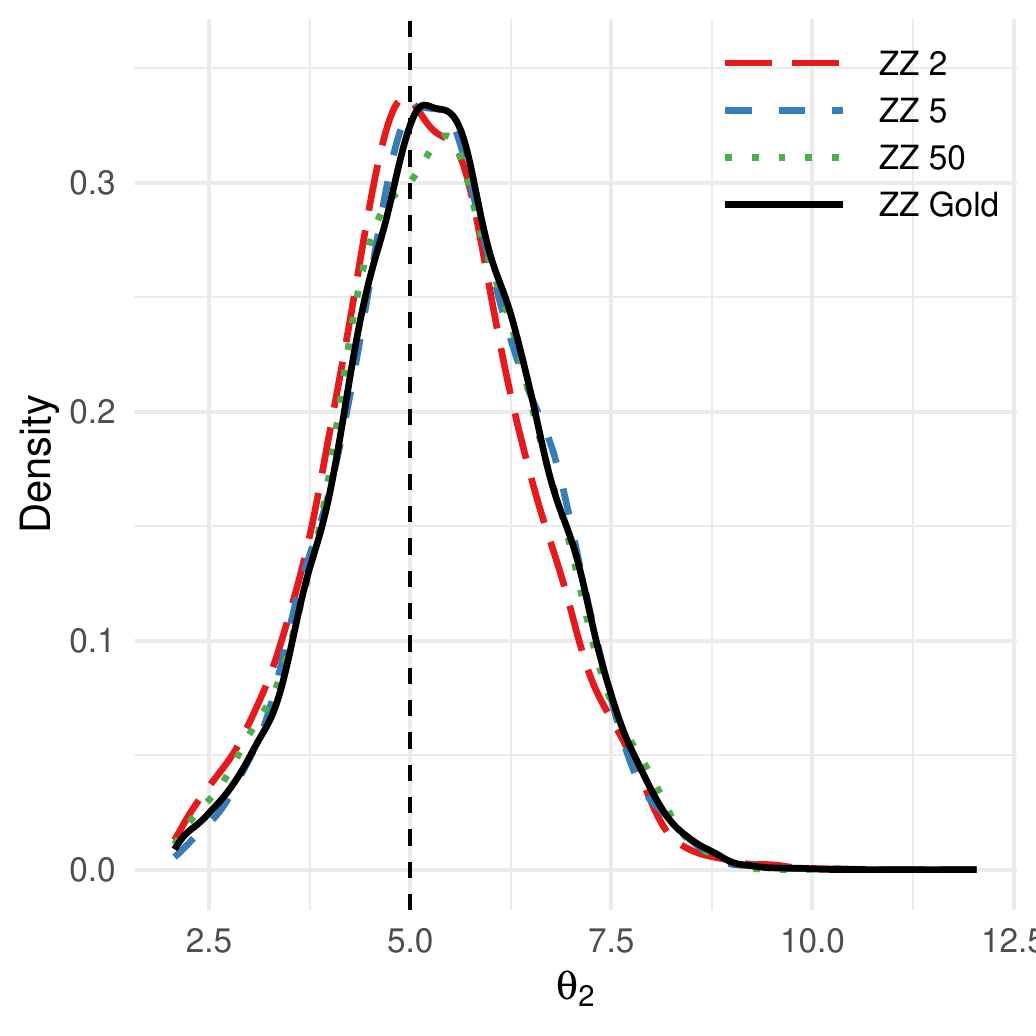}\\
%trim={<left> <lower> <right> <upper>}
\captionsetup{width=1.1\linewidth}
\caption{
$\MMD$-Bayes posterior density for $\theta_1$  and $\theta_2$ in the M/G/1 model. 
Throughout all panels,  \textit{ZZ Gold} should be thought of corresponding to $\pi(\theta\mid\MD_n)$, and was obtained using a long run of zig-zag (\textit{ZZ}) sampling. 
The \textbf{left} panel displays bP-MCMC (\textit{PM}) posterior approximations $\overline{\pi}(\theta\mid\MD_{m,n})$ for $m\in \{2,5,50\}$, which should be contrasted to the outcomes of the zig-zag with $b \in \{2,5,50\}$ in the \textbf{right} panel.
%
%While the posterior approximations of bP-MCMC vary substantively with $m$, the choice of $b$ has no impact on the approximations obtained via  zig-zag sampling.
%
}
\label{Fig:MG1_PM_vs_ZZ_n1000}
\end{center}
\end{figure}

%\caption{
%
%$\MMD$-Bayes posterior density for $a$ and $k$ in the intractable g-and-k model. 
%
%Throughout all panels,  \textit{ZZ Gold} should be thought of corresponding to $\pi(\theta\mid\MD_n)$, and was obtained using a long run of zig-zag (\textit{ZZ}) sampling. 
%
%The \textbf{left} panel displays bP-MCMC (\textit{PM}) posterior approximations $\overline{\pi}(\theta\mid\MD_{m,n})$ for $m\in \{2,10,50\}$, which should be contrasted to the outcomes of the zig-zag with $b \in \{2,10,50\}$ in the \textbf{right} panel.
%
%While the posterior approximations of bP-MCMC vary substantively with $m$, the choice of $b$ has no impact on the approximations obtained via  zig-zag sampling.
%
%}

\subsection{Intractable Loss Examples}
In this section, we compare the behavior of the two approaches for Gibbs measures based on intractable loss functions. The chosen examples study generalised posteriors which are motivated by robustness to the presence of outliers.
The first example uses the MMD loss and extends the work in \citet{alquier2023estimation} for robust estimation of copula models. The second example uses  $\beta$-divergences to model complicated observed galaxy velocities. Additional examples for outlier-robust Poisson regression, and robust linear regression example is considered in Appendix \ref{sec:full-MMD-regression-comparison}.

\subsubsection{Outlier-Robust Inference in Copulas}\label{sec:mmd-copula}

\citet{alquier2023estimation} proposed the MMD-based loss  $\MD_{m,n}^k$ presented in \Cref{subsec:applicability-of-results} for robust parameter estimation in copulas---a class of models that can be particularly susceptible to the presence of outliers.
%, as the model is defined requires 
%
In our first example, we use this loss to construct posteriors as in \citet{cherief2020mmd}, and compare inferences using bP-MCMC and zig-zag-based approaches. 
The parameter of interest is the correlation coefficient $\rho(\theta) = \frac{2}{1+\exp(-\theta)} -1$ which is parameterised with $\theta$ and features in a simple bivariate Gaussian copula $C_{\theta}(u) = \Phi_{2, \rho(\theta)}(\Phi^{-1}(u_1), \Phi^{-1}(u_2))$. 
Here, $u\in(0,1)^2$,  $\Phi^{-1}(\cdot)$ the inverse cdf of a standard normal, and $\Phi_{2, \rho}$ the cdf of a bivariate standard normal with correlation $\rho$.
To infer $\rho = \rho(\theta)$, we posit the prior $\rho(\theta) \sim \operatorname{U}(-1, 1)$, and construct a Gibbs measure using $\MD_{m,n}^k$ for $k$ being a radial basis kernel with lengthscale $\gamma = 1$. 

For simplicity, we consider a well-specified model and generate artificial datasets from the copula model with $n\in\{100,1000\}$ observations  with correlation $\rho=0.50$, before comparing posterior inferences on $\theta$ for the generated data using different choices of $m$ (bP-MCMC) and $b$ (zig-zag).
Our implementation of the zig-zag sampler for this example relies on the representation in \eqref{eq:candidate-indirect-sampling}, and the corresponding computational bounds are derived in \Cref{cor:MMD-copula} of \Cref{sec:MMDappendix}.

\Cref{Fig:copula_PM_vs_ZZ_n100} plots the results, and demonstrates that the implicit target $\overline{\pi}(\theta\mid\MD_{m,n})$ of bP-MCMC varies considerably with $m$.
Additionally, the comparison across $n=100$ and $n=1000$ empirically verifies the practical importance of Theorem \ref{thm:new}, and demonstrates that $m$ has to grow as a function of $n$ to ensure a good quality of posterior approximations: for $n=100$, choosing $m\ge10$ yields a reasonably accurate approximation, but for $n=1000$ we seem to require $m \ge 1000$. 
In contrast, the zig-zag algorithm draws samples from the correct target $\pi(\theta\mid \MD_n)$ irrespective of $b$ and $n$.
%

%PDMPs_MMD_regression2.Rmd
\begin{figure}[!ht]%[t!]
%\vskip -1cm %-0.1in
\begin{center}
\includegraphics[trim= {0.0cm 0.00cm 0.0cm 0.0cm}, clip,  width=0.35\columnwidth]{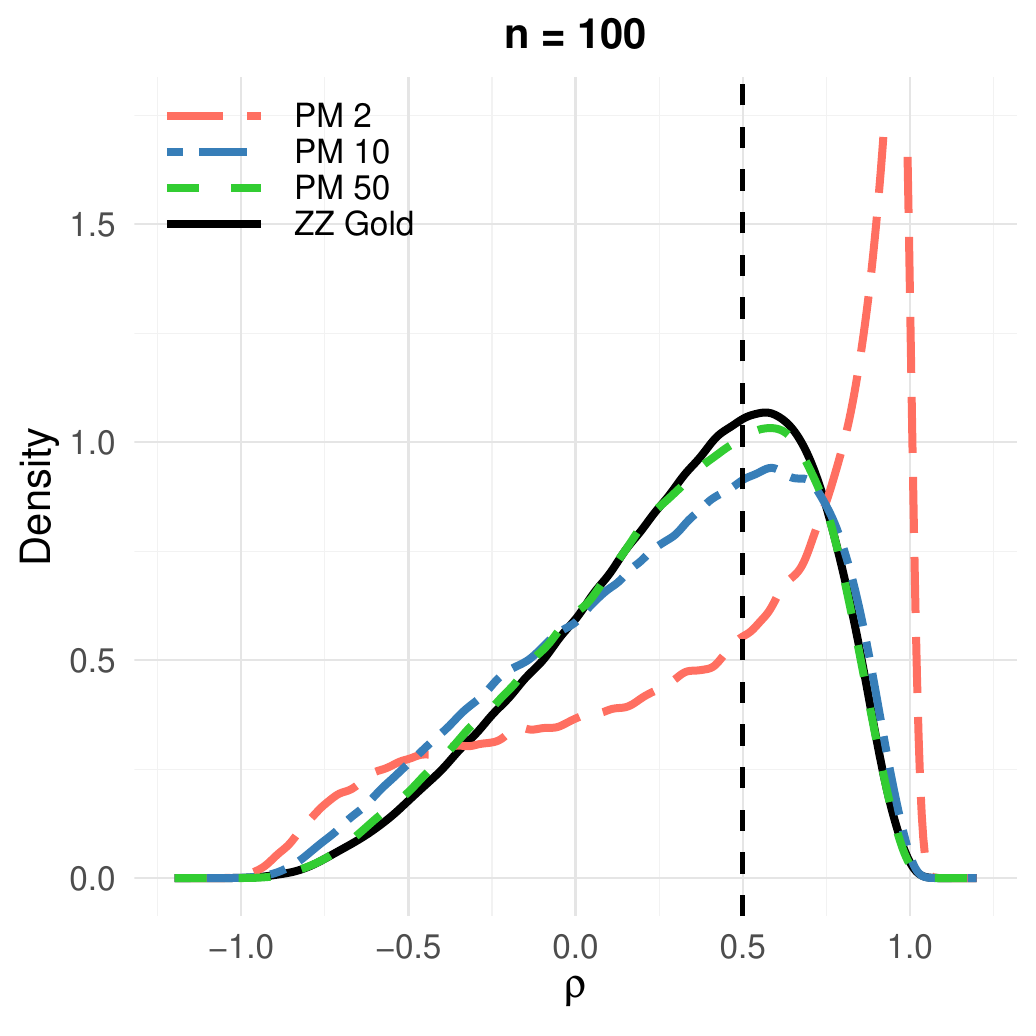}
\includegraphics[trim= {0.0cm 0.00cm 0.0cm 0.0cm}, clip,  width=0.35\columnwidth]{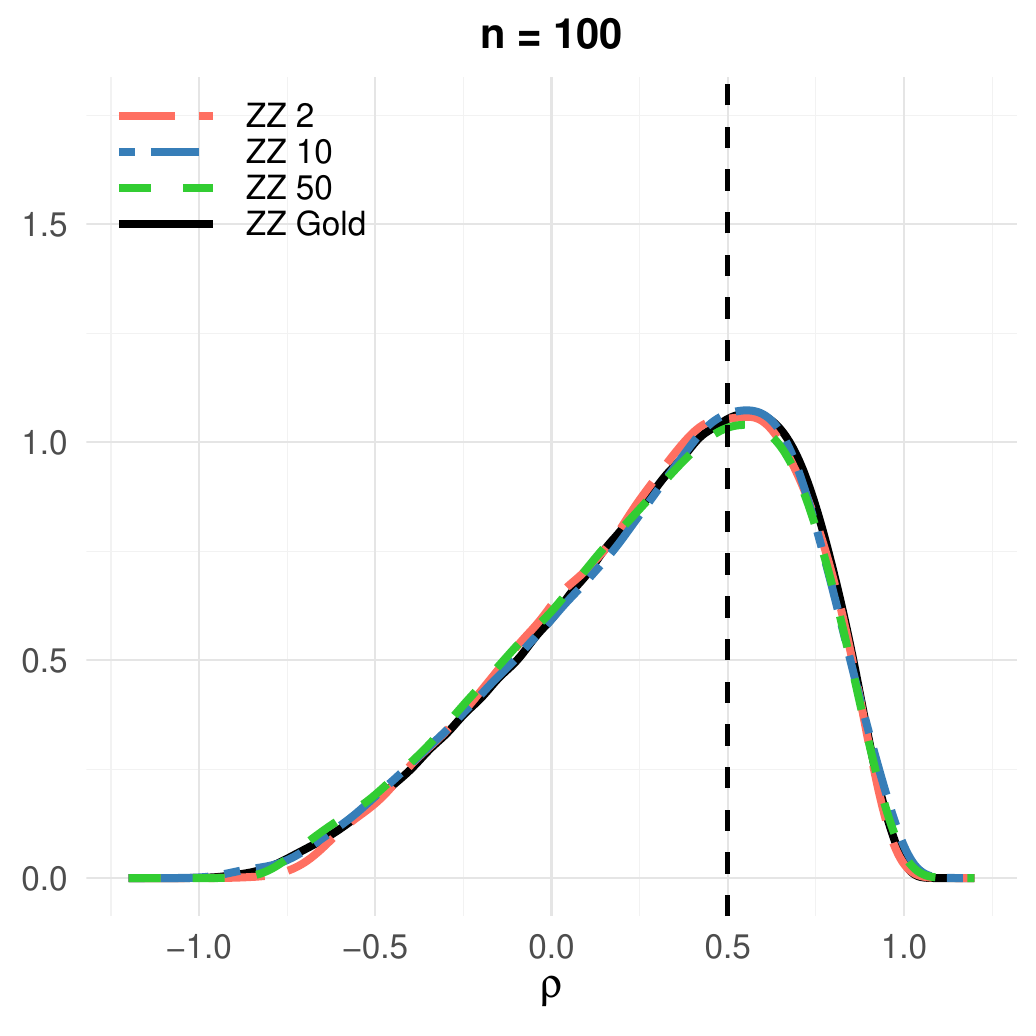}\\
\includegraphics[trim= {0.0cm 0.00cm 0.0cm 0.0cm}, clip,  width=0.35\columnwidth]{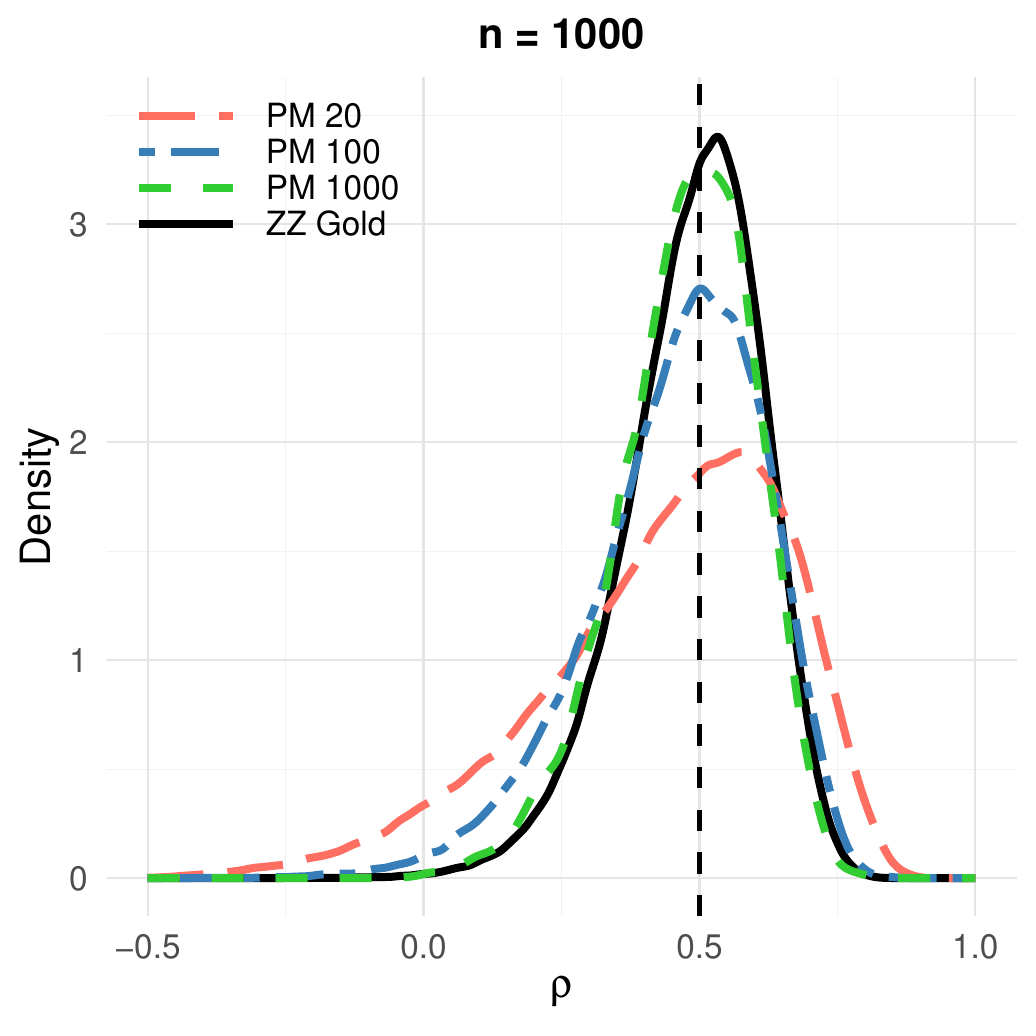}
\includegraphics[trim= {0.0cm 0.00cm 0.0cm 0.0cm}, clip,  width=0.35\columnwidth]{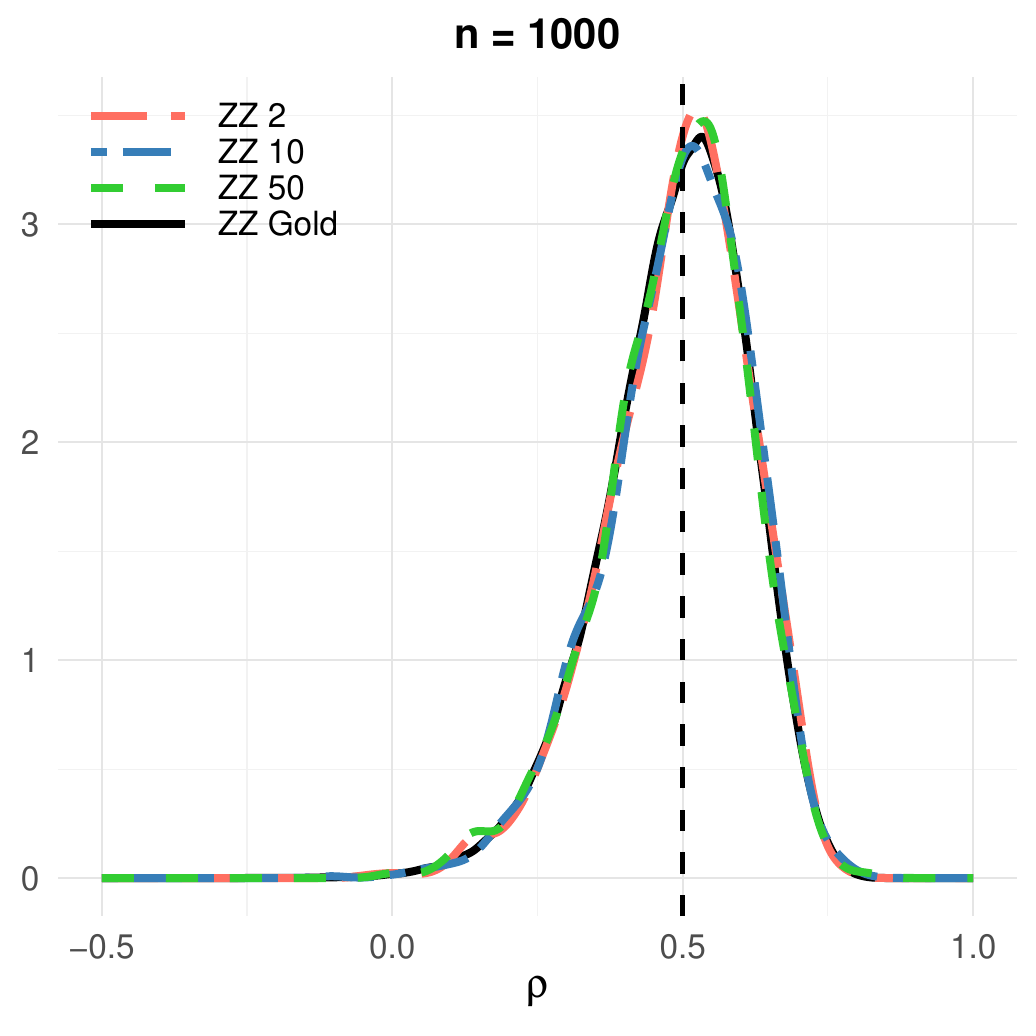}\\
%trim={<left> <lower> <right> <upper>}
\captionsetup{width=1.1\linewidth}
\caption{
$\MMD$-Bayes posterior density for $\rho$ in the Gaussian copula model. Please see \Cref{Fig:MG1_PM_vs_ZZ_n1000} for a detailed description of the figure. 
%
%Throughout all panels,  \textit{ZZ Gold} should be thought of corresponding to $\pi(\theta\mid\MD_n)$, and was obtained using a long run of zig-zag (\textit{ZZ}) sampling. 
%
%For $n=100$, the \textbf{top left} panel displays bP-MCMC (\textit{PM}) posterior approximations $\overline{\pi}(\theta\mid\MD_{m,n})$ for $m\in \{2,10,50\}$, which should be contrasted to the outcomes of the zig-zag with $b \in \{2,10,50\}$ in the \textbf{top right} panel.
%
%While the posterior approximations of bP-MCMC vary substantively with $m$, the choice of $b$ has no impact on the approximations obtained via  zig-zag sampling.
%
%In the bottom panels, the same phenomenon is replicated for the larger sample size $n=1000$.
%
%Despite the fact that we run bP-MCMC with the even larger choices of $m\in \{20,100,1000\}$ (\textbf{bottom left}), the zig-zag sampler performs more reliably, and without requiring any increase of $b$ (\textbf{bottom right}).
%
}
\label{Fig:copula_PM_vs_ZZ_n100}
\end{center}
\end{figure}

To contrast the computational overhead of zig-zag samplers and bP-MCMC, it is not enough to compare $b$ and $m$.
Instead, we consult \Cref{Fig:copula_PM_vs_ZZ_sims}, which for $n=1000$ compares posterior approximation error between both approaches relative to the number of times the algorithms simulate from the copula model.
The plot uses the difference between estimated and true posterior mean and posterior standard deviation as benchmarks; {the `true' mean and standard deviation are based on a zig-zag sampler with $b=500$.} 
The results illustrate the relative cost of reducing posterior approximation errors as a function of calls to the simulators, and shows that while the zig-zag sampler's computational efficiency seems relatively independent of $b$, the choice of $m$ interacts drastically with the performance of bP-MCMC.
Most obviously, the plots show that small choices of $m$ lead to biased estimates. In fact, the bias cannot even be overcome convincingly for large choices of $m$.
Additionally, large choices of $m$ significantly slow down the speed at which the algorithm converges.

\begin{figure}[!htp]%[t!]
%\vskip -1cm %-0.1in
\begin{center}
\begin{subfigure}[t]{0.4\textwidth}
    \includegraphics[scale=0.35]{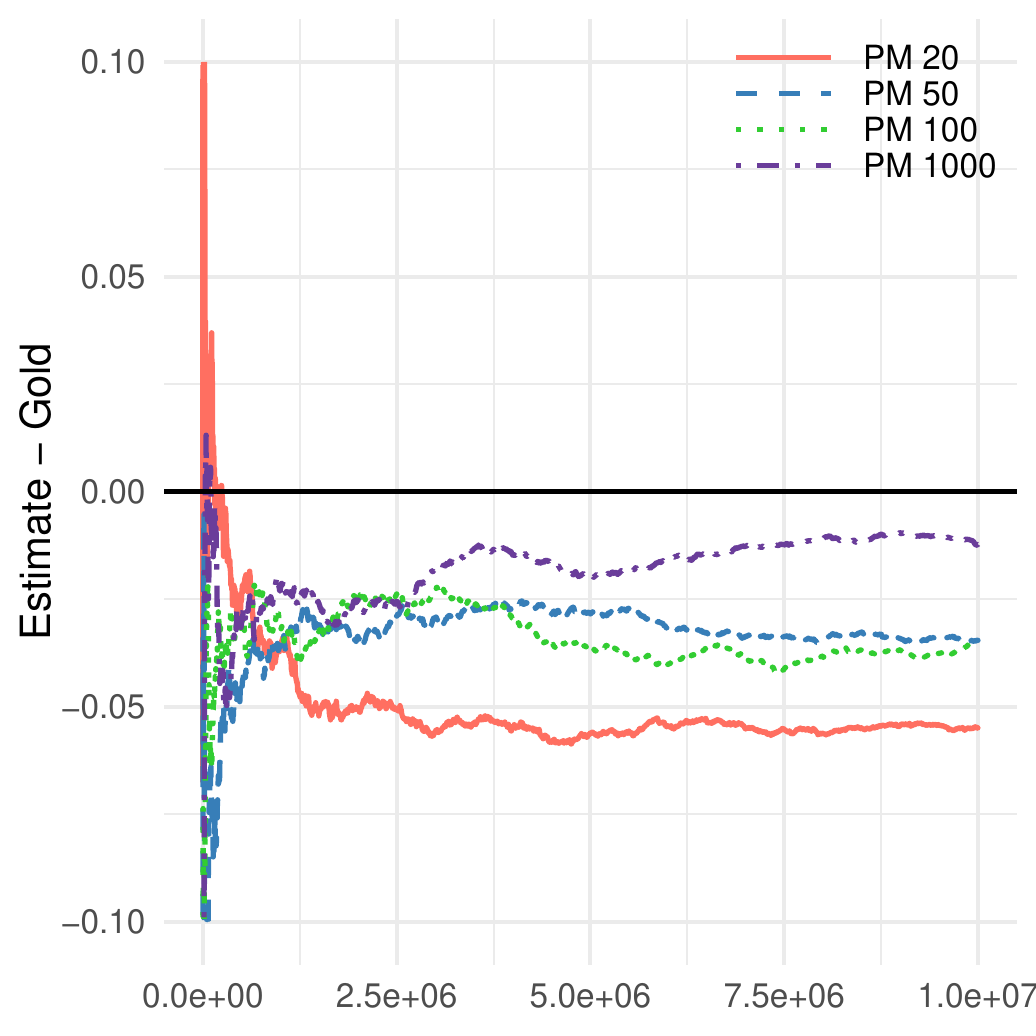}
    \caption{PM, mean}
\end{subfigure}
\begin{subfigure}[t]{0.4\textwidth}
    \includegraphics[scale=0.35]{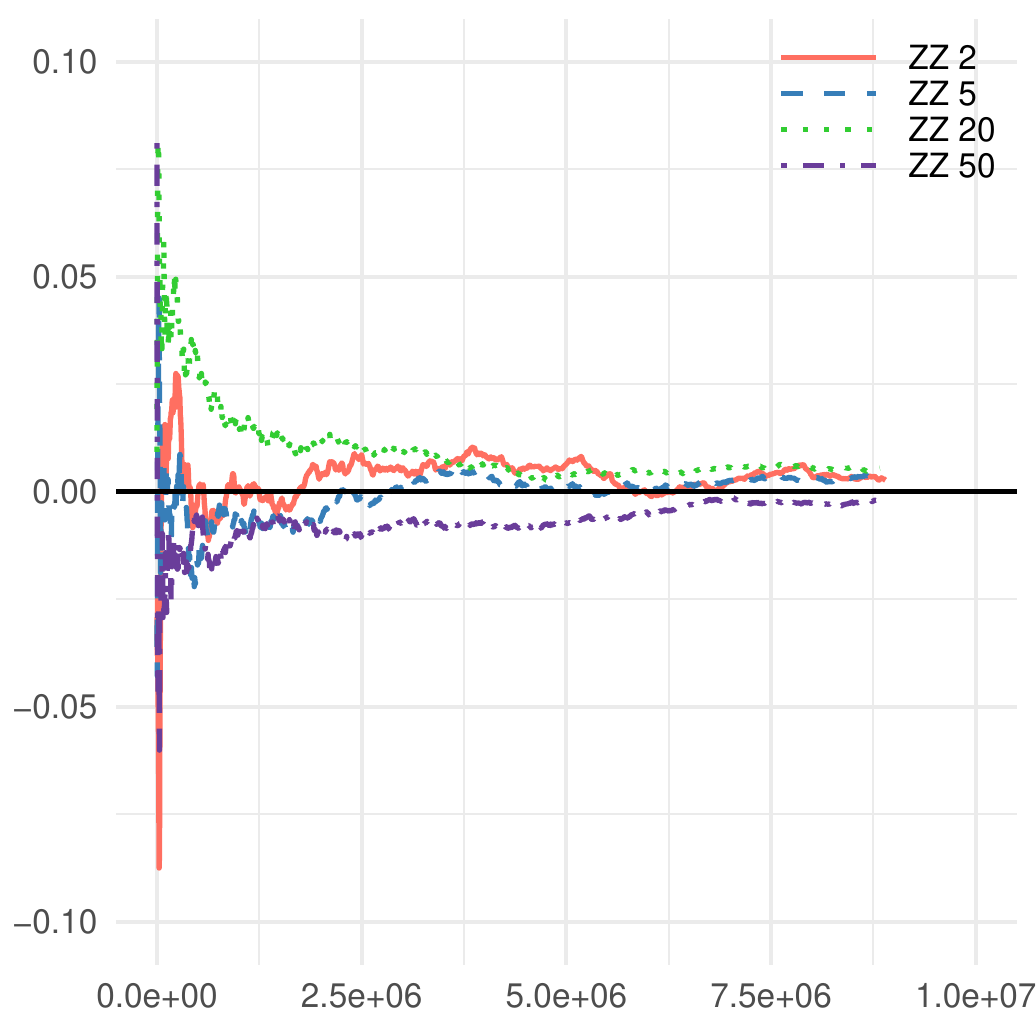}
    \caption{ZZ, mean}
\end{subfigure}
\begin{subfigure}[t]{0.4\textwidth}
    \includegraphics[scale=0.35]{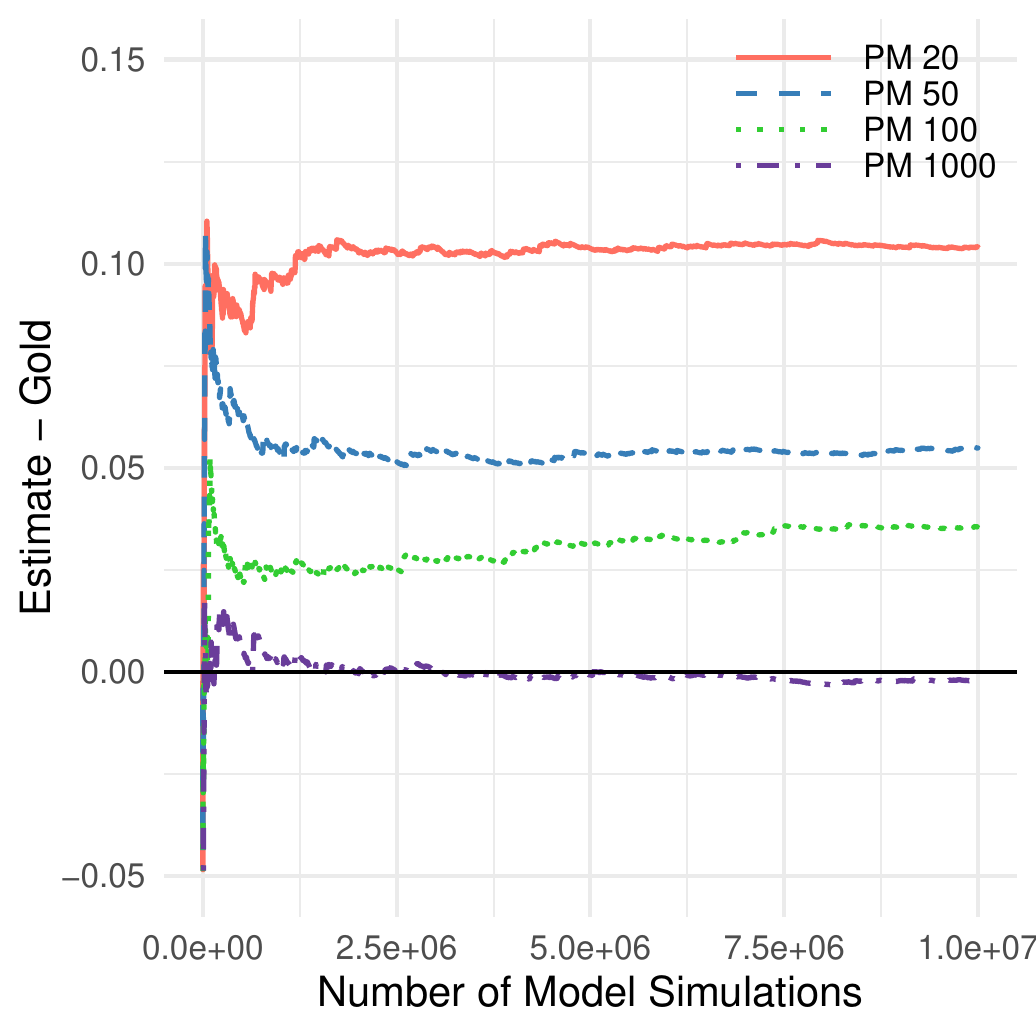}
    \caption{PM, standard deviation}
\end{subfigure}
\begin{subfigure}[t]{0.4\textwidth}
    \includegraphics[scale=0.35]{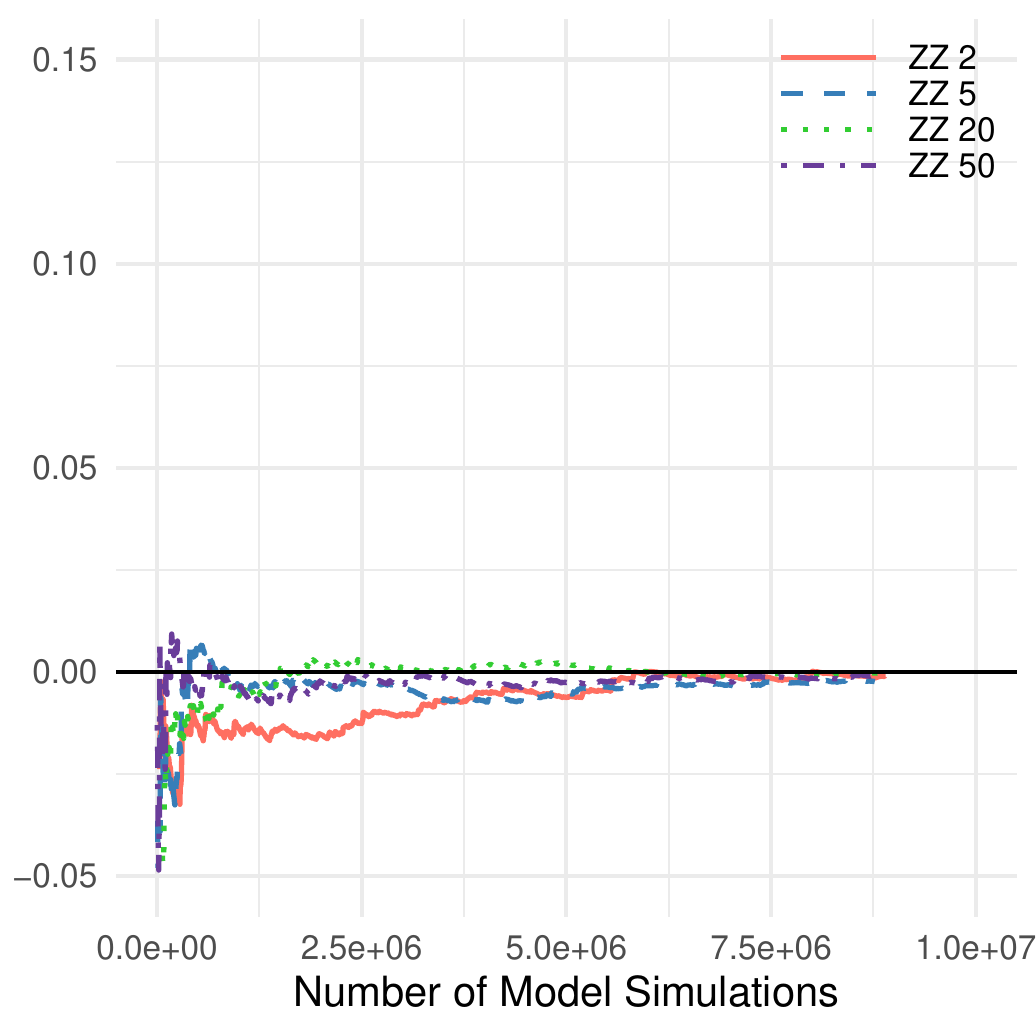}
    \caption{ZZ, standard deviation}
\end{subfigure}
\captionsetup{width=1.1\linewidth}
\caption{Computational efficiency comparison of bP-MCMC and the zig-zag algorithm. Please refer to \Cref{Fig:gandk_PM_vs_ZZ_sims} for a detailed description of the figure. 
%
%We investigate computational efficiency by plotting
%approximation accuracy as a function of the number of times the algorithms draw a simulation from the model. 
%
%Using a long run of zig-zag sampling as the gold standard and an essentially exact approximation of $\pi(\theta\mid\MD_n)$, the plots chart the accuracy of the bP-MCMC (\textit{PM}) and zig-zag (\textit{ZZ}) samplers in the Gaussian copula model for $n=1000$ along the $y$-axis as the number of model simulations increases along the $x$-axis. 
%
%In the top row, this is done by plotting 
%differences between true and estimated posterior means for bP-MCMC with $m\in\{20,50,100,1000\}$ (\textbf{top left}) and zig-zag sampling with $b\in\{2,5,20,50\}$ (\textbf{top right}).
%
%Meanwhile, the bottom row  compares the differences between true and estimated posterior standard deviations for bP-MCMC (\textbf{bottom left}) and zig-zag sampling (\textbf{bottom right}).
%
%While the zig-zag sampler's efficiency seems to behave similarly for all choices of $b$, the bP-MCMC sampler induces a trade-off between convergence speed and approximation accuracy that is determined by the magnitude of $m$. 
}
\label{Fig:copula_PM_vs_ZZ_sims}
\end{center}
\end{figure}

%\subsection{Outlier-Robust Inference in Regression}

%In the remainder of this section, we place the findings for the copula model into a larger context:
%
%In \Cref{Sec:RegressionMMD}, we show that the results of the previous section are not an artifact of the copula model, and replicate for inference in linear regression models.
%
%In \Cref{Sec:PoissonRegression},  we then go on to contrast this with inferences in Poisson regression models using a Gibbs measure based on $\beta$-divergences.
%
%This setting is hand-picked as a pedagogical illustration for the kind of conditions under which the bias induced by $\overline{\pi}(\theta\mid\MD_{m,n})$ can be expected to be relatively small, and in which P-MCMC methods can thus generally still compete with zig-zag sampling.

\subsection{Empirical Application: Robust Clustering}
%We now apply the proposed PDMP approach to conduct Bayesian mixture modeling using robust loss functions. 
Mixture modeling and clustering are fundamental pillars of applied Bayesian inference, and the applications and development of such models and methods is extensive. However, it is well-known that mixture models are highly susceptible to departures from the underlying modeling assumptions. For example, the presence of even a few outliers can cause Gaussian mixture models to produce poorly fitted clusters and/or select an excessive number of clusters; for examples, see \cite{10.3150/18-BEJ1087}, and \cite{miller2019robust}.

While certain generalized posteriors are known to be robust to the presence of outliers, their use in mixture modeling has not been seriously considered: loss functions that are robust to outliers involve the presence of intractable integrals, and until now posterior samples from the intended target could not be drawn reliably. While \cite{rigon2023generalized} have proposed to conduct probabilistic clustering using generalized posteriors, the author's approach is based on a loss function which measures pairwise distances between observations, resulting in a k-means clustering solution that could overpower the underlying information brought by the probabilistic model. 

In this section, we apply the $\beta$-divergence to conduct robust probabilistic clustering under an assumed Gaussian mixture model (GMM). Taking $\beta>0$ naturally flattens the tails of the divergence, making inferences based on the $\beta$-divergence less susceptible to outliers, allowing us to mitigate known issues with Bayesian mixture models in the presence of outliers. We demonstrate the power of this approach by applying the -- clearly misspecified -- GMM to the Shapley galaxy dataset analyzed in \cite{miller2019robust}. The dataset, obtained from \cite{drinkwater2004large}, consists of Galaxy redshifts, which are velocities with respect to Earth, for 4215 galaxies in the Shapley concentration regions. The Shapley data has three distinct clusters and a long right tail, suggesting the presence of extreme outliers or a fourth cluster with large variance (see the shaded histograms within Figure \ref{fig:clustering_results}). These data features cause significant issues for Bayesian GMMs, including poor mixing, and instability in component assignments.

Using the coarsened posterior, the analysis in \cite{miller2019robust} suggests that a three or four mixture component solution adequately fits the Shapley data. {We also find that the standard Bayesian GMM clustering becomes unstable for more than four clusters, with \Cref{fig:clustering_results}, panel (f), indicating that the within-cluster-sum-of-squares increases when adding additional clusters.}
Hence, in what follows we only present results for the four component GMM. 

In \Cref{fig:clustering_results} we plot the {posterior mean for each mixture component as well as the density of their implied mixture} for the $\beta$-divergence loss  and the 
Bayesian posterior. Interestingly, from \Cref{fig:clustering_results}, panels (a)-(d), we see that the value of $\beta$ distinctly alters the shape of the component densities. Across each value of $\beta$ we identify two components associated with small and large velocities. However, for smaller values of $\beta$, panels (a)-(b), we identify two components with similar scales but different locations. Interestingly, as $\beta$ becomes larger, panels (c)-(d), these two components become dissimilar: one cluster becomes highly-peaked and the other becomes more elongated. The Bayes posterior exhibits distinct components to those obtained under the $\beta$-divergence: two component densities closely match the first two data peaks,  one elongated component matches the far right tail of the data, associated with large velocities, and the final component has a density that is nearly flat in comparison to the other components.

The difference in the component density shapes and locations is a consequence of the $\beta$-divergence loss. The Bayes posterior is highly-sensitive to disagreements between the model and observed data in the tail of the assumed model, which forces Bayes to closely match tail observations. In contrast, the $\beta$-divergence is less sensitive as such points are down-weighted within the loss when $\beta>0$. Consequently, when the resulting components are clearly non-Gaussian, the $\beta$-divergence based posterior can be much less sensitive to extreme observations than standard Bayesian methods. 

In general it is unclear how one should choose the value of $\beta$ to model this data. \cite{basu1998robust} suggested to choose $\beta$ based on comparing efficiency and robustness, while \cite{warwick2005choosing} proposed to estimate $\beta$ using (an estimator of) asymptotic mean squared error. Such approaches are clearly ineffective for mixture models since these notions do not directly extend to parameters that are not necessarily identified. Therefore, while such a problem is clearly relevant, it is also beyond the scope of this paper, and warrants future study.

%This abrupt shift when moving from three to four components under the Bayesian GMM also results in dramatic shifts for the cluster locations and scales. In contrast, the $\beta$-divergence does not exhibit the same abrupt changes in cluster shape and support when moving from three to four components, which is due to the loss functions ability to automatically flatten extreme points and down-weight the impact of extreme observations on the clustering assignments. 

\begin{figure}[htbp]
    \centering
    
    % Top row
    \begin{subfigure}[b]{0.48\textwidth}
        \centering
        \includegraphics[width=\textwidth]{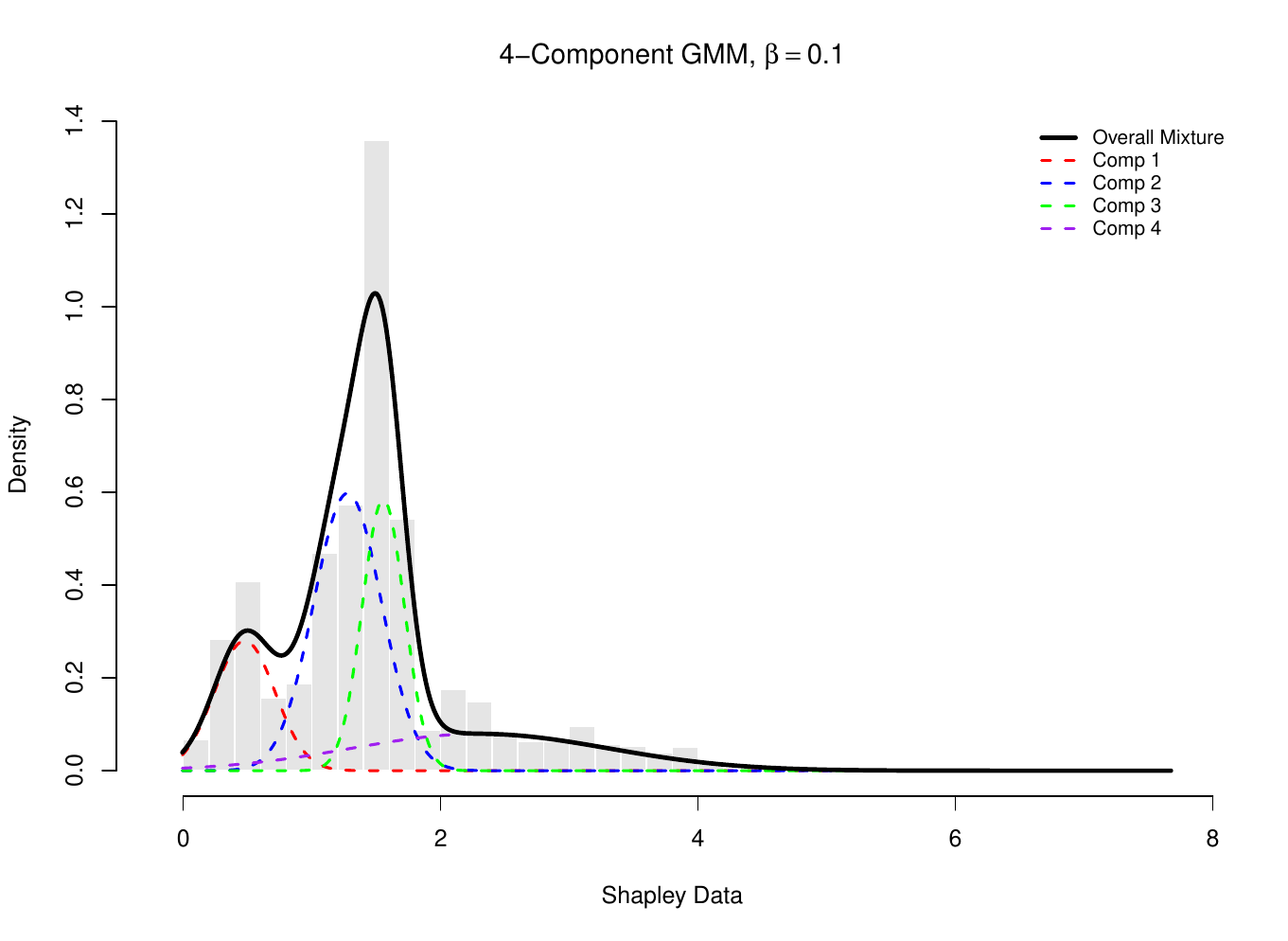}
        \caption{$\beta$-divergence, 4-Component Mixture, $\beta=0.10$ }
        \label{fig:betad_3c}
    \end{subfigure}
    \hfill
    \begin{subfigure}[b]{0.48\textwidth}
        \centering
        \includegraphics[width=\textwidth]{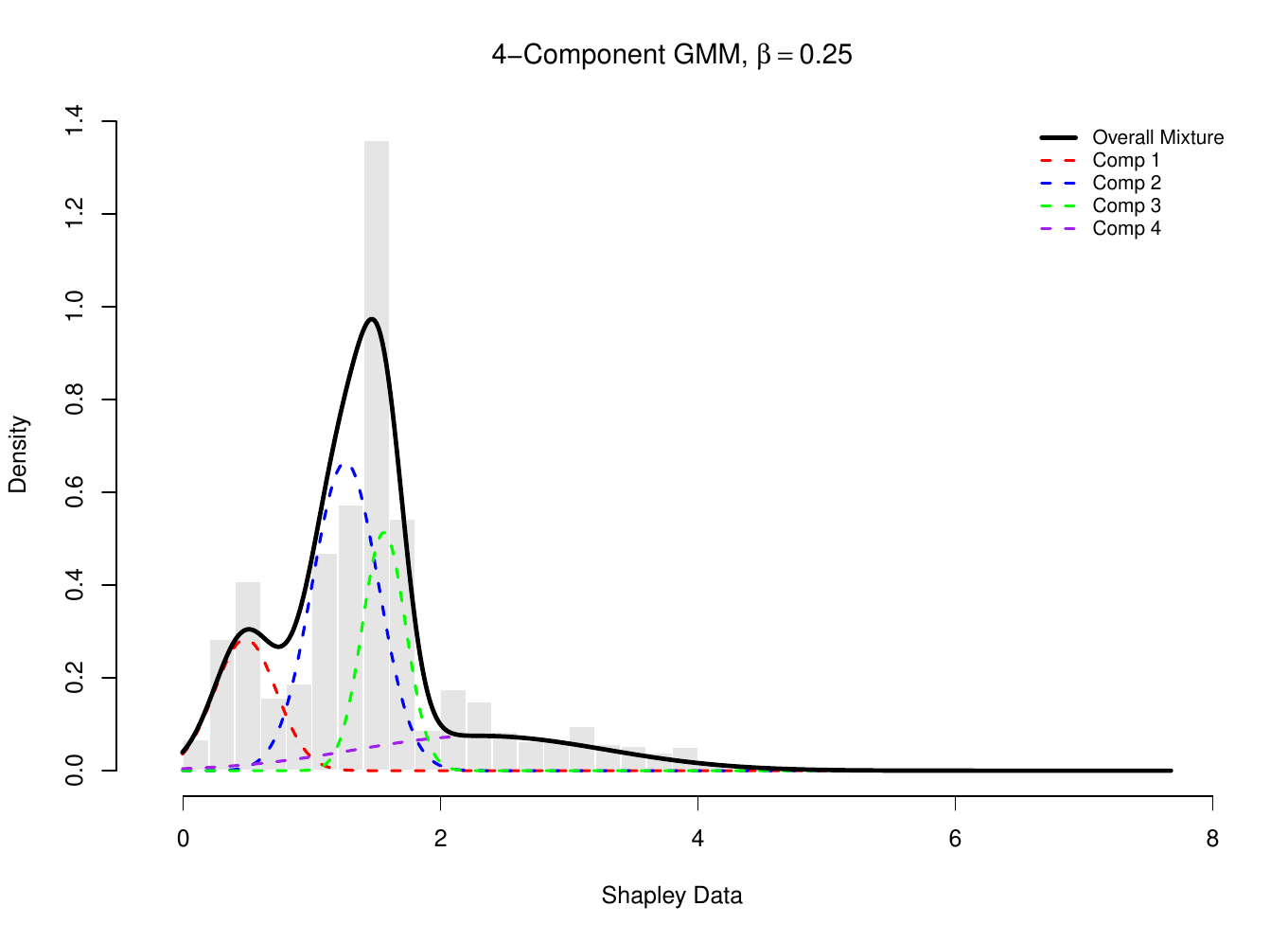}
        \caption{$\beta$-divergence, 4-Component Mixture, $\beta=0.25$}
        \label{fig:betad_4c}
    \end{subfigure}
    
    \vspace{0.5cm}
    
    % Bottom row
    \begin{subfigure}[b]{0.48\textwidth}
        \centering
        \includegraphics[width=\textwidth]{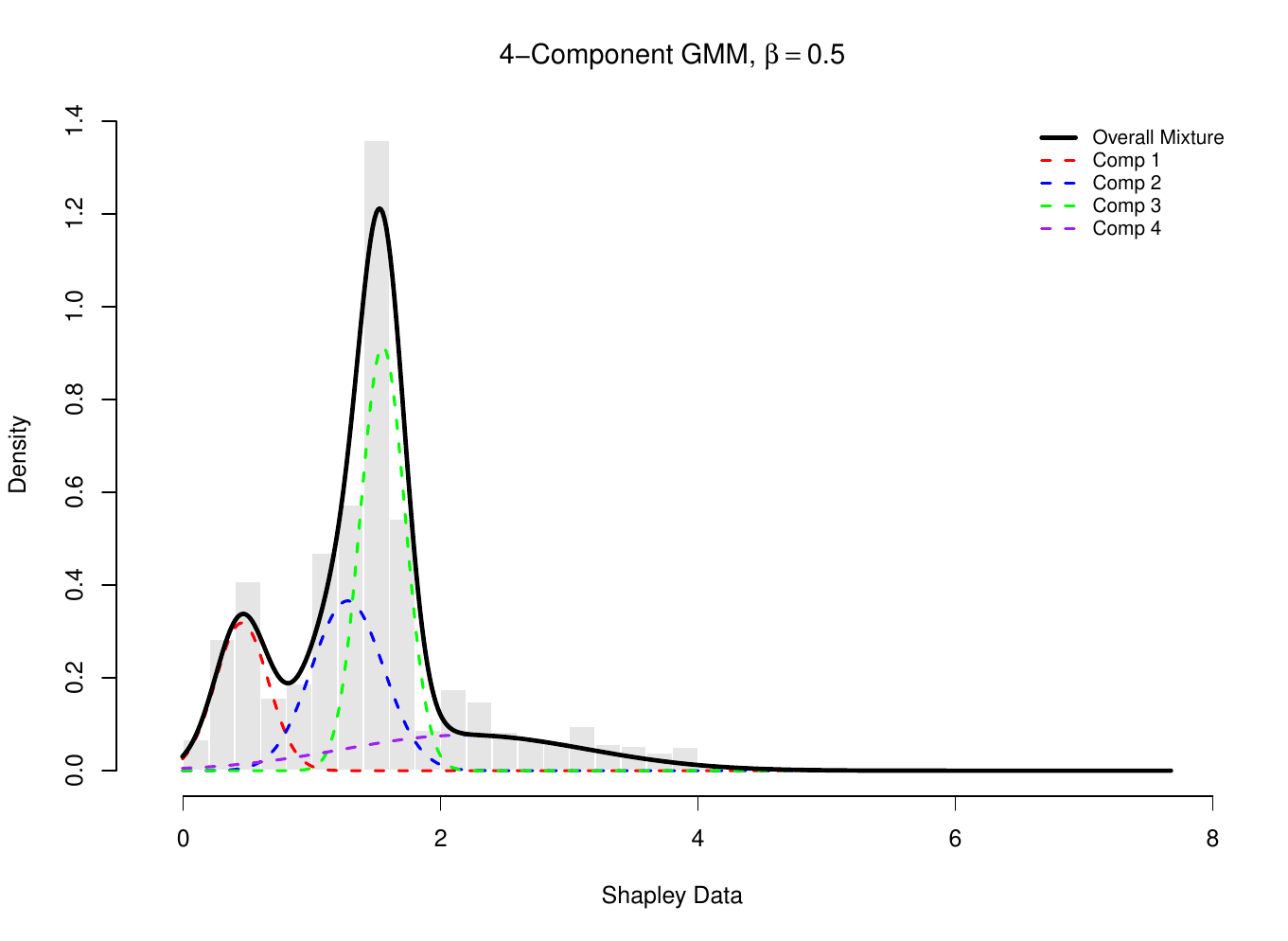}
        \caption{$\beta$-divergence, 4-Component Mixture, $\beta=0.50$}
        \label{fig:bayes_4c}
    \end{subfigure}
    \hfill
    \begin{subfigure}[b]{0.48\textwidth}
        \centering
        \includegraphics[width=\textwidth]{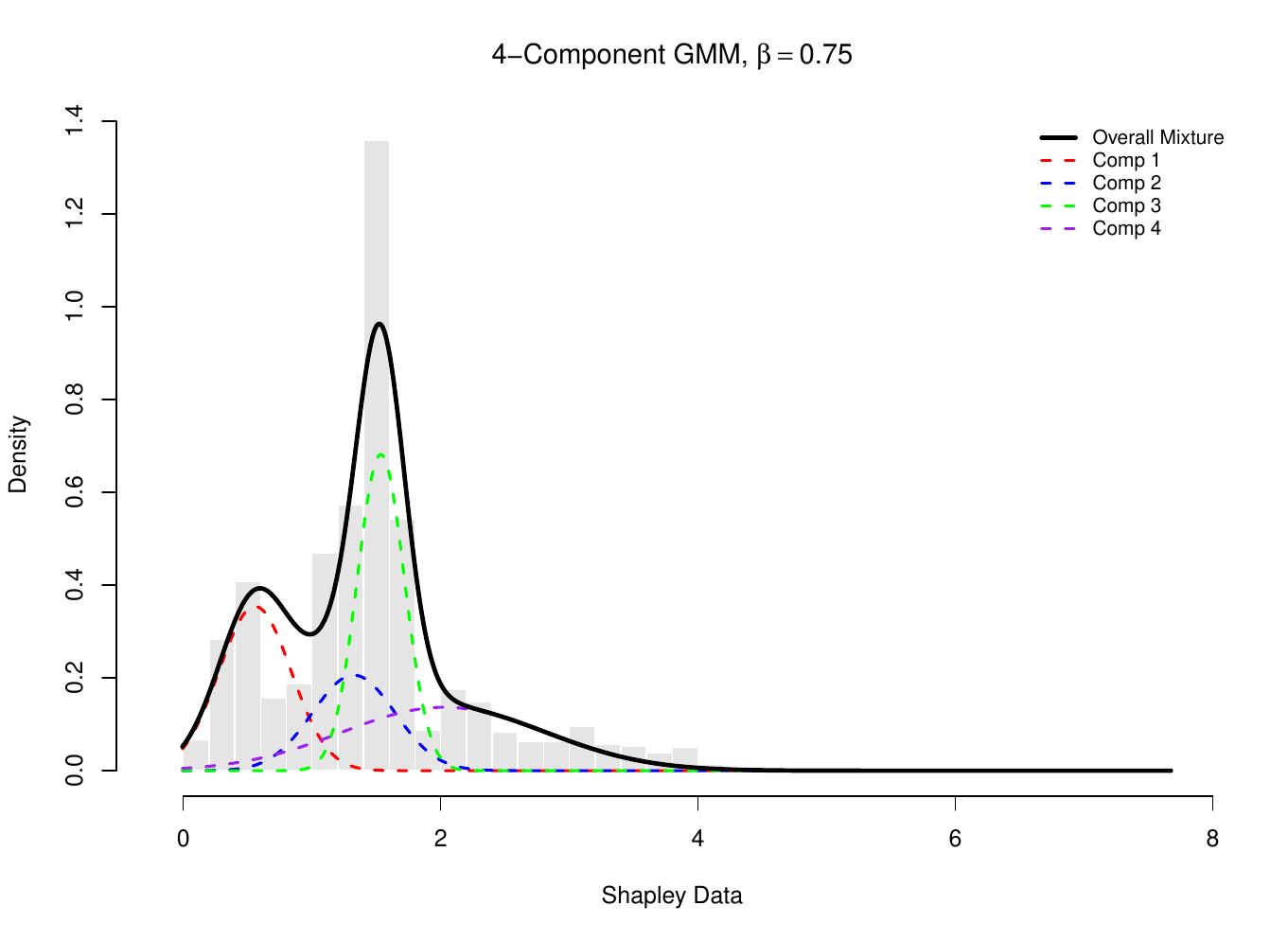}
        \caption{$\beta$-divergence, 4-Component Mixture, $\beta=0.75$}
        \label{fig:elbow}
    \end{subfigure}

   \vspace{0.5cm}

  \begin{subfigure}[b]{0.48\textwidth}
        \centering
        \includegraphics[width=\textwidth]{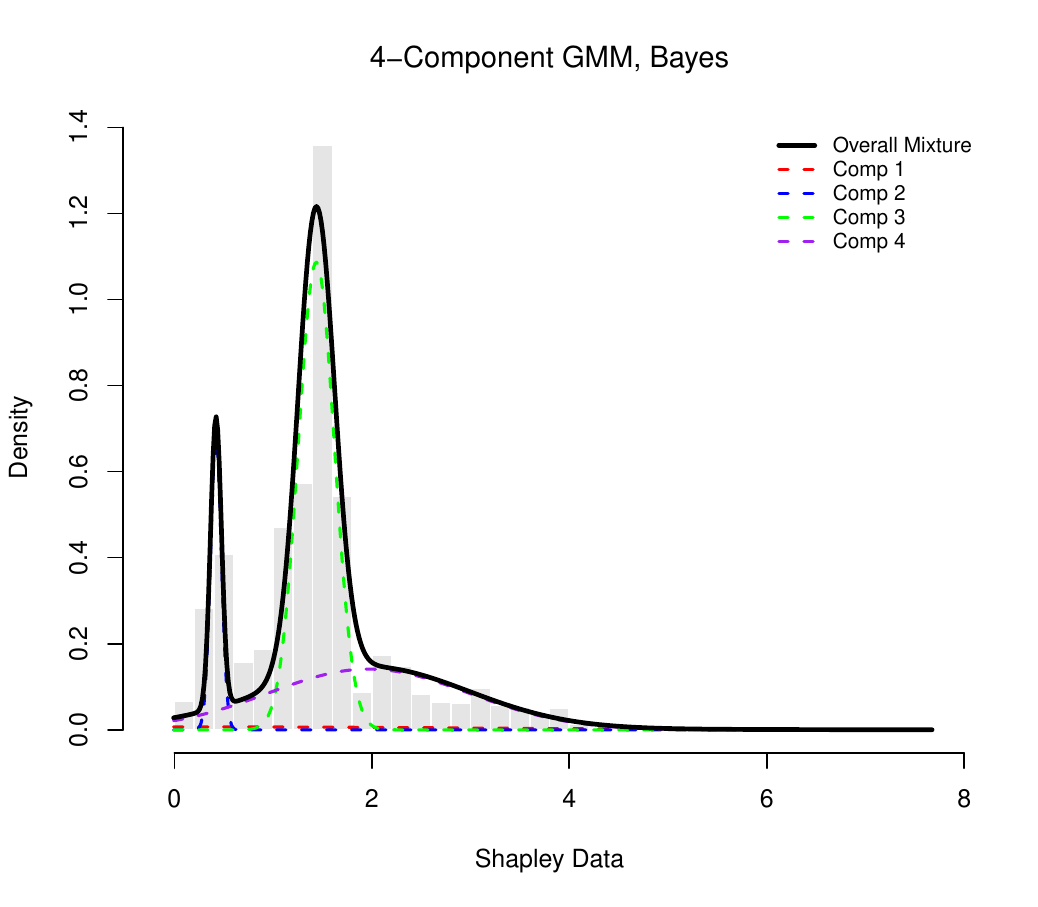}
        \caption{Bayes, 4-Component Mixture.}
        \label{fig:elbow}
    \end{subfigure}
    \hfill
  \begin{subfigure}[b]{0.48\textwidth}
        \centering
        \includegraphics[width=\textwidth]{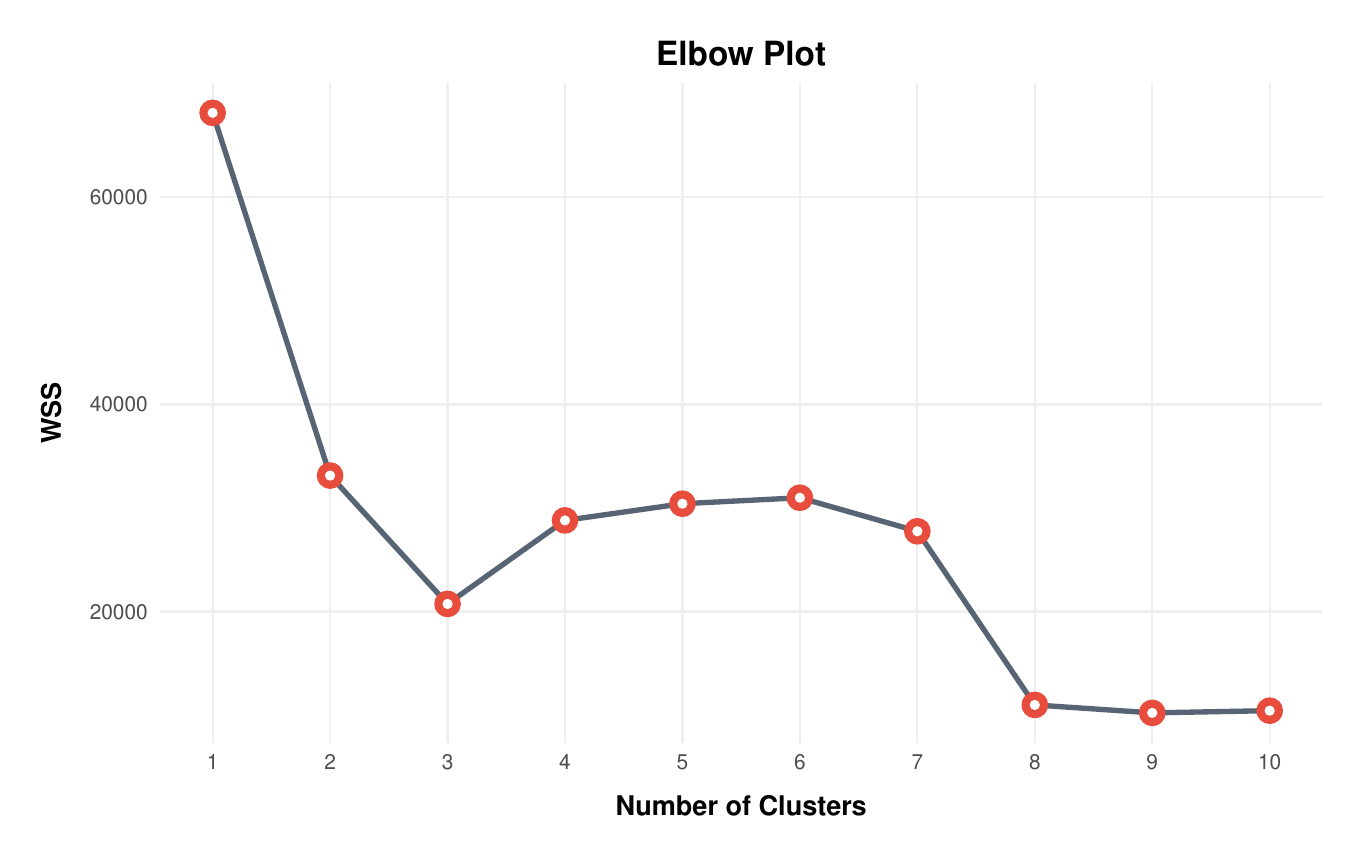}
        \caption{Elbow Plot: within sum-of-squares (WSS) vs number of components {in the Gaussian mixture model}.}
        \label{fig:elbow}
    \end{subfigure}

  \captionsetup{width=1.1\linewidth}  
    \caption{Average component densities across different values of $\beta$. We recall that when $\beta\rightarrow0$, the $\beta$-divergence and standard Bayes coincide. }
    \label{fig:clustering_results}
\end{figure}

\section{Discussion}\label{sec:conclude}

Frequently, a target Gibbs measure $\pi(\theta \mid \MD_n)$ is computationally infeasible due to the intractability of the loss $\MD_n$. {This intractability could be due to the intractability of the statistical model, or to the fact that the loss used cannot be evaluated in closed form under the statistical model. In either case, an obvious approach is to sample from the approximate Gibbs measure $\overline{\pi}(\theta \mid \MD_{m,n})$, based on substituting the intractable loss $\MD_{n}$ for the simulation-based approximation $\MD_{m,n}$. As we have discussed, this substitution effectively amounts to a form of biased pseudo-marginal MCMC (P-MCMC) that is sometimes called noisy Monte Carlo (\citealp{alquier2016noisy}).%; with this approach being used to produce posterior samples across a wide range of problems (see, e.g., Section \ref{subsec:applicability-of-results} for specific examples.)}

While the use of $\overline{\pi}(\theta \mid \MD_{m,n})$ to approximate $\pi(\theta\mid\MD_n)$ is commonplace, our first contribution is to demonstrate, theoretically and empirically, that  large number of simulated datasets within each (MCMC) iteration is required to justify such approximations---inducing a trade-off between computational feasibility and inferential bias.
Our second contribution is the proposal of a novel zig-zag sampler that directly targets $\pi(\theta\mid\MD_n)$ using unbiased gradient estimates of $\MD_n(\theta)$. {We formally show that so long as an unbiased estimate for $\MD_n(\theta)$ exists, then we can indeed produce samples from  $\pi(\theta\mid\MD_n)$ regardless of whether the statistical model or loss is intractable. This paves the way for the use of this methodology to produce exact samples - from the corresponding Gibbs measures - rather than the usual ABC posterior that depends on the choice of tolerance and which is known to be difficult to sample from in many cases. }
This opens the door for more advanced and  efficient sampling schemes for a large class of practically relevant Gibbs measures, and our empirical demonstrations indicate the scale of improvements delivered to generalised Bayesian approaches through this class of sampling schemes.

While the unbiasedness of zig-zag samplers means that it can work out-of-the-box, this approach has additional advantages whose exploration go beyond the scope of this paper.
For example, zig-zag sampling can be easily adapted to target mixtures of continuous and discrete distributions, which frequently arise in variable selection problems and spike-and-slab priors \cite{bierkens2023sticky}.
As a result, the sampler proposed in the current paper is suitable to form the computational foundation of a new generation of generalised Bayesian methodology. For instance, it could be used to combine generalised Bayesian inference for robust regression with variable selection methods, and open up the possibility of new approaches for choosing between different models or Gibbs posteriors.

%\subsection*{Acknowledgments}
{\footnotesize\noindent\textbf{Acknowledgments:}
We are thankful for the many helpful comments from the associate editor and referees, which have greatly improved the paper. % who have for remarks by Dr. Edwin Fong, Dr. Francois-Xavier Briol, Dr. Sam Power, and Dr. Sam Livingstone which greatly improved the quality of this manuscript. 
DTF and CD were supported by Australian Research Council funding schemes DE200101070 and FT210100260, respectively.
JK was supported through the UK's Engineering and Physical Sciences Research Council (EPSRC) via EP/W005859/1 and EP/Y011805/1.
}

{\spacingset{1.0} % DON'T change the spacing!

{\footnotesize
\bibliographystyle{chicago}
\bibliography{Bayes_comp,bib2}
}
}

\newpage
\appendix

\addcontentsline{toc}{section}{Appendix} % Add the appendix text to the document TOC
\part{\LARGE Supplement: Exact Sampling of Gibbs Measures with Estimated Losses} % Start the 
%appendix part

\parttoc % Insert the appendix TOC
%\tableofcontents

\section{Proofs of Main Results}\label{app:results}

For our proofs, we use several notations used sparingly or not at all in the main paper.
First of, we will frequently use the more expressive notation $\MD_{m,n}(\theta,u_{1:m})$, rather than $\MD_{m,n}(\theta)$, to emphasize dependence on the simulated data $u_{1:m} \sim P_\theta$. 
Further, we will usually use the shorthand $u = u_{1:m}$ and write $\MD_{m,n}(\theta,u) = \MD_{m,n}(\theta,u_{1:m})$.
Beyond that, for the limiting loss $\MD: \Theta \to \mathbb{R}$ defined in \Cref{ass:mean_bias}, we additionally define $\MD(\theta, \theta') = \MD(\theta) - \MD(\theta')$
In the remainder we let $p_\theta$ denote the density of the model $P_\theta$. 

\subsection{Lemmas}\label{app:lemmas}
Before proving our main result, we first provide several intermediate results given in the main paper. The first is nothing but an application of Donsker and Varadhan's variational formula that allows us to show that the P-MCMC target $\overline\pi(\theta,u\mid\MD_{m,n})$ is a Gibbs measure. 
While the result is stated for densities $p_{\theta}$ for convenience, it equally applies to general measures $P_{\theta}$.

\begin{lemma}\label{lem:restate}
Let $h(\theta,u)$ be a measurable function such that $\int_\Theta\int_{\mathcal{U}}\pi(\theta)p_\theta(u)e^{h(\theta,u)}\dt\theta\dt u<\infty$. Then, for any $\rho(\theta,u)$ with $\int h(\theta,u)\dt\rho(\theta,u)<\infty$,
$$
\log 
\E_{(\theta,u)\sim\pi\times p_\theta}[e^{h(\theta,u)}]=\sup_{\rho\in\mathcal{P}(\Theta\times \mathcal{U})}\left\{\int_\Theta\int_{\mathcal{U}}{h(\theta,u)}\dt\rho(\theta,u)-\KL\{\rho(\theta,u)\|\pi(\theta)p_\theta(u)\}
\right\}
$$and the supremum is achieved by 
$$
\widehat\rho(\theta,u)=\frac{e^{h(\theta,u)}\pi(\theta)p_\theta(u)}{\int_\Theta\int_{\mathcal{U}}e^{h(\theta,u)}\pi(\theta)p_\theta(u)\dt u\dt \theta}.
$$
\end{lemma}
\begin{proof}[Proof of Lemma \ref{lem:restate}]
The proof of Lemma \ref{lem:restate} actually follows the same argument as in Donsker and Varadhan's variational formula but we state this result in full for ease of reference. For all $\rho\in\mathcal{P}(\Theta\times\mathcal{U})$ and $\widehat{\rho}$ as in the statement of the result,
$$
\KL(\rho\|\widehat{\rho})=-\int h(\theta,u)\dt\rho(\theta,u)+\KL(\rho\|\pi\times p_\theta)+\log \E_{(\theta,u)\sim\pi\times p_\theta}\left[e^{h(\theta,u)}\right].
$$ Since $\KL(P\| Q)\ge0$ for any $P,Q$, the minimum of the right hand side over $\rho$ is achieved when $\KL(\rho\|\widehat{\rho})=0$. 
Setting the right hand side equal to zero and re-arranging terms yields
$$
\log 
\E_{(\theta,u)\sim\pi\times p_\theta}[e^{h(\theta,u)}]=\int h(\theta,u)\dt\rho(\theta,u)-\KL(\rho\|\pi\times p_\theta).
$$Since $\KL(\rho\|\widehat{\rho})=0$ if and only if $\rho=\widehat{\rho}$, it immediately follows that 
$$
\log 
\E_{(\theta,u)\sim\pi\times p_\theta}[e^{h(\theta,u)}]=\sup_{\rho\in\mathcal{P}(\Theta\times \mathcal{U})}\left\{\int{h(\theta,u)}\dt\rho(\theta,u)-\KL(\rho\|\pi\times p_\theta)
\right\}.
$$
\end{proof}
\begin{remark}\label{rem:gibbs}\normalfont
Lemma \ref{rem:gibbs} implies that pseudo-marginal posteriors, such as $\overline\pi(\theta\mid\MD_{m,n})$, can actually be viewed as the marginal of a Gibbs measure on the extended space $(\theta,\U)$. To see this, apply Lemma \ref{lem:restate} to $h(\theta,u)=-\lambda\MD_{m,n}(\theta,u)$, for $\lambda>0$, to see that 
\begin{flalign*}
\widehat{\rho}(\theta,u):=&\argsup_{\rho\in\mathcal{P}(\Theta\times\mathcal{U})}\left\{
-\lambda\int\MD_{m,n}(\theta,u)\dt\rho(\theta,u)-\KL\left\{\rho(\theta,u)\|\pi\times p_\theta
\right\}\right\}\\\equiv &\arginf_{\rho\in\mathcal{P}(\Theta\times\mathcal{U})}\left\{
\lambda\int\MD_{m,n}(\theta,u)\dt\rho(\theta,u)+\KL\left\{\rho(\theta,u)\|\pi\times p_\theta \right\}
\right\}\\=&\frac{e^{-\lambda\MD_{m,n}(\theta,u)}\pi(\theta)p_\theta(u)}{\int_\Theta\int_{\mathcal{U}}e^{-\lambda\MD_{m,n}(\theta,u)}\pi(\theta)p_\theta(u)\dt\theta\dt u}.
\end{flalign*}Thus, $\overline\pi(\theta,u\mid \MD_{m,n})$ is a Gibbs measure over $\Theta\times\mathcal{U}$, while our posterior of interest $\overline\pi(\theta\mid\MD_{m,n})=\int_{\mathcal{U}}\overline\pi(\theta,u\mid \MD_{m,n})\dt u$ is nothing but a \textit{marginalised} version of this Gibbs measure.
\end{remark}

The following result gives simple and interpretable sufficient conditions that guarantee Assumption \ref{ass:tails-new} is satisfied.
\begin{restatable}{lemma}{lemmanew}\label{simple_lemma}
Assume that the following are satisfied: (i) $\E_0\E_{\U \sim P_{\theta}}\left[e^{-\lambda\{\MD_{m,n}(\theta,\U)-\MD_{n}(\theta)\}}\right]\le 
 e^{\lambda^2 \eta_1(\theta)/(2m)} $; (ii) $\E_0\left[e^{-\lambda\{\MD_{n}(\theta)-\MD(\theta)\}}\right] \le 
    e^{\lambda^2 \eta_2(\theta)/(2n)}$. Then, if $\int e^{\frac{\eta_1(\theta)+\eta_2(\theta)}{2}}\dt\pi(\theta)<C$, Assumption \ref{ass:tails-new} is satisfied. 
\end{restatable}

%\lemmanew*
\begin{proof}[Proof of \Cref{simple_lemma}]
Write 
\begin{flalign*}
\lambda[\MD_{m,n}(\theta)-\MD_{}(\theta)]&=
\lambda[\MD_{m,n}(\theta)-\MD_n(\theta)]
+\lambda[\MD_n(\theta)-\MD(\theta)].
\end{flalign*}
Exponentiation of both sides yields
\begin{flalign*}
e^{
-\lambda[\MD_{m,n}(\theta)-\MD_{}(\theta)]}=e^{-\lambda [\MD_{m,n}(\theta)-\MD_n(\theta)]}e^{
-\lambda[\MD_n(\theta)-\MD(\theta)] } 
\end{flalign*}
Now, taking expectations we have 
\begin{flalign*}
\E_0\E_{\U}e^{
-\lambda[\MD_{m,n}(\theta)-\MD_{}(\theta)]}&= \E_0\left\{e^{
-\lambda[\MD_n(\theta)-\MD(\theta)] }\E_{\U} e^{-\lambda [\MD_{m,n}(\theta)-\MD_n(\theta)]}\right\}\\&\le \left\{\E_0 e^{-2\lambda[\MD_n(\theta)-\MD(\theta)] }\right\}^{1/2}\left\{\E_0\E_{\U} e^{-2\lambda [\MD_{m,n}(\theta)-\MD_n(\theta)]}\right\}^{1/2}\\&\le e^{\frac{\lambda^2}{n}\eta_2(\theta)}e^{\frac{\lambda^2}{m}\eta_1(\theta)}.
\end{flalign*}
The second line follows from applying Cauchy-Schwartz and Jensen's inequality, and the third from applying the maintained assumptions. Assumption 4 is then satisfied with $g(\lambda,n,m)\asymp \lambda^2/\min\{m,n\}$.
\end{proof}

\lemmabeta*
\begin{proof}[Proof of Lemma \ref{lem:beta_div}]
Since, by construction, $\U$ is sampled independently from $P_\theta$, and $y$ from $P_0$,
$$
\E_0\E_{u_{1:m} \sim p_\theta}\MD_{m,n}^\beta(\theta)=\int p_\theta(u)^{1+\beta}\dt u-\left(1+\frac{1}{\beta}\right)\E_0[p_\theta(Y)^\beta]=\MD(\theta).
$$Hence, Assumption \ref{ass:mean_bias} part (i) is satisfied.

Define $A_m(\theta)=\frac{1}{m}\sum_{i=1}^{n}p_\theta(u_i)^\beta$, $A(\theta)=\int p_\theta(z)^{1+\beta}\dt z$, $B_n(\theta)=(1+\beta^{-1})\frac{1}{n}\sum_{i=1}^{n}p_\theta(y_i)^\beta$, and $B(\theta)=(1+\beta^{-1})\E_0 [p_\theta(y)^\beta]$. Note that, since $\sup_{y}p_\theta(y)\le M_\theta$,  $$0\le A(\theta)=\int p_\theta(z)^{1+\beta}\dt z=\int p_\theta(z)^\beta p_\theta(z)\dt z\le M_\theta^\beta,$$ and $A(\theta)$ is upper bounded by $M_\theta^\beta$ for all $\theta\in\Theta$. 

We wish to bound 
\begin{flalign*}
\E e^{-\lambda\left[\MD_{m,n}(\theta)-\MD(\theta)\right]}&=\E e^{-\lambda\left[A_m(\theta)-B_n(\theta)-\{A(\theta)-B(\theta)\}\right]}
=\E_0\E_{u_{1:m}} e^{-\lambda\left[A_m(\theta)-A(\theta)\right]}e^{\lambda\{B_n(\theta)-B(\theta)\}}.
\end{flalign*}

Let $A_i(\theta)=p_\theta(u_i)^\beta$, and note that $0\le A_i(\theta)\le M_\theta^\beta$, so that 
$$
 -A(\theta)\le A_i(\theta)-A(\theta)\le M_\theta^\beta-A(\theta).
$$ Therefore, for fixed $\theta\in\Theta$, $A_i(\theta)-A(\theta)$ is bounded and since $A_i(\theta)$ is independent of $A_j(\theta)$ for all $i\ne j$, by Hoeffding's Lemma: 
$
\E_{\U}e^{-\lambda[A_m(\theta)-A(\theta)]}\le e^{\frac{\lambda^2}{m}M_\theta^{2\beta}}. 
$
%Since $A_i$ is independent of $A_j$ for all $i\ne j$, we can applying Hoeffdings inequality to obtain
%$$
%\E_{u_{1:m}} e^{-\lambda\left[A_m(\theta)-A(\theta)\right]}\le e^{\frac{\lambda^2M_\theta^{2\beta}}{8m}}
%$$
Similarly, for $B_i(\theta)=(1+\beta^{-1})f_\theta(y_i)^\beta$, since $0\le B_i(\theta)\le (1+\beta^{-1})M_\theta^\beta$, 
$$
-B(\theta)\le B_i(\theta)-B(\theta)\le (1+\beta^{-1})M_\theta^\beta-B(\theta)
$$ and we can apply Hoeffding to obtain 
%\begin{flalign*}
$\E_0e^{\lambda\{B_n(\theta)-B(\theta)\}}\le e^{\frac{(1+\beta^{-1})^2\lambda^2}{n}M_\theta^{2\beta}}.$
%\end{flalign*}

We then arrive at the bound 
$$
\E \int e^{-\lambda\left[\MD_{m,n}(\theta)-\MD(\theta)\right]}\pi(\theta)\dt\theta\le \int e^{\frac{\lambda^2}{\min\{n,m\}}M_\theta^{2\beta}[1+(1+\beta^{-1})^2]}\pi(\theta)\dt\theta,
$$ which takes the same form as in Assumption \ref{ass:tails-new} with $g(\lambda,m,n)=\lambda^2/\min\{m,n\}$.
\end{proof}

\lemmakenel*

\begin{proof}[Proof of Lemma \ref{lem:kernel}]
Define $C_n=\frac{1}{n^2}\sum_{i=1}^n\sum_{i' = 1}^n
k(y_i, y_{i'})$, take expectation of $\MD_{m,n}^k(\theta)$ wrt $U_{1:m}$ in \eqref{eq:mmd-loss-mn}, and use the fact that $U_i$ is simulated iid from $P_\theta$:
\begin{flalign*}
\E_{U_{1:m}\sim P_\theta}\MD_{m,n}^k(\theta,U_{1:m})&=C_n-\frac{2}{nm}\sum_{i=1}^{n}\sum_{j=1}^{m}\E_{U_j\sim P_\theta}k(U_j,y_i)+\frac{1}{m^2}\sum_{i=1}^{m}\sum_{j=1}^{m}\E_{U_i,U_j\sim P_\theta}[k(U_j,U_j^{'})]\\&=C_n-\frac{2}{n}\sum_{i=1}^{n}\E_{U\sim P_\theta}[k(U,y_i)]+\frac{m-1}{m}\E_{U,U'\sim P_\theta}[k(U,U^{'})]\\&\le \MD_n^k(\theta)+\frac{1}{m}\mathsf{K}
\end{flalign*}verifying Assumption \ref{ass:mean_bias}. 

To deduce Assumption \ref{ass:tails-new}, note that $\MD_{m,n}^k(\theta)$ can be written as the sum of bounded random variables whenever $y_i\stackrel{iid}{\sim}P_0$.
As a result, $\MD_{m,n}^k(\theta)$ is the sum of sub-Gaussian random variables, and satisfies Assumption \ref{ass:tails-new} with $g(\lambda,m,n)\asymp \lambda^2/m+\lambda^2/n$. 
\end{proof}

The following result gives conditions on the energy score which ensure that Assumptions \ref{ass:mean_bias} and \ref{ass:tails-new} are satisfied for this loss. 

\begin{lemma}\label{lem:energy}
Take $k(u,u')=\|u-u'\|^\alpha$, with $\alpha\in(0,2)$. Assumption \ref{ass:mean_bias} is satisfied with $\kappa=1$. Assume that one of the following is satisfied: (i) $U=(U_{1},\dots,U_d)^\top\sim P_\theta$ and $Y=(Y_1,\dots,Y_d)^\top\sim P_0$ are bounded, and $\alpha\in(0,2)$; (ii) $U=(U_{1},\dots,U_d)^\top\sim P_\theta$ has sub-Gaussian components, and $Y=(Y_1,\dots,Y_d)^\top\sim P_0$ is also sub-Gaussian, and $\alpha=1$. Under either condition, Assumption \ref{ass:tails-new} is satisfied with $g(m,n,\lambda)\asymp\lambda^2/\min\{m,n\}$ for the loss in \eqref{eq:mmd-loss-mn}.
\end{lemma}
\begin{proof}
Verification of Assumption \ref{ass:mean_bias} follows similarly to the proof of Lemma \ref{lem:kernel}; the only change is that, since the kernel is unbounded, we obtain 
\begin{flalign*}
\E_{U_{1:m}\sim P_\theta}\MD_{m,n}^k(\theta,\U)&\le \MD_n^k(\theta)+\frac{m-1}{m}\E_{U,U'\sim P_\theta}[k(U,U^{'})].
\end{flalign*} Recall that $k(U,U')=\|U-U'\|^\alpha$. Under either (i) or (ii), $U,U'\sim P_\theta$, are sub-Gaussian, with some sub-Gaussian parameter, say $\sigma(\theta)^2$. Since $U,U'$ have the same distribution, $Z=U-U'$ is a zero-mean sub-Gaussian random variable with sub-Gaussian parameter $2\sigma(\theta)^2$. Hence, 
$$
\E_{U,U'\sim P_\theta}[k(U,U^{'})]=\E_{U,U'}[\|U-U'\|^\alpha]\le \E_{U,U'}[\|U-U'\|]\le C\sqrt{d\sigma^2(\theta)}, $$which verifies Assumption \ref{ass:mean_bias}. 

Under condition (i) maintained in the lemma, the result follows exactly the same as for the MMD in Lemma \ref{lem:kernel}. Given this, we now consider only condition (ii) throughout the remainder.

To verify Assumption \ref{ass:tails-new} under condition (ii), we first show that the energy score is 1-Lipshitz. For a given distribution $P$ and $u$ fixed, the energy distance can be equivalently defined in terms of the scoring rule
$$
S(P,u)=\E_{X,X'\sim P}[\|X-X'\|]-2\E_{X\sim P}[\|X-u\|].
$$For fixed $u,u'\in\mathcal{Y}$, 
\begin{flalign}
|S(P,u)-S(P,u')| &\le 2|\E_{X\sim P}[\|X-u\|]-\E_{X\sim P}[\|X-u'\|]|\nonumber\\&\le 2\E_{X\sim P}|\|X-u\|-\|X-u'\||\nonumber \\&\le 2\|u-u'\|.\label{eq:lipz}
\end{flalign}The first inequality comes from the definition of the score, the second from Jensen's inequality, the third from the reverse triangle inequality.

Now, for $U_i$, $i=1,\dots,d$, sub-Gaussian with sub-Gaussian parameter $\sigma_i^2(\theta)$, we know that, for some $C>0$,
$
\Pr(|U_i|>\varepsilon)\le Ce^{-\varepsilon^2/\sigma_i^2(\theta)}.
$ For any $\varepsilon>0$, 
$
\{\|U\|\ge t\}\subset \bigcup_{i=1}^{d}\{|U_i|>\varepsilon/\sqrt{d}\},
$ which follows since $|U_i|<\varepsilon/\sqrt{d}$ for each $i=1,\dots,d$, then $\|Y\|<\sqrt{d}(\varepsilon/\sqrt{d})=\varepsilon$. Using a union bound, we then have that 
$$
\Pr(\|U\|>\varepsilon)\le \Pr\left(\cup_{i=1}^{d}\{|U_i|>\varepsilon/\sqrt{d}\}\right)\le \sum_{i=1}^{d}\Pr\left\{|U_i|>\varepsilon/\sqrt{d}\right\}\le dCe^{-\varepsilon^2/\{d\sigma^2(\theta)\}},
$$
where $\sigma^2(\theta)=\max_{i}\sigma^2_i(\theta)$. 

Now, for any $U\sim P_\theta$, we use the Lipshitz behavior $u\mapsto S(P,u)$ to deduce that 
\begin{flalign*}
\Pr\left\{|S(P,U)-S(P,0)|>\varepsilon\right\}\le    \Pr\left\{|U|>\varepsilon/2\right\}\le  2dCe^{-\varepsilon^2/4\{d\sigma^2(\theta)\}}.
\end{flalign*}Hence, $S(P_\theta,U)-S(P_\theta,0)$ is sub-Gaussian with parameters $A=2dC$ and $B(\theta)=1/\{d\sigma^2(\theta)\}$. 

We now apply the Lipshitz behavior of $u\mapsto S(P,u)$ and the above bound to prove that $S(P,U)-\E_{U\sim P_\theta}S(P,U)$ is sub-Gaussian. In particular, 
\begin{flalign*}
|S(P,U)-\E_{U\sim P_\theta}S^{}(P,U)|&\le |S^{}(P,U)-S^{}(P,0)|+|S^{}(P,0)-\E_{U\sim P_\theta}S^{}(P,U)|\\&\le 2\|U\|+c(\theta),
\end{flalign*}where $c(\theta)=|S^{}(P,0)-\E_{U\sim P_\theta}S^{}(P,U)|$, and $c(\theta)$ is finite for fixed $\theta\in\Theta$ since $U\sim P_\theta$ is sub-Gaussian by hypothesis. Therefore, for any $\varepsilon\ge 2c(\theta)$
\begin{flalign*}
\Pr\left\{|S^{}(P,U)-\E_{U\sim P_\theta}S^{}(P,U)|>\varepsilon\right\}&= \Pr\left\{|S^{}(P,U)-\E_{U\sim P_\theta}S^{}(P,U)|-c(\theta)>\varepsilon-c(\theta)\right\}\\&\le \Pr\left\{\|U\|> \frac{\varepsilon-c(\theta)}{2}\right\}\\&\le 2de^{-\{\varepsilon-c(\theta)\}^2/4(d\sigma^2(\theta)}.
\end{flalign*}However, when $\varepsilon<2c(\theta)$, we obtain the trivial bound 
$$
\Pr\left\{|S^{}(P,U)-\E_{U\sim P_\theta}S^{}(P,U)|>\varepsilon\right\}\le 1
$$For all $\varepsilon>0$ we must now find constants $A(\theta)$, $B(\theta)$ such that 
$$
\Pr\left\{|S^{}(P,U)-\E_{U\sim P_\theta}S^{}(P,U)|>\varepsilon\right\}\le A(\theta)e^{-t^2B(\theta)}.
$$Going through the cases, one can show that this is satisfied by taking
$$
B(\theta)=\frac{1}{4d\sigma^2(\theta)},\quad A(\theta)=\max\{2d,e^{B(\theta)4c(\theta)^2}\}.
$$Hence, $S^{}(P_\theta,U)-\E_{U\sim P_\theta}S^{}(P_\theta,U)$ is sub-Gaussian with parameters $A(\theta)$ and $B(\theta)$.

We now transfer this  result to $\MD_{m,n}(\theta)$. To do so, we make repeated use of the following three facts. 
\begin{enumerate}
    \item Each component of $U$ being sub-Gaussian, and since each component of $Y\sim P_0$ is also sub-Gaussian. 
    \item The kernel $k(u,y)=\|u-y\|$ is 1-Lipschitz.
    \item If $U\stackrel{iid}{\sim}P_\theta$, with $U$ is sub-Guassian with variance bound $\sigma^2(\theta)$, then $\|U\|$ is itself sub‑Gaussian, with variance bound $\sqrt{d\sigma^2(\theta)}$ in this case.
\end{enumerate}

We rewrite the deviation as
$$\MD_{m,n}(\theta)-\MD(\theta) \;=\; \Delta_1 - 2\Delta_2 - 2\Delta_3,
$$
where, for $U\sim P_\theta$,
\begin{flalign*} \Delta_1 &= \frac{1}{m^2}\sum_{j,j'} \|u_j-u_{j'}\| \;-\; \E_{U,U'}\|U-U'\|,\\[2mm] \Delta_2 &= \frac{1}{nm}\sum_{i=1,j=1}^{n,m}\Bigl(\|u_j-y_i\| - \E_U\|U-y_i\|\Bigr),\\[2mm] \Delta_3 &= \frac{1}{n}\sum_{i=1}^n\Bigl(\E_U\|U-y_i\| - \E_{Y}\E_U\|U-Y\|\Bigr). \end{flalign*}
The proof now proceeds by showing that each of the above terms is sub-Gaussian. \\

\noindent(1) \textbf{$\Delta_3$ term.} The function $\phi(y)=\E_U\|U-y\|$ is 1-Lipschitz in $y$ for any fixed $u$. 
Recall that  $Y$ has independent sub-Gaussian components, so that $\E_U\|Y-U\|$ is sub-Gaussian in $Y$ with parameter $C_1 \sqrt{d} \sigma_Y$. The observations $y_1, \ldots, y_n$ are i.i.d., so $\Delta_3$ is an average of i.i.d. centred sub-Gaussian random variables. Hence, for some constant $c>0$,
$$
\E\Bigl[e^{\lambda \Delta_3}\Bigr] \le \exp\Bigl(\frac{c\,\lambda^2}{n}\,\bigl(d\,\sigma_Y^2\bigr)\Bigr), \qquad \forall\lambda\in\mathbb{R}.
$$Hence, $\Delta_3$ is sub-gaussian.

\noindent(2) \textbf{$\Delta_3$ term.} Since all $i=j$ terms are zero, we can rewrite the $\Delta_1$ term as 
$$
\frac{2}{m^2}\sum_{1\le j<i\le m}\|u_i-u_j\|-\E_{U,U'}\|U-U'\|.
$$
However, For any fixed $P_\theta$, the mapping $(u,u')\mapsto \phi(u,u'):= \|u-u'\|$ is also sub-Gaussian with variance bound $2\sqrt{d\sigma^2(\theta)}$. Hence, $0\le \E_{U,U'}\|U-U'\|\le 2\sqrt{d\sigma^2(\theta)}$. Since $U_i,U_j$ are independent and sub-Gaussian, we know that $\phi(U_i,U_j)-\E\phi(U_i,U_j)$ is a mean-zero sub-Gaussian random variable with variance bound $2\sqrt{d\sigma^2(\theta)}$. Let $h(u_i,u_j)=\phi(u_i,u_j)-\E\phi(u_i,u_j)$,  
$$
\Delta_1=\frac{2}{m^2}\sum_{1\le j< i\le m}h(u_i,u_j)=2\frac{\binom{m}{2}}{m^2}\frac{1}{\binom{m}{2}}\sum_{1\le j< i\le m}h(u_i,u_j)=(1-1/m)U_m.
$$ where $U_m$ is a zero-mean U-statistic.  %from which we see that $U_m$ is a constant multiple of a U-statistic with constant multiplier: $m(m-1)/m^2=(1-1/m)$ 
%and consider that our goal is to show that $\E[e^{\lambda U_n}]$ satisfies a sub-Gaussian type of bound. 

%\begin{flalign*}
%\E[e^{\lambda U_n}]=\E\prod_{1\le j<i\le m}\exp\left[\frac{\lambda}{m^2}\left\{\phi(U_i,U_j)-\E\phi(U_i,U_j)\right\}\right]  %  
%\end{flalign*}

%For any $i$ and $j$ pair, $\phi(U_i,U_j)$ is sub-Gaussian with mean $\E\phi(U_i,U_j)$, however, this does not deal with the dependence in the cross $(i,j)$ pairs. 
%%$$
%%\exp\left[\frac{\lambda}{m^2}\left\{\phi(U_i,U_j)-\E\phi(U_i,U_j)\right\}\right]\le e^{\frac{\lambda^2}{m^4}\sqrt{d\sigma^2(\theta)}}
%%However, this is not satisfied for any of the cross pairs. 
%To deal with this dependence, we use a decoupling approach. 

Take $k=m/2$, and without loss of generality, let $m$ be even; else define $k=\left\lceil\frac{m}{2}\right\rceil$. Set
$$
W\left(u_1, \ldots, u_m\right)=\frac{h\left(u_1, u_2\right)+h\left(u_3, u_4\right)+\ldots+h\left(u_{2k-1)}, \ldots, u_{2 k}\right)}{k}.
$$
Let $\mathcal{S}^m$ be the permutation group on $\{1, \ldots, m\}$. Then
$$
k \sum_{\sigma \in \mathcal{S}^m} W\left(U_{\sigma_1}, \ldots, X_{\sigma_n}\right)=2(n-2)!\sum_{1\le i_1<i_2\le n} h\left(U_{i_1}, U_{i_2}\right)
$$
which implies that
$$
\sum_{\sigma \in S^n} W\left(X_{\sigma_1}, \ldots, U_{\sigma_m}\right)=\underbrace{2!(n-2)!\binom{m}{2}}_{m!} U_m
$$
and so 
\begin{equation}\label{eq:u-stat-rep}
U_m=\frac{1}{m!} \sum_{\sigma \in \mathcal{S}^m} W\left(X_{\sigma_1}, \ldots, X_{\sigma_m}\right) .
\end{equation}

For a fixed index $\sigma$, $W(U_{\sigma_1},\dots,U_{\sigma_m})$ is the sum of $k$ iid mean-zero sub-Gaussian random variables, so that for any $k$: 
\begin{flalign*}
\E\exp\left[{\lambda}W(U_{\sigma_1},\dots,U_{\sigma_m})\right]\le \E\prod_{i=1}^{k}e^{\frac{\lambda}{k}h(U_{\sigma_{2k-1}},U_{\sigma_{2k}})}&= \prod_{i=1}^{k}\E e^{\frac{\lambda}{k}h(U_{\sigma_{2k-1}},U_{\sigma_{2k}})}\\&\le e^{\frac{\lambda^2}{2k}{d\sigma^2(\theta)}} 
\end{flalign*}
Now, let $T_\sigma=W(U_{\sigma_1},\dots,U_{\sigma_m})$, use equation \eqref{eq:u-stat-rep}, the above, and Jensen's inequality to obtain 
\begin{flalign*}
 \E e^{{\lambda}U_m}&= \E e^{{\lambda}\frac{1}{m!}\sum_{\sigma}T_\sigma}\le \frac{1}{m!}\sum_{\sigma\in\mathcal{S}^m}  e^{{\lambda}T_\sigma}\le e^{\frac{\lambda^2}{2k}{d\sigma^2(\theta)}}\asymp e^{\frac{\lambda^2}{m}{d\sigma^2(\theta)}} .
\end{flalign*}

\noindent(2) \textbf{$\Delta_2$ term.} Define the map $\phi(u_j,y_i)=\|u_j-y_i\|-\E_U\|U-y_i\|$. From the reverse triangle inequality, 
$$
|\phi(u_j,y_i)-\bar\phi(0,y_i)|\le |\|u_j-y_i\|-\|y_i\||\le \|u_j\|
$$Define the mapping $\bar\phi(u_j)=\frac{1}{n}\sum_i\phi(u_j,y_i)$, and use the above to see that
$$
\|\bar\phi(u_j)-\bar\phi(0)\|\le \|u_j\|. 
$$Conditional on $Y_{1:n}$, $\bar\phi(U_j)$ is an iid zero-mean sub-Gaussian random variable that is bounded by $\|U_j\|$. Hence, 
\begin{flalign*}
\E e^{\lambda\Delta_2}&=\E e^{\frac{1}{m}\sum_{j=1}^{m}\phi_m(U_j)}\le \E e^{\frac{1}{m}\sum_{j=1}^{m}\|U_j\|}\le e^{\frac{\lambda^2}{m}\sqrt{d\sigma^2(\theta)}}  .
\end{flalign*}

Placing the three terms together then delivers the stated result. 

\end{proof}

\subsection{Main Results}\label{sec:main_proofs}

\lemtvbound*
\begin{proof}[Proof of \Cref{lem:tvbound}]
The result uses the following refinement of Jensen's inequality due to \cite{liao2018sharpening}: for $\varphi(\cdot)$ a convex function, and $X_\theta\in\mathcal{X}$ a random variable with $\sigma^2_\theta=\text{Var}(X_\theta)<\infty$, 
\begin{flalign*}
\varphi(\E_{0}[X_\theta])+\frac{\sigma_\theta^2}{2}\inf_{x\in\mathcal{X}}\frac{\partial^2 \varphi(x)}{\partial x^2}\le \E_{0}\varphi(X_\theta)\le \frac{\sigma^2_\theta}{2}\sup_{x\in\mathcal{X}}\frac{\partial^2 \varphi(x)}{\partial x^2}+\varphi(\E_{0}[X_\theta]).
\end{flalign*}

Take $X_\theta=\MD_{m,n}(\theta)$, which satisfies $X_\theta\ge0$ by \Cref{ass:positive}, $\varphi(X_\theta)=e^{-X_\theta}$ and apply the above to obtain 
\begin{flalign*}
e^{-\E_{0}[\MD_{m,n}(\theta)]}+\frac{c_2(\theta)}{2m}\inf_{x\ge0}e^{-x}   \le \E_{0} e^{-\MD_{m,n}(\theta)}&\le e^{-\E_{0}[\MD_{m,n}(\theta)]}+\frac{c_2(\theta)}{2m}\sup_{x\ge0}e^{-x},%\\&\ge e^{-\E_{0}[X_\theta]}+\frac{\sigma^2_\theta}{2}\inf_{x_\theta\in[0,\infty)}e^{-x}   
\end{flalign*}
% Note, this bound also works for bounded losses like the MMD since the lower bound will always be less than when the loss takes at most a value of $b$<\infty$. That is, 
% we have that 
%$$
%e^{-\E_{0}[X_\theta]}+\frac{\sigma^2_\theta}{2}\inf_{x_\theta\in[0,K]}e^{-x}   \ge e^{-\E_{0}[X_\theta]}+\frac{\sigma^2_\theta}{2}\inf_{x_\theta\in[0,\infty)}e^{-x}   
%$$
which simplifies to 
\begin{flalign*}
e^{-\E_{0}[\MD_{m,n}(\theta)]}  \le \E_{0} e^{-\MD_{m,n}(\theta)}&\le e^{-\E_{0}[\MD_{m,n}(\theta)]}+\frac{c_2(\theta)}{2m}.%\label{eq:jensen1}%\\&\ge e^{-\E_{0}[X_\theta]}+\frac{\sigma^2_\theta}{2}\inf_{x_\theta\in[0,\infty)}e^{-x}   
\end{flalign*}Applying the stated assumptions on the mean and variance into the above equation, and multiplying by $\pi(\theta)$ yields
\begin{flalign*}
e^{-\MD_n(\theta)}\pi(\theta)e^{-c_1(\theta)/m^\kappa}  \le \E_{0} [e^{-\MD_{m,n}(\theta)}]\pi(\theta)&\le e^{-\MD_n(\theta)}\pi(\theta)e^{-c_1(\theta)/m^\kappa}+\pi(\theta)\frac{c_2(\theta)}{2m}.%\label{eq:jensen2}%\\&\ge e^{-\E_{0}[X_\theta]}+\frac{\sigma^2_\theta}{2}\inf_{x_\theta\in[0,\infty)}e^{-x}   
\end{flalign*}Now, since $c_1(\theta)\ge0$ by hypothesis, use the fact that, for $x\ge0$, $1-x\le e^{-x}\le 1-x+{x^2}$, to obtain
\begin{flalign}
e^{-\MD_n(\theta)}\pi(\theta)\left(1-\frac{c_1(\theta)}{m^\kappa}\right)  &\le \E_{0} [e^{-\MD_{m,n}(\theta)}]\pi(\theta)\nonumber\\&\le e^{-\MD_n(\theta)}\pi(\theta)\left(1-\frac{c_1(\theta)}{m^\kappa}+\frac{c_1^2(\theta)}{m^{2\kappa}}\right)+\pi(\theta)\frac{c_2(\theta)}{2m}.\label{eq:jensen2}%\\&\ge e^{-\E_{0}[X_\theta]}+\frac{\sigma^2_\theta}{2}\inf_{x_\theta\in[0,\infty)}e^{-x}   
\end{flalign}
To simplify the remainder of the derivation, define
\begin{flalign*}
f(\theta)=e^{-\MD_n(\theta)}\pi(\theta),\quad f_m(\theta)= \E_{0} e^{-\MD_{m,n}(\theta)}\pi(\theta),\quad Z:=\int f(\theta)\dt\theta,\quad Z_m:=\int f_m(\theta)\dt\theta.    
\end{flalign*}From which we can restate \eqref{eq:jensen2} as
\begin{flalign}
f(\theta)\left(1-\frac{c_1(\theta)}{m^\kappa}
\right)&\le f_m(\theta)\le f(\theta)\left(1-\frac{c_1(\theta)}{m^\kappa}+\frac{c_1^2(\theta)}{m^{2\kappa}}\right)+\pi(\theta)\frac{c_2(\theta)}{2m}.\label{eq:jensen3}
\end{flalign}
Equation \eqref{eq:jensen3} delivers lower and upper bounds that we can use to derive the stated result. 
Bound the posterior difference $|\overline\pi(\theta\mid\MD_{m,n})-\pi(\theta\mid\MD_{n})|$ as
\begin{flalign*}
|\overline\pi(\theta\mid\MD_{m,n})-\pi(\theta\mid\MD_{n})|\le\frac{Z|f_m(\theta)-f(\theta)|}{Z_mZ}+\frac{f(\theta)}{Z_mZ}|Z-Z_m|,
\end{flalign*}which implies that 
\begin{flalign}
\int |\overline\pi(\theta\mid\MD_{m,n})-\pi(\theta\mid\MD_{n})|\dt\theta\le\frac{\int|f_m(\theta)-f(\theta)|\dt\theta}{Z_m}+\frac{|Z-Z_m|}{Z_m}\label{eq:jensen4}.
\end{flalign}We now bound these terms individually.
\\

\noindent\textbf{Bounding $Z_m$.} 
From the lower bound in equation \eqref{eq:jensen3}, we see that 
\begin{flalign*}
Z_m=\int f_m(\theta)\dt\theta&\ge Z\left(1-\frac{\int c_1(\theta)\pi(\theta\mid\MD_n)\dt\theta}{m^\kappa}\right)\\&\le Z\left(1-\frac{\int c_1(\theta)\pi(\theta\mid\MD_n)\dt\theta}{m^\kappa}+ \frac{\int c_1^2(\theta)\pi(\theta\mid\MD_n)\dt\theta}{m^{2\kappa}}\right)+\frac{\int c_2(\theta)\pi(\theta)\dt\theta}{2m}
\end{flalign*}
From the lower bound of the above, and Assumption \ref{ass:positive}(iv), there exists an  $m'$ large enough such that for all $m\ge m'$, $0\le m^{-\kappa}\int c_1(\theta)\pi(\theta\mid\MD_n)\dt\theta\le 1/2$. Then, for all $m\ge m'$,
\begin{equation}
Z_m\ge Z/2.\label{eq:LB1}
\end{equation}

Moreover, from the lower and upper bound on $Z_m$, we also know that
\begin{flalign*}
Z_m-Z&\ge -Z\frac{\int c_1(\theta)\pi(\theta\mid\MD_n)\dt\theta}{m^\kappa}\\&\le    Z\left\{ \frac{\int c_1^2(\theta)\pi(\theta\mid\MD_n)\dt\theta}{m^{2\kappa}}-\frac{\int c_1(\theta)\pi(\theta\mid\MD_n)\dt\theta}{m^\kappa}\right\}+\frac{\int c_2(\theta)\pi(\theta)\dt\theta}{2m} . 
\end{flalign*}
We then see that, since $c_1(\theta),c_2(\theta)\ge0$, 
\begin{equation}\label{eq:LB2}
|Z_m-Z|\le Z\left\{ \frac{\int c_1(\theta)\pi(\theta\mid\MD_n)\dt\theta}{m^\kappa} +\frac{\int c^2_1(\theta)\pi(\theta\mid\MD_n)\dt\theta}{m^{2\kappa}}\right\}+\frac{\int c_2(\theta)\pi(\theta)\dt\theta}{2m}.
\end{equation}
Now, since $0<Z<\infty$ for $n$ fixed, under Assumption \ref{ass:positive}(iv) equation \eqref{eq:LB2} implies that $|Z_m-Z|=O(1/\min\{m,m^\kappa\})$. 
\\

\noindent\textbf{Bounding $|f_m(\theta)-f(\theta)|$.} Returning to equation \eqref{eq:jensen3} we have, since $c_1(\theta)\ge0$ and $\kappa>0$, 
$$
|f_m(\theta)-f(\theta)|\le f(\theta)\left\{\frac{c_1(\theta)}{m^\kappa}+\frac{c_1^2(\theta)}{m^{2\kappa}}\right\}+\frac{c_2(\theta)\pi(\theta)}{2m}\le Z\pi(\theta\mid\MD_n)\frac{c_1(\theta)+c_1^2(\theta)}{m^\kappa}+\frac{c_2(\theta)\pi(\theta)}{2m}.
$$Integrating the above then yields
\begin{flalign}
\int |f_m(\theta)-f(\theta)|\dt\theta&\le \frac{Z}{m^\kappa}\left[\int\{c_1(\theta)+c_1^2(\theta)\}\pi(\theta\mid\MD_n)\dt\theta\right]+\frac{\int c_2(\theta)\pi(\theta)\dt\theta}{2m}\nonumber\\&=O(1/\min\{m,m^\kappa\}),\label{eq:UB1}
\end{flalign}
where the equality follows since $0<Z<\infty$ and Assumption \ref{ass:positive}(iv). \\

\noindent\textbf{Bringing terms together.} Using equations \eqref{eq:LB1}-\eqref{eq:LB2}, we automatically have, for $m\ge m'$, 
\begin{flalign*}
\int |\overline\pi(\theta\mid\MD_{m,n})-\pi(\theta\mid\MD_{n})|\dt\theta\le\frac{1}{Z}\left\{2\int|f_m(\theta)-f(\theta)|\dt\theta+{O(1/\min\{m,m^\kappa\})}\right\}. 
\end{flalign*}Since $Z>0$, the result then follows if $\int |f_m(\theta)-f(\theta)|\dt\theta=O(1/\min\{m,m^\kappa\})$, which is shown in equation \eqref{eq:UB1}.
\end{proof}

To prove \Cref{cor:beta_div}, we make specific use of the structure in the $\beta$-divergence loss by writing 
$$
\MD_n(\theta)=B(\theta)+A_n(\theta)=\int p_\theta^{1+\beta}(u)\dt u-\left(1+\frac{1}{\beta}\right)\frac{1}{n}\sum_{i=1}^{n}p_\theta(y_i).
$$ Since only $B(\theta)$ is intractable, the loss can be estimated by drawing $u_{1:m}\stackrel{iid}{\sim}p_\theta$ and forming
$$
\MD_{m,n}(\theta)=B_m(\theta)+A_n(\theta)=\frac{1}{m}\sum_{j=1}p_\theta^{\beta}(u_j)+A_n(\theta). 
$$ Since only the first term in $\MD_{m,n}(\theta)$ is estimated, we only require conditions similar to those in \Cref{ass:positive} for $B_m(\theta)$, which motivates the following assumption that can be applied to losses that are similar to the $\beta$-divergence. 
\begin{assumption}\label{ass:beta_positive}
For fixed $n$, the following are satisfied. 
\newline\noindent(i) $B_m(\theta)\ge0$ and $\E_{u_{1:m}}[ B_{m}(\theta)]=B(\theta)$ for all $\theta\in\Theta$.
\newline\noindent(ii) There exists $c_2:\Theta\rightarrow\mathbb{R}_+$ such that $\Var_{\U}\{B_m(\theta)\}=c_2(\theta)/m$. 
\newline\noindent(iii) For $c_2(\theta)=\int c_2^2(\theta)\pi(\theta\mid \MD_n)\dt\theta\le C<\infty$ and $\int c_2(\theta)e^{-A_n(\theta)}\pi(\theta)\dt\theta\le C<\infty$.
\end{assumption}
The standard $\beta$-divergence satisfies \Cref{ass:beta_positive}(i)-(ii) automatically: $B_m(\theta)$ is positive by construction, and unbiased for all $\theta\in\Theta$; and $c_2(\theta)=\Var\{p_\theta^\beta (u)\}/m$. Assumption \Cref{ass:beta_positive}(iii) will be satisfied so long as the prior does not place mass on elements in $\theta$ that cause $p_\theta$ to exhibit polls, which ultimately must be checked on a case-by-case basis.  
\cortvbound*
\begin{proof}[Proof of \Cref{cor:beta_div}]
The proof follows similar arguments to those used to prove \Cref{lem:tvbound} but for the modified posterior object. In particular, writing
$$
\E e^{-\MD_{m,n}(\theta)}=e^{-A_n(\theta)}\E e^{-B_m(\theta)}, 
$$ and $\MD_n(\theta)=A_n(\theta)+B(\theta)$, we have, from the refinement of Jensen's inequality, in the case where $\E_{\U}[B_m(\theta)]=B(\theta)$,
$$
e^{-A_n(\theta)+B(\theta)}+e^{-A_n(\theta)}\frac{c_2(\theta)}{2m}\inf_{x\ge0} e^{-x}\le e^{-A_n(\theta)}\E e^{-B_m(\theta)}\le e^{-A_n(\theta)+B(\theta)}+e^{-A_n(\theta)}\frac{c_2(\theta)}{2m}\sup_{x\ge0} e^{-x}
$$which simplifies to 
$$
e^{-\MD_n(\theta)}\le e^{-A_n(\theta)}\E e^{-B_m(\theta)}\le e^{-\MD_n(\theta)}+e^{-A_n(\theta)}\frac{c_2(\theta)}{2m}=\frac{e^{-A_n(\theta)+\log c_2(\theta)}}{2m}.
$$

If we redefine $c_2(\theta)=e^{-A_n(\theta)+\log c_2(\theta)}$, then we can set $c_1(\theta)=0$ in the proof of \Cref{lem:tvbound} and follow the same steps to obtain the stated result using the modified version of \Cref{ass:beta_positive}(iii).
\end{proof}
\thmnew*

The proof of Theorem \ref{thm:new} relies on arguments that are somewhat similar to the analysis of Gibbs measures common in PAC-Bayesian analysis, see \citet{alquier2016properties} for an example. This is noteworthy for two reasons: first, $\overline{\pi}(\theta\mid \MD_{m,n})$ itself is \textit{not a Gibbs measure; second, PAC-Bayes arguments do not treat simulation error due to estimation of loss functions.} These two features require us to generalized existing arguments along these two interesting directions. 
In the proof,  this conundrum is partly solved by leveraging the fact that $\overline\pi(\theta\mid\MD_{m,n})$ is the marginal of the joint distribution
\begin{IEEEeqnarray}{rCl}
\overline{\pi}(\theta,u_{1:m}\mid\MD_{m,n})
& = & \frac{\pi(\theta)p_\theta(u_{1:m})\exp\{-\omega\cdot \MD_{m,n}(\theta,u_{1:m})\}}{\int\int\pi(\theta)p_\theta(u_{1:m})\exp\{-\omega\cdot \MD_{m,n}(\theta,u_{1:m})\}\dt\theta\dt u_{1:m}},
\nonumber
\end{IEEEeqnarray}
which \textit{is representable as a Gibbs measure} and solves a generalised variational problem (see Lemma \ref{lem:restate} in the supplement).
In the proof, PAC-Bayes arguments are applied to the joint distribution $\overline{\pi}(\theta,u_{1:m}\mid\MD_{m,n})$ above, and are then transferred to $\overline{\pi}(\theta\mid\MD_{m,n})$ via a novel marginalisation argument.
We believe that this technique  may prove useful for future work on generalising Bayesian techniques---including for predictive approaches and hierarchical Bayesian methods.

\begin{proof}[Proof of \Cref{thm:new}]
The arguments proceed in 5 steps:
\begin{itemize}
    \item[\textbf{Step 1}:] Derive a PAC-Bayes bound for the joint measure $\overline{\pi}(\theta, u\mid \MD_{m,n})$ via  \Cref{lem:restate};
    \item[\textbf{Step 2}:] Marginalise this bound, which allows the bound to apply to $\overline{\pi}(\theta\mid \MD_{m,n})$;
    \item[\textbf{Step 3}.] Obtain convergence rates for the generalisation bound of $\overline{\pi}(\theta, u\mid \MD_{m,n})$ via Step 1;
    \item[\textbf{Step 4}:]
    Using the argument in Step 2, transfer these  rates to $\overline{\pi}(\theta\mid \MD_{m,n})$;
    \item[\textbf{Step 5}:]
    Re-formulate the bound into a posterior concentration result via Markov's inequality.
\end{itemize}

\paragraph{Step 1:}
The start of the proof deviates from the usual PAC-Bayes recipe in order to treat the specific nature of $\MD_{m,n}(\theta,u)$. Assume for now that $m,n,\lambda$ are such that  $g(m,n,\lambda)\le 1$; this is without loss of generality as we will later choose $\lambda$ to ensure that this condition is satisfied. Applying Assumption \ref{ass:tails-new}, we have
\begin{flalign*}
\E_0\E_{(\theta,u)\sim \pi\times p_\theta}e^{-\lambda[\MD_{m,n}(\theta)-\MD(\theta)]}&\le \int e^{g(m,n,\lambda)\eta(\theta)}\pi(\theta)\dt\theta\le e^{g(m,n,\lambda)}\int e^{\eta(\theta)}\pi(\theta)\dt\theta,
\end{flalign*}where the second inequity follows since, for $x\in(0,1]$, and any $t\ge0$, $e^{xt}\le e^x e^t$. By Assumption \ref{ass:tails-new}, $\tau:=\int e^{\eta(\theta)}\pi(\theta)\dt\theta<\infty$, and we can use Fubini's Theorem and the above to write 
\begin{equation}\label{eq:start}
\E_0\E_{(\theta,u)\sim \pi\times p_\theta}e^{-\lambda[\MD_{m,n}(\theta)-\MD(\theta)]-g(m,n,\lambda)-\log(\tau)}\le 1.
\end{equation}
Apply  Lemma \ref{lem:restate} with 
$$
h(\theta,u)=-\lambda\left[\MD_{m,n}(\theta,u)-\MD_{}(\theta)\right]-g(m,n,\lambda)-\log(\tau)
$$to obtain 
\begin{flalign*}
	&\log\E_{(\theta,u)\sim\pi\times p_\theta}\left[e^{-\lambda\left\{\MD_{m,n}(\theta,u)-\MD_{}(\theta)\right\}-g(m,n,\lambda)-\log(\tau)}\right]
	\\&=\sup_{\rho\in\mathcal{P}(\Theta\times \mathcal{U})}\bigg{\{}-\lambda\int\left\{\MD_{m,n}(\theta,u)-\MD_{}(\theta)\right\}\dt\rho(\theta,u)-\KL(\rho\|\pi\times p_\theta)-g(m,n,\lambda)-\log(\tau) \bigg{\}},
\end{flalign*}and apply the above to \eqref{eq:start} to deduce
\begin{flalign*}
	1\ge \E_0\exp\bigg{[} \sup_{\rho\in\mathcal{P}(\Theta\times \mathcal{U})}&\bigg{\{}-\lambda\int\left[\MD_{m,n}(\theta,u)-\MD_{}(\theta)\right]\dt\rho(\theta,u)-\KL(\rho\|\pi\times p_\theta)-g(m,n,\lambda)-\log(\tau) \bigg{\}}\bigg{]},
\end{flalign*}where we recall that $\pi\times p_\theta$ is shorthand notation for the joint distribution $\pi(\theta)\times p_\theta(u)$. Take $\rho=\widehat\rho(\theta,u)\propto e^{-\lambda \MD_{m,n}(\theta,u)}\pi(\theta)p_\theta(u)$ in the above to obtain 
\begin{flalign*}
	1\ge \E_0\exp &\bigg{\{}-\lambda\int\left[\MD_{m,n}(\theta,u)-\MD_{}(\theta)\right]\dt\widehat{\rho}(\theta,u)-\KL(\widehat{\rho}\|\pi\times p_\theta) -g(m,n,\lambda)-\log(\tau)\bigg{\}}.
\end{flalign*}
Applying Jensen's inequality, and taking logs on both sides yields
\begin{flalign*}
	0\ge \E_0&\bigg{\{}-\lambda\int\left\{\MD_{m,n}(\theta,u)-\MD_{}(\theta)\right\}\dt\widehat{\rho}(\theta,u)-\KL(\widehat{\rho}\|\pi\times p_\theta)-g(m,n,\lambda)-\log(\tau) \bigg{\}},
\end{flalign*}which, using the fact that $\MD(\theta_0)=\E_0\MD_n(\theta_0)$, and recalling that $\MD(\theta,\theta_0)=\MD(\theta)-\MD(\theta_0)\ge0$ we can rewrite as 
$$
\E_0\int\MD(\theta,\theta_0)\dt\widehat{\rho}(\theta,u)\le \E_0\bigg{\{}\int[\MD_{m,n}(\theta,u)-\MD_n(\theta_0)]\dt\widehat{\rho}(\theta,u)+\frac{\KL(\widehat{\rho}\|\pi\times p_\theta)}{\lambda}\bigg{\}}+\frac{g(m,n,\lambda)}{\lambda}+\frac{\log(\tau)}{\lambda} .
$$
Re-arranging terms and using the fact that 
\begin{flalign*}
\widehat{\rho}:=&%\arg\sup_{\rho\in\mathcal{P}(\Theta\times\mathcal{U})}\left\{-\lambda\int \MD_{m,n}(\theta,u)\dt\rho-\KL(\rho\|\pi\times p_\theta)\right\}
%\\
%\equiv& 
\arginf_{\rho\in\mathcal{P}(\Theta\times\mathcal{U})}\left\{\int \MD_{m,n}(\theta,u)\dt\rho+\frac{\KL(\rho\|\pi\times p_\theta)}{\lambda}\right\}
\end{flalign*}
to obtain
\begin{flalign}
&\E_{0}\int\MD(\theta,\theta_0)\dt\widehat{\rho}(\theta,u)\nonumber\\\le &\E_{0}\bigg{\{}\int[\MD_{m,n}(\theta,u)-\MD_{n}(\theta_0)]\dt\widehat{\rho}(\theta,u)+\frac{\KL(\widehat{\rho}\|\pi\times p_\theta)}{\lambda}\bigg{\}} 	+\frac{g(m,n,\lambda)}{\lambda}+\frac{\log(\tau)}{\lambda}
\nonumber\\=&\E_0\inf_{\rho\in\mathcal{P}(\Theta\times\mathcal{U})}\left\{\int [\MD_{m,n}(\theta,u)-\MD_n(\theta_0)]\dt\rho+\frac{\KL({\rho}\|\pi\times p_\theta)}{\lambda}\right\} 	+\frac{g(m,n,\lambda)}{\lambda}+\frac{\log(\tau)}{\lambda}.\label{eq:last_new_eq}	
\end{flalign}
Note that, from the discussion in \Cref{rem:gibbs},  
$$
\overline\pi(\theta\mid \MD_{m,n})=\frac{\int_{\mathcal{U}}e^{-\MD_{m,n}(\theta,u)}p_\theta(u)\pi(\theta)\dt u}{\int_\Theta\int_{\mathcal{U}}e^{-\MD_{m,n}(\theta,u)}p_\theta(u)\pi(\theta)\dt\theta\dt u}=\int_{\mathcal{U}}\widehat\rho(\theta,u)\dt u.
$$

\paragraph{Step 2:}
Now, Fubini's Theorem implies $\int \MD(\theta,\theta_0)\dt\widehat{\rho}(\theta,u) = \int \MD(\theta,\theta_0)\overline{\pi}(\theta\mid \MD_{m,n})\dt\theta$, so that \eqref{eq:last_new_eq} becomes
\begin{flalign}
\E_0&\int\MD(\theta,\theta_0)\overline\pi(\theta\mid\MD_{m,n})\dt\theta\nonumber\\\E_0&\int\MD(\theta,\theta_0)\dt\widehat{\rho}(\theta,u)\nonumber\\\le \E_0&\inf_{\rho\in\mathcal{P}(\Theta\times\mathcal{U})}\left\{\int [\MD_{m,n}(\theta,u)-\MD_n(\theta_0)]\dt\rho(\theta)+\frac{\KL(\rho\|\pi\times p_\theta)}{\lambda}\right\}+\frac{g(m,n,\lambda)}{\lambda}+\frac{\log(\tau)}{\lambda}\label{eq:lhs1}.
\end{flalign}
Recall that $\mathcal{A}_{\epsilon_n}:=\{\theta\in\Theta:|\MD(\theta)-\MD(\theta_0)|\le\epsilon_n\}$ and define 
$$
\rho_n(\theta,u)=
\begin{cases}
	\frac{\pi(\theta)p_\theta(u)}{\Pi(\mathcal{A}_{\epsilon_n})}&\text{ if }\theta\in\mathcal{A}_{\epsilon_n}\\0&\text{ else }	
\end{cases}
.$$	Since \eqref{eq:lhs1} is true for the infimum it must hold for any $\rho \in \mathcal{P}(\Theta\times \mathcal{U})$, including $\rho_n$. Hence,
\begin{flalign}
	&\E_0\int \MD(\theta,\theta_0)\overline\pi(\theta\mid\MD_{m,n})\dt\theta\nonumber\\\le  &\E_0\left\{\int [\MD_{m,n}(\theta,u)-\MD_{n}(\theta_0)]\dt\rho_n(\theta,u)+\frac{\KL(\rho_n\|\pi\times p_\theta)}{\lambda}\right\}+\frac{g(m,n,\lambda)}{\lambda}+\frac{\log(\tau)}{\lambda}\label{eq:lhs2}.
\end{flalign}

\paragraph{Step 3:}
We now analyze the right hand side of \eqref{eq:lhs1}, which depends on the joint distribution of $(\theta,u)$. Apply Assumption \ref{ass:mean_bias} to see that 
\begin{flalign}
	&\E_0\int [\MD_{m,n}(\theta,u)-\MD_{n}(\theta_0)]\dt\rho_n(\theta,u)\nonumber\\=&\E_0\int [\MD_{m,n}(\theta,u)-\MD_n(\theta)]\dt\rho_n(\theta,u)+\E_0\int [\MD_n(\theta)-\MD_{n}(\theta_0)]\dt\rho_n(\theta,u)\nonumber\\\le &\frac{1}{\Pi(\mathcal{A}_{\epsilon_n})}\left\{\frac{1}{m^\kappa}\int_{\mathcal{A}_{\epsilon_n}} \gamma(\theta)\pi(\theta)\dt\theta+\int_{\mathcal{A}_{\epsilon_n}}\MD(\theta,\theta_0)\pi(\theta)\dt\theta\right\}\label{eq:almost},
\end{flalign}
where the first term in the inequality follows from Fubini's Theorem and Assumption \ref{ass:mean_bias} along with the specific definition of $\rho_n(\theta,u)$;  the second follows from Fubini and the definition of   $\rho_n(\theta,u)$.
To handle the second term in \eqref{eq:almost}, note that on $\mathcal{A}_{\epsilon_n}$, $\MD(\theta,\theta_0)\le\epsilon_n$, so that 
\begin{flalign*}
	\frac{1}{\Pi(\mathcal{A}_{\epsilon_n})}\int_{\mathcal{A}_{\epsilon_n}}\MD(\theta,\theta_0)\pi(\theta)\dt\theta\le \epsilon_n.
\end{flalign*}
For the first term, since $\gamma(\theta)$ is continuous by Assumption \ref{ass:mean_bias}, so that $\sup_{\theta\in\mathcal{A}_{\epsilon_n}}\gamma(\theta)<\infty$, and 
\begin{equation*}
	\frac{1}{m^\kappa}\int_{\mathcal{A}_{\epsilon_n}}\frac{\gamma(\theta)\pi(\theta)}{\Pi(\mathcal{A}_{\epsilon_n})}\dt\theta\le \frac{\sup_{\theta\in\mathcal{A}_{\epsilon_n}}\gamma(\theta)}{m^\kappa}\int_{\mathcal{A}_{\epsilon_n}}\frac{\pi(\theta)}{\Pi(\mathcal{A}_{\epsilon_n})}\dt\theta\lesssim\frac{1}{m^\kappa}.
\end{equation*}Plugging the above display equations into \eqref{eq:almost} then delivers
\begin{flalign}
	\E_0\int [\MD_{m,n}(\theta,u)-\MD_{n}(\theta_0)]\dt\rho_n(\theta,u)&\lesssim\epsilon_n+m^{-\kappa}\label{eq:term1}.
\end{flalign}

\paragraph{Step 4:}
Returning to \eqref{eq:lhs1}, and applying \eqref{eq:term1} delivers
\begin{flalign}
	&\E_0\int \MD(\theta,\theta_0)\overline\pi(\theta\mid\MD_{m,n})\dt\theta\nonumber\\\le&  \E_0\left\{\int [\MD_{m,n}(\theta,u)-\MD_{n}(\theta_0)]\dt\rho_n(\theta,u)+\frac{\KL(\rho_n\|\pi\times p_\theta)}{\lambda}\right\}\nonumber\\\lesssim &\epsilon_n+\frac{1}{m^\kappa}+\frac{\KL(\rho_n\|\pi\times p_\theta)}{\lambda}\label{eq:lhs3}.
\end{flalign}To handle the third term in equation \eqref{eq:lhs3}, apply the Standing Assumption (prior mass) and the definition of $\rho_n(\theta,u)$, to see that 
$$
\KL(\rho_n\|\pi\times p_\theta)=-\log\Pi(\mathcal{A}_{\epsilon_n})\le \lambda\epsilon_n.
$$
We then have
\begin{flalign}
&\E_0\int \MD(\theta,\theta_0)\overline\pi(\theta\mid\MD_{m,n})\dt\theta\nonumber\\\le&  \E_0\left\{\int [\MD_{m,n}(\theta,u)-\MD_{n}(\theta_0)]\dt\rho_n(\theta,u)+\frac{\KL(\rho_n\|\pi\times p_\theta)}{\lambda}\right\}+\frac{g(m,n,\lambda)}{\lambda}+\frac{\log(\tau)}{\lambda}\nonumber\\\lesssim &\epsilon_n+\frac{1}{m^\kappa}+{\epsilon_n}+\frac{g(m,n,\lambda)}{\lambda}+\frac{\log(\tau)}{\lambda}\nonumber\\\lesssim &\epsilon_n+\frac{1}{m^\kappa}+{\epsilon_n}+\frac{\lambda}{\min\{m,m^\kappa,n\}}+\frac{\log(\tau)}{\lambda}\label{eq:final_term}.
\end{flalign}
where the last line follows from the fact that $g(m,n,\lambda)\le \lambda^2/\min\{m,m^\kappa,n\}$. 

%Equation \eqref{eq:almost_final_term} immediately implies that we must choose $m$ so that $m\gtrsim n^{1/\kappa}$ to control 
%$\E_0\int \MD(\theta,\theta_0)\overline\pi(\theta\mid\MD_{m,n})\dt\theta$.%%, and delivers the reasoning for this maintained assumption. 
%Recall that $g(m,n,\lambda)\le \lambda^2/\min\{m^\kappa,n\}$, and take %$\lambda\ge \min $
%$
%f(m,n,\lambda)=\frac{\lambda-g(m,n,\lambda)}{\lambda}
%$. Since $\lambda-g(m,n,\lambda)>0$ by Assumption \ref{ass:tails-new}, choose $\lambda$ such that
%$
%\frac{\lambda-g(m,n,\lambda)}{\lambda}<1/K
%$ for some $K>0$. We can then write %\eqref{eq:almost_final_term} as
%\begin{flalign}
%	&\E_0\int \MD(\theta,\theta_0)\overline\pi(\theta\mid\MD_{m,n})\dt\theta\nonumber\\\le &K  \E_0\left\{\int [\MD_{m,n}(\theta,u)-\MD_{n}(\theta_0)]\dt\rho_n(\theta,u)(\theta,u)+\frac{\KL(\rho_n\|\pi\times p_\theta)}{\lambda}\right\}\nonumber\\\lesssim &2\epsilon_n+\frac{1}{m^\kappa}\label{eq:final_term}.
%\end{flalign}

\paragraph{Step 5:}
The result now follows by applying Markov's inequality. In particular, recall  the definition $r_{m,n}=\min\{m^\kappa,n\}$, and use equation \eqref{eq:final_term} to obtain
\begin{flalign*}
\E_0\overline\Pi\left[|\MD(\theta)-\MD(\theta_0)|>M_n\epsilon_n\mid\MD_{m,n}\right]\le& \frac{\E_0\int_\Theta |\MD(\theta)-\MD(\theta_0)|\overline\pi(\theta\mid\MD_{m,n})\dt\theta}{M_n\epsilon_n}\\\lesssim& \frac{1}{M_n}\left[2+\frac{1}{m^\kappa\epsilon_n}+\frac{\lambda}{\epsilon_nr_{m,n}}+\frac{\log(\tau)}{\epsilon_n\lambda}\right]%\\\asymp &\frac{1}{M}\left(1+\frac{1}{m^\kappa\epsilon_n}\right).
%\\&\lesssim\frac{1}{M_n}\left[2+\frac{1}{n\epsilon_n}\right]\\&\lesssim\frac{3}{M_n},  
\end{flalign*}
Take $\lambda= r_{m,n}^{1/2}/c_1$ for some $c_1>1$, where we recall that $r_{m,n}=\min\{m,m^\kappa,n\}$. Note that $g(m,n,\lambda)\le \lambda^2/r_{m,n}\le 1/c_1$, which can always be made less than unity by choosing $c_1$ large enough, and validates the first step of the proof. Now, applying this choice for $\lambda$, we obtain
\begin{flalign*}
\E_0\overline\Pi\left[|\MD(\theta)-\MD(\theta_0)|>M_n\epsilon_n\mid\MD_{m,n}\right]\lesssim& \frac{1}{M_n}\left[2+\frac{1}{m^\kappa\epsilon_n}+\frac{1}{\epsilon_nr_{m,n}^{1/2}}+\frac{\log(\tau)}{\epsilon_nr_{m,n}^{1/2}}\right].
%\\&\lesssim\frac{1}{M_n}\left[2+\frac{1}{n\epsilon_n}\right]\\&\lesssim\frac{3}{M_n},  
\end{flalign*}
Choosing $\epsilon_n>0$ such that $\epsilon_nr_{m,n}^{1/2}\rightarrow\infty$ then yields the stated result.

%where the second inequality applies that $m^\kappa\gtrsim n^{}$ and the last the fact that $n\epsilon_n\ge 1$. 
%To obtain the second part of the result, use the fact that $m^\kappa\asymp n^{}$ and the $n\epsilon_n\ge 1$ to see that $1/(m^{\kappa}\epsilon_n)\asymp 1/(n\epsilon_n)\lesssim 1$.

\end{proof}

\section{Zig-zag results and implementation details}
In this section we give the proof of
Theorem \ref{thm:StationaryDist_ZigZag}, which establishes that Algorithm \ref{Alg:ZigZag_EstimatedLoss} samples from $\pi(\theta \mid \MD_n)$,  elaborate on the implementation of Algorithm \ref{Alg:ZigZag_EstimatedLoss} for the $\beta$-divergence and MMD loss functions, and cover two straightforward and practical extensions of the simpler version of the algorithm presented in the main paper.

\subsection{Proofs and technical results}

\subsubsection{Technical Lemmas}

For completeness, we restate the Poisson thinning result of \cite{lewis1979simulation}.
This is a classical result, and required for the proof of Theorem \ref{thm:StationaryDist_ZigZag}.

\begin{lemma}[Poisson thinning \cite{lewis1979simulation}]
     Let $\lambda:
\mathbb{R}_{+}\mapsto \mathbb{R}_{+}$ and $\Lambda: \mathbb{R}_{+} \mapsto \mathbb{R}_{+}$ be continuous such that $\lambda(t) \leq \Lambda(t)$ for $t \geq 0$.
Let $\tau_1, \tau_2, \ldots$ be the increasing finite or infinite sequence of points sampled from a Poisson process with rate function $(\Lambda(t))_{t\geq 0}$. For all $i$, delete the point $\tau_i$ with probability $1 -\lambda(\tau_i)/ \Lambda(\tau_i)$. Then the remaining points, say  $\tilde{\tau}_1, \tilde{\tau}_2, \ldots$, are distributed according to a Poisson process with rate function $(\lambda(t))_{t\geq 0}$.
\label{Prop:Thinning}
\end{lemma}

Additionally, we restate a result and its conditions from \cite{bierkens2019zig}  required for the proof of Theorem \ref{thm:StationaryDist_ZigZag}. We refer to Section \ref{sec:PDMP} a definition of zig-zag processes.

\begin{assumption}[Assumption 2.1 of \cite{bierkens2019zig}]{\label{ass:bierkens}}
For $C^1(\mathbb{R}^d)$ is the space of continuously differentiable functions on $\mathbb{R}^d$ and some function $\Psi\in C^1(\mathbb{R}^d)$  satisfying
\begin{IEEEeqnarray}{rCl}
    \int_{\mathbb{R}^d}\exp(-\Psi(\theta))d\theta < \infty,\nonumber
\end{IEEEeqnarray}
we have zig-zag Poisson intensity functions $\lambda_j: \mathbb{R}_{+}\times \mathbb{R}^d\times \{-1,1\}\mapsto \mathbb{R}_{+}$ such that
\begin{IEEEeqnarray}{rCl}
    %\lambda_j(t; \theta, \nu) - \lambda_j(t; \theta, S_j(\nu)) = \nu_j\frac{\partial}{\partial \theta_j} \Psi(\theta + \nu \cdot t)\textrm{ for all } t > 0, \theta\in\mathbb{R}^d, \nu\in \{-1, 1\}^d, j = 1,\ldots, d, \nonumber
    \lambda_j(\theta, \nu) - \lambda_j(\theta, S_j(\nu)) = \nu_j\frac{\partial}{\partial \theta_j} \Psi(\theta)\textrm{ for all } \theta\in\mathbb{R}^d, \nu\in \{-1, 1\}^d, j = 1,\ldots, d, \nonumber
\end{IEEEeqnarray} 
where $S_j:\{-1, 1\}^d\mapsto\{-1,1\}^d$ such that $\{S_j(\nu)\}_j = -\nu_j$ and $\{S_j(\nu)\}_{j^\prime} = \nu_{j^\prime}$ for all $j^\prime\neq j$.
\end{assumption}

%\begin{theorem}[Theorem 2.2 of \cite{bierkens2019zig}]{\label{thm:bierkens}}
%    Suppose Assumption \ref{ass:bierkens} holds, let $\mu$ denote the probability distribution over $\mathbb{R}^d\times\{-1,1\}^d$ such that $\mu$ has Radon-Nikodym derivative
%    \begin{IEEEeqnarray}{rCl}
%        \frac{d\mu}{d\mu_0}(\theta, \nu) = \frac{\exp(-\Psi(\theta)}{\int_{\mathbb{R}^d\times\{-1,1\}^d} \exp(-\Psi(\theta))d\mu_0}\nonumber
%    \end{IEEEeqnarray}
%    Then the Zig-Zag process with switching
%rates $\lambda_j$ for $j = 1,\ldots, p$ has invariant distribution $\mu$.
%\end{theorem}

\begin{theorem}[Theorem 2.2 of \cite{bierkens2019zig}]{\label{thm:bierkens}}
    Suppose Assumption \ref{ass:bierkens} holds. For a probability density on $\mathbb{R}^d$ given by
    \begin{IEEEeqnarray}{rCl}
        \rho(\theta) = \frac{\exp(-\Psi(\theta))}{\int_{\mathbb{R}^d} \exp(-\Psi(\theta))d\theta},\label{eq:zig-zag-stationary-dist}
    \end{IEEEeqnarray}
    the invariant distribution of the zig-zag process with switching rates $\lambda_j$ for $j = 1,\ldots, d$ after marginalisation over the velocities $\nu$ has the density  $\rho$ over $\Theta$ defined as in \eqref{eq:zig-zag-stationary-dist}.
\end{theorem}

\subsubsection{Proof of Theorem \ref{thm:StationaryDist_ZigZag}}\label{sec:ProofStationaryDist_ZigZag}

The proof of Theorem \ref{thm:StationaryDist_ZigZag} closely follows the proof of Theorem 4.1 in \cite{bierkens2019zig}. We divide the required results into two Lemmas for simplicity.
Lemma \ref{Lem:switching} follows the first part of Theorem 4.1 in \cite{bierkens2019zig} and establishes the effective switching rates, $\mathbb{P}\left(\nu_j \mapsto -\nu_{j}\mid \theta\right)$, of the zig-zag algorithm presented in Algorithm \ref{Alg:ZigZag_EstimatedLoss}.

\begin{lemma}[Effective switching rates of Algorithm \ref{Alg:ZigZag_EstimatedLoss}] Assume that %$\hat{\Lambda}_j(t; \theta, \nu)$ 
$\hat{\Lambda}_j(\theta, \nu)$ satisfies \eqref{eq:computational-bound-estimated}. The effective switching rates, $\tilde{\lambda}_j(\theta, \nu) := \mathbb{P}\left(\nu_j \mapsto -\nu_{j}\mid \theta\right)$, of Algorithm \ref{Alg:ZigZag_EstimatedLoss} are
   \begin{IEEEeqnarray}{rCL}
        %\tilde{\lambda}_j(t; \theta, \nu) &:=& \mathbb{E}_{u_{1:b}\sim P_{\theta}}\left[(\nu_j \cdot \left[\Psi_{b,n}'(\theta + \nu \cdot t)\right]_j)^{+}\right],~j =1,\ldots, d. \label{Equ:SwitchingRate}
        \tilde{\lambda}_j(\theta, \nu) &=& \mathbb{E}_{u_{1:b}\sim P_{\theta}}\left[\max\left\lbrace \nu_j \cdot \left[\Psi_{b,n}'(\theta)\right]_j, 0\right\rbrace\right],~j =1,\ldots, d. \label{Equ:SwitchingRate}
    \end{IEEEeqnarray}
    {\label{Lem:switching}}
\end{lemma}

\begin{proof}

Start from the previous skeleton point $(\tau^k, \theta^{(\tau^k)}, \nu^{(\tau^k)})$. Conditional on  $\tau_{j^*} = \argmin_j\tau_j$ with $\tau_j \sim \mathbb{P}(\tau_j \geq t) = \exp\left\{-\int_0^t\hat{\Lambda}_j(\theta^{(\tau^{k})} + \nu^{(\tau^{k})}\cdot s, \nu^{(\tau^{k})}) ds\right\}$ for $j=1,2,\dots, d$, the probability that dimension $j^{\ast}$ of $\theta$ has its velocity  `\textit{switched}' so that $\nu^{(\tau^{k+1})}_{j^{\ast}} = -\nu^{(\tau^{k})}_{j^{\ast}}$ at time $\tau^{k+1} = \tau^k + \tau_{j^*}$ is seen to be 
 
\begin{IEEEeqnarray}{rCl}   &&\mathbb{P}\left(\nu_{j^{\ast}}^{(\tau^{k+1})} = 
 -\nu_{j^{\ast}}^{(\tau^{k})}\mid \theta^{(\tau^{k+1})}\right) 
 \nonumber \\
& = & \int \mathbb{P}\left(\nu_{j^{\ast}}^{(\tau^{k+1})} = -\nu_{j^{\ast}}^{(\tau^{k})}\mid \theta^{(\tau^{k+1})}, u_{1:b}\right)dP_{\theta^{(\tau^{k})} + \nu^{(\tau^{k})}\cdot\tau_{j^*}}(u_{1:b})
\nonumber\\
&=& \mathbb{E}_{u_{1:b} \sim P_{\theta^{(\tau^{k})} + \nu^{(\tau^{k})}\cdot\tau_{j^*}}}\left[\frac{\hat{\lambda}_{j^{\ast}}(\theta^{(\tau^{k})} + \nu^{(\tau^{k})}\cdot\tau_{j^*}, \nu^{(\tau^{k})})}{\hat{\Lambda}_{j^{\ast}}(\theta^{(\tau^{k})} + \nu^{(\tau^{k})}\cdot\tau_{j^*}, \nu^{(\tau^{k})})}\right]
\nonumber\\
&=& \frac{1}{\hat{\Lambda}_{j^{\ast}}(\theta^{(\tau^{k})} + \nu^{(\tau^{k})}\cdot\tau_{j^*}, \nu^{(\tau^{k})})}\mathbb{E}_{u_{1:b} \sim P_{\theta^{(\tau^{k})} + \nu^{(\tau^{k})}\cdot\tau_{j^*}}}\left[\max\left\lbrace\nu^{(\tau^k)}_{j^{\ast}} \cdot \left[\Psi_{b,n}'(\theta^{(\tau^k)} + \nu^{(\tau^k)}\cdot \tau_{j^*})\right]_{j^{\ast}}, 0\right\rbrace\right]\nonumber
    \end{IEEEeqnarray}
    which by Poisson Thinning \citep{lewis1979simulation} (Lemma \ref{Prop:Thinning}) gives
    \begin{align}
        %\tilde{\lambda}_j(t; \theta, \nu) &= \mathbb{E}_{u_{1:b}\sim P_{\theta}}\left[(\nu_{j} \cdot \left[\Psi_{b,n}'(\theta + \nu \cdot t)\right]_{j})^{+}\right],~ \nonumber
        \tilde{\lambda}_j(\theta, \nu) &= \mathbb{E}_{u_{1:b}\sim P_{\theta}}\left[\max\left\lbrace\nu_{j} \cdot \left[\Psi_{b,n}'(\theta)\right]_{j}, 0, \right\rbrace\right],~ \nonumber
        %\label{Equ:SwitchingRate}
    \end{align}
     as the effective switching rate for the $j^{th}$ component of $\theta$.
\end{proof}

Lemma \ref{Lem:switching_correct}
follows the second part of the proof of Theorem 4.1 in \cite{bierkens2019zig} and verifies that the effective switching rates in \eqref{Equ:SwitchingRate} satisfy Assumption \ref{ass:bierkens}.

\begin{lemma}[Effective switching rates of Algorithm \ref{Alg:ZigZag_EstimatedLoss} satisfy Assumption \ref{ass:bierkens}]{\label{Lem:switching_correct}}
Assume that $\varphi_{b,n}(\theta)$ satisfies \eqref{eq:needed}. Then the effective switching rates  in
\eqref{Equ:SwitchingRate} satisfy Assumption \ref{ass:bierkens}.
\end{lemma}

\begin{proof}
%We then verify that the switching rates $(\tilde{\lambda}_j)$ given by \eqref{Equ:SwitchingRate} satisfies 
The effective switching rates in \eqref{Equ:SwitchingRate} are such that
%    \begin{align}
%        &\tilde{\lambda}_j(t; \theta, \nu) - \tilde{\lambda}_j(t; \theta, S_j[\nu])\nonumber \\
%        &= \mathbb{E}_{u_{1:b} \sim P_{\theta+\nu \cdot t}}\left[(\nu_j \cdot \left[\Psi_{b,n}'(\theta + \nu \cdot t)\right]_j)^{+}\right] - \mathbb{E}_{u_{1:b} \sim P_{\theta + \nu \cdot t}}\left[(-\nu_j \cdot \left[\Psi_{b,n}'(\theta + \nu \cdot t)\right]_j)^{+}\right]\nonumber\\
%        &= \mathbb{E}_{u_{1:b} \sim P_{\theta + \nu \cdot t}}\left[(\nu_j \cdot \left[\Psi_{b,n}'(\theta + \nu \cdot t)\right]_j)^{+} - (-\nu_j \cdot \left[\Psi_{b,n}'(\theta + \nu \cdot t)\right]_j)^{+}\right]\nonumber\\
%        &= \mathbb{E}_{u_{1:b} \sim P_{\theta + \nu \cdot t}}\left[\nu_j \cdot \left[\Psi_{b,n}'(\theta + \nu \cdot t)\right]_j\right]\nonumber\\
%       &= \mathbb{E}_{u_{1:b} \sim P_{\theta + \nu \cdot t}}\left[\nu_j\left(-\frac{\partial}{\partial\theta}\log \pi(\theta) + \omega  \varphi_{b,n}(\theta)\right)\right]\nonumber\\
%        &=-\nu_j\left(-\frac{\partial}{\partial\theta}\log \pi(\theta) + \omega \mathbb{E}_{u_{1:b} \sim P_{\theta + \nu \cdot t}}\left[\varphi_{b,n}(\theta)\right]\right)\nonumber\\
%        &= \nu_j\frac{\partial}{\partial{\theta_j}}\Psi\left(\theta + \nu \cdot t\right),\nonumber
%    \end{align}
\begin{align}
        \tilde{\lambda}_j(\theta, \nu) - \tilde{\lambda}_j(\theta, S_j[\nu])
        &= \mathbb{E}_{u_{1:b} \sim P_{\theta}}\left[\max\left\lbrace\nu_j \cdot \left[\Psi_{b,n}'(\theta)\right]_j, 0 \right\rbrace\right]\nonumber\\
        &\quad -\mathbb{E}_{u_{1:b} \sim P_{\theta}}\left[\max\left\lbrace-\nu_j \cdot \left[\Psi_{b,n}'(\theta)\right]_j, 0 \right\rbrace\right]\nonumber\\
        % &= \mathbb{E}_{u_{1:b} \sim P_{\theta}}\left[(\nu_j \cdot \left[\Psi_{b,n}'(\theta )\right]_j)^{+} - (-\nu_j \cdot \left[\Psi_{b,n}'(\theta)\right]_j)^{+}\right]\nonumber\\
        &= \mathbb{E}_{u_{1:b} \sim P_{\theta}}\left[\nu_j \cdot \left[\Psi_{b,n}'(\theta)\right]_j\right]\nonumber\\
        &= \mathbb{E}_{u_{1:b} \sim P_{\theta}}\left[\nu_j\left(-\frac{\partial}{\partial\theta}\log \pi(\theta) + \omega  \varphi_{b,n}(\theta)\right)\right]\nonumber\\
        &=-\nu_j\left(-\frac{\partial}{\partial\theta}\log \pi(\theta) + \omega \mathbb{E}_{u_{1:b} \sim P_{\theta}}\left[\varphi_{b,n}(\theta)\right]\right)\nonumber\\
        &= \nu_j\frac{\partial}{\partial{\theta_j}}\Psi_n\left(\theta\right),\nonumber
    \end{align}
    by \eqref{eq:needed}. Therefore, $\tilde{\lambda}_j(\theta, \nu)$ satisfies Assumption \ref{ass:bierkens} as required.
\end{proof}

We now restate and  prove Theorem \ref{thm:StationaryDist_ZigZag}. 

\thmzigzag*

\begin{proof}[Proof of \Cref{thm:StationaryDist_ZigZag}]   Lemma \ref{Lem:switching_correct} verified that the effective switching rates of Algorithm \ref{Alg:ZigZag_EstimatedLoss}, established in Lemma \ref{Lem:switching}, satisfy Assumption \ref{ass:bierkens}. Therefore by Theorem \ref{thm:bierkens}, the zig-zag process has invariant marginal density
    \begin{align}
        \pi(\theta \mid \MD_n) = \frac{\exp(-\Psi_n(\theta))}{\int\exp(-\Psi_n(\theta))d\theta}\nonumber
    \end{align}
    for $\theta$, as required.
\end{proof}

The proof of Theorem \ref{thm:StationaryDist_ZigZag} is related to Theorem 4.1 in \cite{bierkens2019}: both results show that the zig-zag can sample exactly when using unbiased gradient estimators. However, the source of gradient estimation in the two results is completely different, and so the results themselves are distinct. Indeed, \cite{bierkens2019} generate sub-samples for the observed data, which are then used to approximate a tractable, but potentially time consuming average over $n$ data-points. In their construction, the sub-sampling mechanism is independent on the location of the sampler $\theta^{(\tau)}$. On the other hand, Theorem \ref{thm:StationaryDist_ZigZag} samples the pseudo observations, $u$, from the model implied by the location of the sampler $\theta^{(\tau)}$ and uses these to construct an unbiased estimate of an intractable expectation.

This difference results in different requirements for the computational bounds needed to implement the algorithms. \cite{bierkens2019} merely requires bounds that hold across the $n$ observations, while Condition \ref{eq:computational-bound-estimated} requires these bounds to be valid across the entire support of the pseudo-observations $u$, which is a more onerous requirement.

\subsection{Practical Implementations and Computational bounds}{\label{sec:zigzag_implementation}}

Following on from Section \ref{sec:constructing_grad}, we discuss how to construct computational bounds $\hat{\Lambda}_j$ satisfying 
\eqref{eq:computational-bound-estimated} for Algorithm \ref{Alg:ZigZag_EstimatedLoss} when using the $\beta$-divergence or MMD losses.
%
% Finding zig-zag computational bounds %$\hat{\Lambda}_j(t, \theta,\nu)$ 
% $\hat{\Lambda}_j(\theta,\nu)$ 
%  satisfying 
% \eqref{eq:computational-bound-estimated} as functions of $\theta$ and $\nu$ is a requirement for the zig-zag to be used for standard posteriors and 
%Finding such bounds is generally not  challenging, although loose bounds may degrade performance of the sampler.
%
For our setting, the main  challenge is to find a bound on $|\varphi_{b, n}(\theta)|$ that holds uniformly for all $u_{1:b}$.
While the exact form of such bounds will depend on both the loss and the model, we provide generic conditions that simplify the problem and exemplify these on the examples we treat in Section \ref{sec:examples}. 
%Full details for our applications are presented in Section \ref{Sec:PoissonRegression-app}.

\subsubsection{$\beta$-divergence}{\label{sec:betaDappendix}}

Built on the discussion in Section \ref{sec:constructing_grad}, Lemma \ref{lem:betaD_direct} provides $\varphi^{\beta}_{m, n}(\theta)$ for $\MD_{m,n}^{\beta}$, and shows that a generic computational bound can be constructed via $\left|\frac{\partial}{\partial{\theta_j}} \log p_\theta(u) p^\beta_\theta(u)\right|\leq R_j(\theta,\beta)$, proving \Cref{corr:betaD_bounds}.
%The $\beta$D loss \eqref{eq:loss-estimated-beta-div} requires the evaluation of $p_{\theta}(y_i)$ and therefore constructing $\varphi_{m, n}(\theta)$ via \textit{direct} sampling is logical. Lemma \ref{lem:betaD_direct} provides the form of $\varphi_{m, n}(\theta)$ and conditions for computational bounds.

\begin{lemma}\label{lem:betaD_direct}
 For the $\beta$-divergence loss, taking 
    \begin{align}
&\varphi^{\beta}_{b, n}(\theta)=\frac{(\beta + 1)}{bn}\sum_{i=1}^n\sum_{k=1}^b\frac{\partial}{\partial \theta} \log p_{\theta}(u_k)p_{\theta}(u_k)^{\beta} - \frac{(\beta + 1)}{n}\sum_{i=1}^n\frac{\partial}{\partial \theta} \log p_{\theta}(y_i)p_{\theta}(y_i)^{\beta}\nonumber
\end{align} for $u_{1:b} \sim P_{\theta}$ satisfies \eqref{eq:needed}. 
Further, if there exists $R_j(\theta,\beta) \geq
\left|\frac{\partial}{\partial{\theta_j}} \log p_\theta(u) p^\beta_\theta(u)\right|
$, then
\begin{flalign*}
    %\hat{\lambda}_{j}(t; \theta,\nu) \leq& -\nu_j\frac{\partial}{\partial\theta}\log(\pi(\theta + \nu\cdot t)) +\frac{\omega(\beta + 1)}{n}\sum_{i=1}^nR_j(\theta + \nu\cdot t,\beta)\\
    %&- \frac{\nu_j\omega(\beta + 1)}{n}\sum_{i=1}^n\frac{\partial}{\partial{\theta_j}} \log p_{\theta + \nu\cdot t}(y_i)p_{\theta + \nu\cdot t}(y_i)^{\beta}.
    \hat{\lambda}_{j}(\theta,\nu) \leq& -\nu_j\frac{\partial}{\partial\theta_j}\log(\pi(\theta)) +\frac{\omega(\beta + 1)}{n}\sum_{i=1}^nR_j(\theta,\beta)- \frac{\nu_j\omega(\beta + 1)}{n}\sum_{i=1}^n\frac{\partial}{\partial{\theta_j}} \log p_{\theta}(y_i)p_{\theta}(y_i)^{\beta}.
\end{flalign*}
\end{lemma}

\begin{proof}
We decompose \eqref{eq:loss-estimated-beta-div} as
$
{\MD}^{\beta}_{b,n}(\theta)=  {\MD}^{\beta,1}_{b,n}(\theta) + {\MD}^{\beta,2}_{n}(\theta)$, where 
\begin{flalign*}
    {\MD}^{\beta,1}_{b,n}(\theta) &= \frac{1}{bn}\sum_{i=1}^n\sum_{k=1}^bp_{\theta}(u_k)^{\beta}\textrm{ and }{\MD}^{\beta,2}_{n}(\theta) = -\left(1+\frac{1}{\beta}\right)\frac{1}{n} \sum_{i=1}^n p_{\theta}(y_i)^\beta.
\end{flalign*}
Noting that $\mathbb{E}_{u_{1:b}\sim P_{\theta}}\left[{\MD}^{\beta}_{b,n}(\theta)\right] = {\MD}^{\beta}_{n}(\theta)$, we use the chain rule to conclude that 
\begin{IEEEeqnarray}{rCl}
&&\frac{\partial}{\partial \theta}\mathbb{E}_{u_{1:b}\sim P_{\theta}}\left[{\MD}^{\beta}_{b,n}(\theta)\right] = \frac{\partial}{\partial \theta}\left\{\mathbb{E}_{u_{1:b}\sim P_{\theta}}\left[{\MD}^{\beta,1}_{b,n}(\theta)\right] + {\MD}^{\beta,2}_{n}(\theta)\right\}\nonumber\\
&=& \int\frac{\partial}{\partial \theta}\left({\MD}^{\beta,1}_{b,n}(\theta)p_{\theta}(u_{1:b})\right)du_{1:b} + \frac{\partial}{\partial \theta}{\MD}^{\beta,2}_{n}(\theta)\nonumber\\
&=& \int\frac{\partial}{\partial \theta}{\MD}^{\beta,1}_{b,n}(\theta)p_{\theta}(u_{1:b})du_{1:b} + \int{\MD}^{\beta,1}_{b,n}(\theta)\frac{\partial}{\partial \theta}p_{\theta}(u_{1:b})du_{1:b}+ \frac{\partial}{\partial \theta}{\MD}^{\beta,2}_{n}(\theta)\nonumber\\
&=& \int\frac{\partial}{\partial \theta}{\MD}^{\beta,1}_{b,n}(\theta)p_{\theta}(u_{1:b})du_{1:b} + \int{\MD}^{\beta,1}_{b,n}(\theta)p_{\theta}(u_{1:b})\frac{\partial}{\partial \theta}\log p_{\theta}(u_{1:b})du_{1:b}+ \frac{\partial}{\partial \theta}{\MD}^{\beta,2}_{n}(\theta)\nonumber
\end{IEEEeqnarray}
where the final line uses the log-derivative trick. As a result, taking
\begin{align}
&\varphi^{\beta}_{b, n}(\theta)=\frac{\partial}{\partial\theta} \MD^{\beta,1}_{b,n}(\theta)+\MD^{\beta,1}_{b,n}(\theta)\frac{\partial}{\partial\theta} \log p_\theta(\U) + \frac{\partial}{\partial{\theta_j}}\MD_n^{\beta,2}(\theta)\nonumber\\
&=\frac{(\beta + 1)}{bn}\sum_{i=1}^n\sum_{k=1}^b\frac{\partial}{\partial \theta} \log p_{\theta}(u_k)p_{\theta}(u_k)^{\beta} - (\beta + 1)\frac{1}{n}\sum_{i=1}^n\frac{\partial}{\partial \theta} \log p_{\theta}(y_i)p_{\theta}(y_i)^{\beta}\nonumber
\end{align}
satisfies \eqref{eq:needed}.
Further, if $
\left|\frac{\partial}{\partial{\theta_j}} \log p_\theta(u) p^\beta_\theta(u)\right|\leq R_j(\theta,\beta)$, then the result follows, as
\begin{align}
    &\hat{\lambda}_{j}(\theta,\nu) = \max\left\{0, \nu_j \cdot \left[\Psi_{b,n}'(\theta)\right]_j\right\} =  \max\left\{0, \nu_j\left(-\frac{\partial}{\partial\theta}\log(\pi(\theta)) + \omega\varphi_{m, j}(\theta)\right)\right\}\nonumber\\
    &\leq -\nu_j\frac{\partial}{\partial\theta_j}\log(\pi(\theta)) +  \omega\frac{(\beta + 1)}{bn}\sum_{i=1}^n\sum_{k=1}^b R_j(\theta,\beta)\nonumber- \nu_j\omega(\beta + 1)\frac{1}{n}\sum_{i=1}^n\frac{\partial}{\partial \theta_j} \log p_{\theta}(y_i)p_{\theta}(y_i)^{\beta},
\end{align}
which finally concludes the proof.
\end{proof}

Proving \Cref{corr:betaD_bounds} is now trivial. 
\begin{proof}
    \Cref{lem:betaD_direct} provided $\varphi_{b,n}^{\beta}(\theta)$ satisfying \eqref{eq:needed} for $\MD_{n}(\theta)=\MD^{\beta}_{n}(\theta)$ and showed that if  $R_j(\theta,\beta) \geq
\left|\frac{\partial}{\partial{\theta_j}} \log p_\theta(u) p^\beta_\theta(u)\right|
$ then there exists $\hat{\Lambda}$ satisfying \eqref{eq:computational-bound-estimated}. These are the two requirements for Algorithm \ref{Alg:ZigZag_EstimatedLoss} to sample form $\pi(\theta\mid\MD_n)$, thus proving \Cref{corr:betaD_bounds}.
\end{proof}

While  $\frac{\partial}{\partial{\theta_j}} \log p_\theta(u)$ may be unbounded for many popular models, constructing $R_j(\theta,\beta)$ is often still possible thanks to the fact that $p_\theta(u)\rightarrow 0$ in the same tail regions of the model where $\frac{\partial}{\partial{\theta_j}} \log p_\theta(u) \rightarrow\infty$. 
Below, we demonstrate this for the Poisson regression model that we study numerically in Section \ref{sec:examples}.

\begin{corollary} Choosing the $\beta$-divergence loss $\MD_{m,n}^{\beta}$ for a Poisson regression model with likelihood $p_{\theta}(u; x)=\frac{\lambda(\theta,x)^ue^{-\lambda(\theta,x)}}{u!}$, $\lambda(\theta,x)=e^{x^\top\theta}$, $u \in \mathbb{N}$, $x\in\mathbb{R}^p$ and $\theta\in\mathbb{R}^p$, the expression
    \begin{IEEEeqnarray}{rCl}
    R_j(\theta,\beta) =   \exp\{\lambda\beta-\beta\log\pi/2\}\exp\{e^{\{2\log(\lambda)-1\}}\}\left|x_j\right|      \nonumber
    \end{IEEEeqnarray}
    is an upper bound that satisfies Lemma \ref{lem:betaD_direct}.
    \label{Cor:betaD_Poisson}
\end{corollary}

\begin{proof}
%     We show that $
% \left|\frac{\partial}{\partial{\theta_j}} \log p_\theta(u; x) p^\beta_\theta(u; x)\right|\leq R_j(\theta,\beta)$. 
%
For simplicity of notation, define $\lambda=\exp(x^\top\theta)$. Since $p^\beta_\theta(u; x) < 1$ and $\frac{\partial}{\partial{\theta_j}} \log p_\theta(u; x) = (u - \lambda)x_j$, we have
\begin{IEEEeqnarray}{rCl}
    \left|\frac{\partial}{\partial{\theta_j}} \log p_\theta(u; x) p^\beta_\theta(u; x)\right| & = & \left|(u - \lambda)x_jp^\beta_\theta(u)\right|
    <
    \left(up^\beta_\theta(u; x) + \lambda\right)\left|x_j\right|.\nonumber
\end{IEEEeqnarray}
If we can show that,  
\begin{equation}\label{eq:show1}
up^\beta_\theta(u; x) \leq  \exp[\lambda\beta-\beta\log\pi/2]\exp\{e^{\{2\log(\lambda)-1\}}\}
\end{equation}
then
\begin{IEEEeqnarray}{rCl}
    \left|\frac{\partial}{\partial{\theta_j}} \log p_\theta(u; x) p^\beta_\theta(u; x)\right| 
    & < & \exp\{\lambda\beta-\beta\log\pi/2\}\exp[e^{\{2\log(\lambda)-1\}}]\left|x_j\right|,\nonumber
\end{IEEEeqnarray}
as required.

Now, rewrite 
$$
u p_\theta (u;x)^\beta = \exp\{\log(u)+\lambda\beta+u\beta\log(\lambda)-\beta\log(u!)\}.
$$
To demonstrate \eqref{eq:show1}, we note that, from the Ramanujan expansion of $\log(u!)$
$$
\log u!\approx u \log u-u+\frac{\log (u(1+4 u(1+2 u)))}{6}+\frac{\log (\pi)}{2},
$$ it can be verified that 
$$
\log (u!) \ge u\log u - u +\log(u)+\log(\pi)/2.
$$Applying the above then delivers
\begin{flalign*}
  u p_\theta (u;x)^\beta &\le \exp\{\log(u)+\lambda\beta+u\beta[1+\log(\lambda)] -u\beta\log(u)-\beta\log(u)-\beta\log(\pi)/2 \}\\&=\exp\{\lambda\beta-\beta\log\pi/2\}\exp\{\log(u)[1-\beta(u+1)]+u\beta[1+\log(\lambda)]\}
\end{flalign*}
First, let us rewrite the inner term as  
$$
\log(u)[1-(1/2)(u+1)]+(1/2)u[1+\log(\lambda)]=\log(u)/2-u\log(u)/2+u[1+\log\lambda]/2
$$Further, note that, for $u\ge 2$ larger, we have that $\log(u)\le u\log(u)/2$ so that we have 
\begin{flalign*}
\log(u)/2-u\log(u)/2+u[1+\log\lambda]/2&\le u\log(u)/4-u\log(u)/2+u[1+\log\lambda]/2  \\&\le   -u\log(u)/4+u[1+\log\lambda]2.
\end{flalign*}Applying this we arrive at 
\begin{flalign*}
  u p_\theta (u;x)^\beta &\le\exp\{\lambda\beta-\beta\log\pi/2\}\exp\{ -u\log(u)/4+u[1+\log\lambda]/2\}
\end{flalign*}
Now, consider maximizing the second term as a function of $u$. Since $\exp(\cdot)$ is monotonic, this is equivalent to maximizing the inner term, which is achieved at 
$$
u^\star=\exp\{2[1+\log(\lambda)]-1\}=\exp\{2\log(\lambda)-1\}.
$$Plugging in $u^\star$ into this term in we arrive at the upper bound
\begin{flalign*}
  u p_\theta (u;x)^\beta\le &\exp\{\lambda\beta-\beta\log\pi/2\}\\&\times\exp\{ -e^{\{2\log(\lambda)-1\}}\{2\log(\lambda)-1\}/4+e^{\{2\log(\lambda)-1\}}[1+\log\lambda]/2\}\\&\le \exp\{\lambda\beta-\beta\log\pi/2\}\\&\times \exp\{ -e^{\{2\log(\lambda)-1\}}\left[\{2\log(\lambda)-1\}/4-[1+\log\lambda]/2\right]\}\\&\le \exp\{\lambda\beta-\beta\log\pi/2\}\\&\times \exp\{-e^{\{2\log(\lambda)-1\}}(-3/4)\}\\&\le \exp\{\lambda\beta-\beta\log\pi/2\}\exp\{e^{\{2\log(\lambda)-1\}}\}
\end{flalign*}

\end{proof}

\subsubsection{Maximum Mean Discrepancy (MMD).}{\label{sec:MMDappendix}}

First, note that the gradient of the estimated MMD loss in \eqref{eq:mmd-loss-mn} is \textit{biased} for $\frac{\partial}{\partial\theta}\MD^k_{n}(\theta)$, and thus does not satisfy the requirements of \eqref{eq:needed}.
We instead consider 
\begin{IEEEeqnarray}{rCl}
\MD'^k_{b,n}(\theta, u_{1:b})
& = &
-
2\frac{1}{nb}\sum_{i=1}^n\sum_{j=1}^b  k(u_j, y_i) 
+
\frac{1}{b(b-1)}\sum_{j=1}^b\sum_{j'\neq j}^b
k(u_j, u_{j'}),
\label{eq:mmd-loss-mn_unbiased}
\end{IEEEeqnarray}
which allows us the derivation of an unbiased estimator of the gradient.
Following the discussion in Section \ref{sec:constructing_grad}, we now derive $\varphi_{b, n}(\theta)$  for this loss and provide conditions for the kernel function and generator such that computational bounds can be obtained. 

\begin{lemma}\label{lem:MMD_indirect}
 For the MMD-loss in \eqref{eq:mmd-loss-mn_unbiased} applied to observations $y_{1:n}$ with $y_i\in\mathbb{R}^p$, taking 
\begin{align}
&\varphi_{b, n}(\theta)=\sum_{k=1}^b\sum_{l=1}^p\left[\frac{\partial}{\partial u_{kl}} \MD'^k_{b,n}(\theta,u_{1:b})\right]_{\U = G_{\theta}(v_{1:b})}\frac{\partial}{\partial \theta}G_{\theta,l}(v_{k})\nonumber
\end{align} 
where $u_k = G_{\theta}(v) = \{G_{\theta,1}(v_{k}), \ldots, G_{\theta,p}(v_{k})\}\in\mathbb{R}^p$ and $v_{1:b} \sim p_v$ satisfies \eqref{eq:needed}. If the kernel $k(\cdot, \cdot)$ and generator $G_{\theta}(\cdot)$ are such that there exist $Q_{j}: \mathbb{R}^p\times\Theta\mapsto \mathbb{R}$ and $R_{j}:\Theta\mapsto \mathbb{R}$ for $j=1,\dots d$ so that
\begin{IEEEeqnarray}{rCl}
    Q_{j}(y_i, \theta)&\geq&\left|\sum_{l=1}^p\left[\frac{\partial}{\partial u_l} k(y_i, u)\right]_{u = G_{\theta}(v)}\frac{\partial}{\partial \theta_j}G_{\theta,l}(v)\right|\nonumber\\
    R_{j}(\theta) &\geq& \left|\sum_{l=1}^p\left[\frac{\partial}{\partial u_l} k(u, G_{\theta}(v^{\prime}))\right]_{u = G_{\theta}(v)}\frac{\partial}{\partial \theta_j}G_{\theta,l}(v) + \left[\frac{\partial}{\partial u^{\prime}_l}k(G_{\theta}(v), u^{\prime})\right]_{u^{\prime} = G_{\theta}(v^{\prime})}\frac{\partial}{\partial \theta_j}G_{\theta,l}(v^{\prime})\right|,\nonumber
\end{IEEEeqnarray}
for all $v, v^{\prime}$ in the support of $p_v$, then 
\begin{align}
        &\hat{\lambda}_{j}(\theta,\nu) \leq   \left|\frac{\partial}{\partial \theta}\log(\pi(\theta))\right| + \omega\frac{2}{n}\sum_{i=1}^n Q_j(y_i, \theta) + \omega R_{j}(\theta).\nonumber
\end{align}
\end{lemma}

\begin{proof}
As we sample $u_{1:b}\sim p_{\theta}$ by first sampling $v_{1:b} \sim p_v$ and then setting $u_{1:b} = G_{\theta}(v_{1:b})$, 
\begin{IEEEeqnarray}{rCl}
    \MD_n^k(\theta) &=& \mathbb{E}_{u_{1:b}\sim P_{\theta}}\left[\MD'^k_{b,n}(\theta, u_{1:b})\right] = \mathbb{E}_{v_{1:b}\sim P_{v}}\left[\MD'^k_{b,n}(\theta, G_{\theta}(v_{1:b}))\right].\nonumber
\end{IEEEeqnarray}
The multivariate chain rule and the fact that $\MD'^k_{b,n}$ depends on $\theta$ only through $\U$ demonstrates that
\begin{flalign*}
\varphi_{b, n}(\theta)=\sum_{k=1}^b\sum_{l=1}^p\left[\frac{\partial}{\partial u_{kl}} \MD'^k_{b,n}(\theta,u_{1:b})\right]_{\U = G_{\theta}(v_{1:b})}\frac{\partial}{\partial \theta}G_{\theta,l}(v_{k})\nonumber
\end{flalign*}
with $v_{1:b} \sim p_v$ satisfies \eqref{eq:needed}. 
The multivariate chain rule is required as $\MD'^k_{b,n}(\theta,u_{1:b})$ is a function of $u_1, \ldots, u_b$ with each $u_k$ being $p$-dimensional. Further, by the definition of $\MD'^k_{b,n}$ 
\begin{IEEEeqnarray}{rCL}
    \varphi_{b, n}(\theta) & = & -\frac{2}{bn}\sum_{i=1}^n\sum_{k=1}^b\sum_{l=1}^p\left[\frac{\partial}{\partial u_{kl}} k(y_i, u_k)\right]_{u_{k} = G_{\theta}(v_{k})}\frac{\partial}{\partial \theta}G_{\theta,l}(v_{k})\nonumber\\
    && + \frac{1}{b(b-1)}\sum_{k=1}^b\sum_{k^{\prime}\neq k}\sum_{l=1}^p \left\{\frac{\partial}{\partial u_{kl}} \left[k(u_k, G_{\theta}(v_{k^{\prime}}))\right]_{u_k = G_{\theta}(v_{k})}\frac{\partial}{\partial \theta}G_{\theta,l}(v_{k})\right.\nonumber\\
    && \qquad\left.+ \left[\frac{\partial}{\partial u_{k^{\prime}l}}k(G_{\theta}(v_{k}), u_{k^{\prime}})\right]_{u_{k^{\prime}} = G_{\theta}(v_{k^{\prime}})}\frac{\partial}{\partial \theta}G_{\theta,l}(v_{k^{\prime}})\right\}\nonumber.
\end{IEEEeqnarray}
Therefore, assuming for $j = 1,\ldots, d$ that $Q_{j}(y_i, \theta) \geq \left|\sum_{l=1}^p\left[\frac{\partial}{\partial u_l} k(y_i, u)\right]_{u = G_{\theta}(v)}\frac{\partial}{\partial \theta_j}G_{\theta,l}(v)\right|$ for all $v$ and 
\begin{IEEEeqnarray}{rCl}
    R_{j}(\theta) &\geq& \left|\sum_{l=1}^p\left[\frac{\partial}{\partial u_l} k(u, G_{\theta}(v^{\prime}))\right]_{u = G_{\theta}(v)}\frac{\partial}{\partial \theta_j}G_{\theta,l}(v) + \left[\frac{\partial}{\partial u^{\prime}_l}k(G_{\theta}(v), u^{\prime})\right]_{u^{\prime} = G_{\theta}(v^{\prime})}\frac{\partial}{\partial \theta_j}G_{\theta,l}(v^{\prime})\right|,\nonumber
\end{IEEEeqnarray}
for all $v$ and $v^{\prime}$, 
is sufficient to construct computational bounds bounds $\hat{\Lambda}_{j}(\theta,\nu)$ satisfying \eqref{eq:computational-bound-estimated} as 
\begin{align}
    \hat{\lambda}_{j}(\theta,\nu) &= \max\left\{\nu_j\left(-\frac{\partial}{\partial \theta}\log(\pi(\theta_j)) + \omega\varphi_{b,n}(\theta)\right), 0\right\}\nonumber\\
    &\leq \left|\frac{\partial}{\partial \theta_j}\log(\pi(\theta))\right| + \omega\frac{2}{n}\sum_{i=1}^n Q(y_i, \theta) + \omega R_{j}(\theta),\nonumber
\end{align}
which completes the proof.
\end{proof}

The following corollary exemplifies the conditions of Lemma \ref{lem:MMD_indirect} for the  radial basis function (RBF) kernel.

\begin{corollary}    \label{lem:MMDRBF_indirect}
    For the radial basis function (RBF) kernel with hyperparameter $\gamma$ and $y_i\in \mathbb{R}^p$, the computational bounds of Lemma \ref{lem:MMD_indirect} require the generator $G_{\theta}(\cdot)$ to admit functions $Q_j(y, \theta)$ and $R_j(\theta)$ for $j=1,\dots d$ so that for all $v, v^{\prime}$ in the support of $p_v$,
    \begin{IEEEeqnarray}{rCl}
        %Q_{j}(y_i, \theta)&\geq&\left|\frac{(y-G_{\theta}(v))}{\sqrt{2\pi}\gamma^{3/2}}\exp\left(-\frac{(y-G_{\theta}(v))^2}{2\gamma}\right)\frac{\partial}{\partial \theta_j}G_{\theta}(v)\right|\nonumber\\
        Q_{j}(y_i, \theta)&\geq&\left|\exp\left(-\sum_{l=1}^p\frac{(y_{il}-G_{\theta,l}(v))^2}{2\gamma}\right)\sum_{l=1}^p\frac{(y_{il}-G_{\theta,l}(v))}{(2\pi)^{p/2}\gamma^{p/2+1}}\frac{\partial}{\partial \theta_j}G_{\theta,l}(v)\right|\nonumber\\
        %R_{j}(\theta) &\geq& \left|\frac{(G_{\theta}(v)-G_{\theta}(v^{\prime}))}{\sqrt{2\pi}\gamma^{3/2}}\exp\left(-\frac{(G_{\theta}(v)-G_{\theta}(v^{\prime}))^2}{2\gamma}\right)\left(\frac{\partial}{\partial \theta_j}G_{\theta}(v) - \frac{\partial}{\partial \theta_j}G_{\theta}(v^{\prime})\right)\right|.\nonumber\\
        R_{j}(\theta) &\geq& \left|\exp\left(-\sum_{l=1}^p\frac{(G_{\theta,l}(v)-G_{\theta,l}(v^{\prime}))^2}{2\gamma}\right)\right.\nonumber\\
        &&\cdot\left.\sum_{l=1}^p\frac{(G_{\theta,l}(v)-G_{\theta,l}(v^{\prime}))}{(2\pi)^{p/2}\gamma^{p/2+1}}\left(\frac{\partial}{\partial \theta_j}G_{\theta,l}(v) - \frac{\partial}{\partial \theta_j}G_{\theta,l}(v^{\prime})\right)\right|.\nonumber
\end{IEEEeqnarray}
\end{corollary}

\begin{proof}
    When using the RBF kernel $k(y, u) = \frac{1}{(2\pi)^{p/2}\gamma^{p/2}}\exp\left(-\sum_{l=1}^p\frac{(y_l-u_l)^2}{2\gamma}\right)$ and therefore,
    \begin{IEEEeqnarray}{rCl}
        \frac{\partial}{\partial u_l}k(y, u) = \frac{(y_l-u_l)}{(2\pi)^{p/2}\gamma^{p/2+1}}\exp\left(-\sum_{l=1}^p\frac{(y_l-u_l)^2}{2\gamma}\right),\nonumber
    \end{IEEEeqnarray}
    and $\frac{\partial}{\partial u_l}k(y, u) = -\frac{\partial}{\partial y_l}k(y, u)$. As a result,
    \begin{IEEEeqnarray}{rCL}
        \frac{\partial}{\partial u_l} k(y, u)\Bigr|_{u = G_{\theta}(v)}\frac{\partial}{\partial \theta_j}G_{\theta,l}(v) =  \frac{(y_l-G_{\theta,l}(v))}{(2\pi)^{p/2}\gamma^{p/2+1}}\exp\left(-\sum_{l=1}^p\frac{(y_l-G_{\theta,l}(v))^2}{2\gamma}\right)\frac{\partial}{\partial \theta_j}G_{\theta,l}(v),\nonumber
    \end{IEEEeqnarray}
    and 
    \begin{IEEEeqnarray}{rCl}
    &&\frac{\partial}{\partial u_l} k(u, G_{\theta}(v^{\prime}))\Bigr|_{u = G_{\theta}(v)}\frac{\partial}{\partial \theta_j}G_{\theta,l}(v) + \frac{\partial}{\partial u^{\prime}_l}k(G_{\theta}(v), u^{\prime})\Bigr|_{u^{\prime} = G_{\theta}(v^{\prime})}\frac{\partial}{\partial \theta_j}G_{\theta,l}(v^{\prime})\nonumber\\
    &=&-\frac{(G_{\theta,l}(v)-G_{\theta,l}(v^{\prime}))}{(2\pi)^{p/2}\gamma^{p/2+1}}\exp\left(-\sum_{l=1}^p\frac{(G_{\theta,l}(v)-G_{\theta,l}(v^{\prime}))^2}{2\gamma}\right)\left(\frac{\partial}{\partial \theta_j}G_{\theta,l}(v) - \frac{\partial}{\partial \theta_j}G_{\theta,l}(v^{\prime})\right)\nonumber
\end{IEEEeqnarray}
    which by Lemma \ref{lem:MMD_indirect} provides the required result.
\end{proof}

Before providing specific computational bounds for the applications in Section \ref{sec:examples}, we prove \Cref{corr:MMD_bounds} which provides sufficient conditions on the generator $G_{\theta}(v)$ for the existence of $Q_j$ and $R_j$ in \Cref{lem:MMDRBF_indirect}. 

%\begin{corollary}    \label{lem:MMDRBF_indirect_general}
%    For the radial basis function (RBF) kernel with hyperparameter $\gamma$ and $y_i\in \mathbb{R}^p$, the existence of bounding functions $Q_j(y, \theta)$ and $R_j(\theta)$, for $j=1,\dots d$, of Corollary \ref{lem:MMDRBF_indirect} is guaranteed providing the generator $G_{\theta}(\cdot)$ is such that
%    \begin{IEEEeqnarray}{rCl}
%    \left|\frac{\partial}{\partial \theta_j}G_{\theta,l}(v)\right| \leq O\left(\frac{1}{G_{\theta,l}(v)}\exp\left\{\frac{\partial}{\partial \theta_j}G_{\theta,l}(v)^2\right\}\right)
%\end{IEEEeqnarray}
%as $G_{\theta,l}(v)\rightarrow\pm \infty$ for all $j = 1, \ldots, d$, $l = 1\,\ldots, p$
%\end{corollary}

\begin{proof}

\Cref{lem:MMD_indirect} provided $\varphi_{b,n}(\theta)$ satisfying \eqref{eq:needed} for $\MD_{n}(\theta)=\MMD^2(P_\theta,\mathbb{P}_n)$ and showed that the existence of bounding functions $Q_j(y, \theta)$ and $R_j(\theta)$, for $j=1,\dots d$  was sufficient to find a $\hat{\Lambda}$ satisfying \eqref{eq:computational-bound-estimated}. These are the two requirements for Algorithm \ref{Alg:ZigZag_EstimatedLoss} to sample from $\pi(\theta\mid\MD_n)$. 
Corollary \ref{lem:MMDRBF_indirect} simplifies $Q_j$ and $R_j$ under the radial basis (RBF) kernel with hyperparameter $\gamma$, and 
therefore, proving \Cref{corr:MMD_bounds} requires us to show that generator $G_{\theta}(\cdot) = (G_{\theta,1}(\cdot), \ldots, G_{\theta,p}(\cdot))$ satisfying    
%\begin{IEEEeqnarray}{rCl}
$    \left|\frac{\partial}{\partial \theta_j}G_{\theta,l}(v)\right| \leq O\left(\frac{1}{G_{\theta,l}(v)}\exp\left\{\frac{\partial}{\partial \theta_j}G_{\theta,l}(v)^2\right\}\right)$ for $G_{\theta,l}(v)\rightarrow\pm \infty$ 
%\end{IEEEeqnarray}
and $j = 1, \ldots, d$, $l = 1\,\ldots, p$ guarantees the existence of $Q_j(y, \theta)$ and $R_j(\theta)$, for $j=1,\dots d$ in Corollary \ref{lem:MMDRBF_indirect}.

The proof of Corollary \ref{lem:MMDRBF_indirect} established that 
\begin{IEEEeqnarray}{rCl}
    \frac{\partial}{\partial u_l}k(y, u) = \frac{(y_l-u_l)}{(2\pi)^{p/2}\gamma^{p/2+1}}\exp\left(-\sum_{l=1}^p\frac{(y_l-u_l)^2}{2\gamma}\right)  %\lesssim \left|u_l\exp\left\{-u_l^2\right\}\right|\nonumber
    = O\left(\left|u_l\right|\exp\left\{-u_l^2\right\}\right),
\end{IEEEeqnarray}
as $u_l\rightarrow \pm \infty$. As a result, for fixed $y_i$ 
\begin{IEEEeqnarray}{rCl}
    &&\left|\exp\left(-\sum_{l=1}^p\frac{(y_{il}-G_{\theta,l}(v))^2}{2\gamma}\right)\sum_{l=1}^p\frac{(y_{il}-G_{\theta,l}(v))}{(2\pi)^{p/2}\gamma^{p/2+1}}\frac{\partial}{\partial \theta_j}G_{\theta,l}(v)\right|\label{eq:term1}\\
    %&\lesssim& \left|G_{\theta,l}(v)\exp\left\{-G_{\theta,l}(v)^2\right\}\right|\left|\frac{\partial}{\partial \theta_j}G_{\theta,l}(v)\right|\\
    %&\lesssim& 1\nonumber
    &=& O\left(\left|G_{\theta,l}(v)\right|\exp\left\{-G_{\theta,l}(v)^2\right\}\right)\left|\frac{\partial}{\partial \theta_j}G_{\theta,l}(v)\right|\nonumber\\
    &=& O(1)\nonumber
\end{IEEEeqnarray}
as $u_l\rightarrow \pm \infty$, 
under the condition that $\left|\frac{\partial}{\partial \theta_j}G_{\theta,l}(v)\right| \leq O\left(\left|\frac{1}{G_{\theta,l}(v)}\right|\exp\left\{\frac{\partial}{\partial \theta_j}G_{\theta,l}(v)^2\right\}\right)$
%$\left|\frac{\partial}{\partial \theta_j}G_{\theta,l}(v)\right| \lesssim \frac{1}{G_{\theta,l}(v)}\exp\left\{\frac{\partial}{\partial \theta_j}G_{\theta,l}(v)^2\right\}$ 
as $u_l\rightarrow \pm \infty$. 
This shows that \eqref{eq:term1} is bounded away from $\infty$ which is sufficient for the existence of $Q_{j}(y_i, \theta)$ in Corollary \ref{lem:MMDRBF_indirect}. 

Further, for fixed $v^{\prime}$
\begin{IEEEeqnarray}{rCl}
    &&\left|\exp\left(-\sum_{l=1}^p\frac{(G_{\theta,l}(v)-G_{\theta,l}(v^{\prime}))^2}{2\gamma}\right)\sum_{l=1}^p\frac{G_{\theta,l}(v)-G_{\theta,l}(v^{\prime})}{(2\pi)^{p/2}\gamma^{p/2+1}}\left(\frac{\partial}{\partial \theta_j}G_{\theta,l}(v) - \frac{\partial}{\partial \theta_j}G_{\theta,l}(v^{\prime})\right)\right|\label{eq:term2}\\
    &&= O\left(\left|G_{\theta,l}(v)\right|\exp\left\{-G_{\theta,l}(v)^2\right\}\right)\left|\frac{\partial}{\partial \theta_j}G_{\theta,l}(v)\right| = O(1)\nonumber
\end{IEEEeqnarray}
as $u_l\rightarrow \pm \infty$, 
under the condition that  $\left|\frac{\partial}{\partial \theta_j}G_{\theta,l}(v)\right| \leq O\left(\left|\frac{1}{G_{\theta,l}(v)}\right|\exp\left\{\frac{\partial}{\partial \theta_j}G_{\theta,l}(v)^2\right\}\right)$, 
as $u_l\rightarrow \pm \infty$.
This shows that \eqref{eq:term2} is bounded away from $\infty$, 
which is sufficient for the existence of $R_{j}(\theta)$ from Corollary  \ref{lem:MMDRBF_indirect}.

\end{proof}

Corollary \ref{corr:MMD_bounds} identifies that unless the gradient of the generator in its parameters $\theta$, explodes like the exponential of the generator squared. Then computational bounds via functions $Q_{j}(y_i, \theta)$ and $R_{j}(\theta)$ exist.

In the knowledge that such bounds exist, we propose a default bound that can be used in situations where finding  explicit bounds is complicated. In particular, we propose an approximate bound and check that it is never exceeded during sampling by tracking the Poisson thinning acceptance probabilities. We use the fact that 
\begin{IEEEeqnarray}{rCl}           \left|\frac{\partial}{\partial u_l}k(y, u)\right| \leq \frac{1}{(2\pi)^{p/2}\gamma^{(p+1)/2}}\exp\left\{\frac{1}{2}\right\}, \nonumber
\end{IEEEeqnarray}
and assume $\left|\frac{\partial}{\partial \theta_j}G_{\theta,l}(v)\right| \leq Z_j$ which provides 
    \begin{align}
        Q_j(y_i, \theta; x_i) &\leq \frac{1}{(2\pi)^{p/2}\gamma^{(p+1)/2}}\exp\left\{\frac{1}{2}\right\}Z_j;\nonumber\\
        R_j(\theta; x_i) &\leq \frac{1}{(2\pi)^{p/2}\gamma^{(p+1)/2}}\exp\left\{\frac{1}{2}\right\}Z_j.\nonumber
    \end{align}
We use this strategy in both the $g$-and-$k$ and M/G/1 queuing examples in Section \ref{sec:examples}.

We now derive explicit bounding functions satisfying Lemma \ref{lem:MMD_indirect} for the Gaussian regression and Gaussian copula examples of Sections \ref{sec:examples}. 

\begin{corollary}
    \label{cor:regression-MMD}
    Consider a Gaussian regression model with $p_{\theta}(y_i; x_i) = \mathcal{N}(y; x_i^{\top}\beta, \sigma^2)$ where $\theta = \{\beta, \log(\sigma)\}\in\mathbb{R}^{d+1}$, $x_i\in\mathbb{R}^d$, and $y_i \in \mathbb{R}$. Under the RBF kernel with hyperparameter $\gamma$, functions satisfying Lemma \ref{lem:MMD_indirect} are
    \begin{align}
        Q_j(y_i, \theta; x_i) &= \frac{1}{\sqrt{2\pi}\gamma}\exp\left(-\frac{1}{2}\right)\left|x_{ij}\right|;\nonumber\\
        R_j(\theta; x_i) = 0;\nonumber
    \end{align}
    for $j = 1,\ldots, d$ and
    \begin{align}
        Q_{d+1}(y_i, \theta; x_i) &= \max_{z\in\mathbb{R}}\left|\frac{(-z^2 + z(y_i+x_j^\top\beta) - y_ix_j^\top\beta) }{\sqrt{2\pi}\gamma^{3/2}}\exp\left(-\frac{(y_i-z)^2}{2\gamma}\right)\right|;\nonumber\\
        R_{d+1}(\theta; x_i) &= \frac{2}{\sqrt{2\pi\gamma}}\exp(-1).\nonumber
    \end{align}
\end{corollary}

\begin{proof}
    This result concerns the RBF kernel for univariate observations so it is sufficient to satisfy the conditions of Corollary \ref{lem:MMDRBF_indirect} with $p = 1$.

    Sampling $\U \sim p_{\theta}(\cdot; x) $ can be achieved by first sampling $v_{1:b}$ with $v_k = \{v_{k1}, v_{k2}\} \sim p_v$ and $p_v(v) = \operatorname{unif}(v_1; 0, 1)\operatorname{unif}(v_2; 0, 1)$, and setting 
    \begin{align}
        u_k = G_{\theta}(v_k; x) &= \beta^Tx + \exp(\log(\sigma))\sqrt{-2\log(v_{k1})}\cos(2\pi v_{k2}).\nonumber
    \end{align}
    For the regression coefficients $j = 1,\ldots, p$    
    \begin{align}
        \frac{\partial}{\partial\beta_j}G_{\theta}(v; x) &= x_j,~ j = 1\ldots. p,\nonumber
    \end{align}
    Therefore, 
    \begin{IEEEeqnarray}{rCl}
        &&
        \left|\frac{(y-G_{\theta}(v;x))}{\sqrt{2\pi}\gamma^{3/2}}\exp\left(-\frac{(y-G_{\theta}(v;x))^2}{2\gamma}\right)\frac{\partial}{\partial \beta_j}G_{\theta}(v;x)\right| \nonumber\\
        &=& \left|\frac{(y-G_{\theta}(v;x))}{\sqrt{2\pi}\gamma^{3/2}}\exp\left(-\frac{(y-G_{\theta}(v;x))^2}{2\gamma}\right)x_j\right|\nonumber\\
        &\leq& \frac{1}{\sqrt{2\pi}\gamma}\exp\left(-\frac{1}{2}\right)\left|x_j\right|,\nonumber
    \end{IEEEeqnarray}
    and
    \begin{align}
        \frac{(G_{\theta}(v;x)-G_{\theta}(v^{\prime};x))}{\sqrt{2\pi}\gamma^{3/2}}\exp\left(-\frac{(G_{\theta}(v;x)-G_{\theta}(v^{\prime};x))^2}{2\gamma}\right)\left(\frac{\partial}{\partial \beta_j}G_{\theta}(v;x) - \frac{\partial}{\partial \beta_j}G_{\theta}(v^{\prime};x)\right) = 0.\nonumber
    \end{align}
    as required. For the residual variance $j = d+1$
    \begin{align}
        \frac{\partial}{\partial\log(\sigma)}G_{\theta}(v; x) &= \exp(\log(\sigma))\sqrt{-2\log(v_{1})}\cos(2\pi v_{2})=G_{\theta}(v; x) - x_j^\top\beta,\nonumber
    \end{align}
    so that 
    \begin{IEEEeqnarray}{rCl}
        &&\left|\frac{(y-G_{\theta}(v;x))}{\sqrt{2\pi}\gamma^{3/2}}\exp\left(-\frac{(y-G_{\theta}(v;x))^2}{2\gamma}\right)\frac{\partial}{\partial \log(\sigma)}G_{\theta}(v;x)\right|\nonumber\\
        &=&\left| \frac{(y-G_{\theta}(v;x))}{\sqrt{2\pi}\gamma^{3/2}}\exp\left(-\frac{(y-G_{\theta}(v;x))^2}{2\gamma}\right)\left(G_{\theta}(v; x) - x_j^\top\beta\right)\right|\nonumber\\
        &=&\left| \frac{(-G_{\theta}(v;x)^2 + G_{\theta}(v;x)(y+x_j^\top\beta) - yx_j^\top\beta) }{\sqrt{2\pi}\gamma^{3/2}}\exp\left(-\frac{(y-G_{\theta}(v;x))^2}{2\gamma}\right)\right|\nonumber\\
        &\leq& \max_{z\in\mathbb{R}}\left|\frac{(-z^2 + z(y+x_j^\top\beta) - yx_j^\top\beta) }{\sqrt{2\pi}\gamma^{3/2}}\exp\left(-\frac{(y-z)^2}{2\gamma}\right)\right|\nonumber
    \end{IEEEeqnarray}
    and
    \begin{IEEEeqnarray}{rCl}
        &&\bigg|\frac{(G_{\theta}(v;x)-G_{\theta}(v^{\prime};x))}{\sqrt{2\pi}\gamma^{3/2}}\exp\left(-\frac{(G_{\theta}(v;x)-G_{\theta}(v^{\prime};x))^2}{2\gamma}\right)
        \times 
        \nonumber \\&& \quad \left(\frac{\partial}{\partial \log(\sigma)}G_{\theta}(v;x) - \frac{\partial}{\partial \log(\sigma)}G_{\theta}(v^{\prime};x)\right) \bigg|\nonumber\\
        &=&\left|\frac{(G_{\theta}(v;x)-G_{\theta}(v^{\prime};x))^2}{\sqrt{2\pi}\gamma^{3/2}}\exp\left(-\frac{(G_{\theta}(v;x)-G_{\theta}(v^{\prime};x))^2}{2\gamma}\right)\right|\nonumber\\
        &\leq&\frac{2}{\sqrt{2\pi\gamma}}\exp(-1)\nonumber
    \end{IEEEeqnarray}
    as required. The final step is to establish that $Q_{d+1}(y_i, \theta; x_i)$ is finite for all $y_1, x_i$ and $\theta$. A continuous function on a bounded domain is bounded, and therefore it is sufficient to check that $\left|\frac{(-z^2 + z(y_i+x_j^\top\beta) - y_ix_j^\top\beta) }{\sqrt{2\pi}\gamma^{3/2}}\exp\left(-\frac{(y_i-z)^2}{2\gamma}\right)\right|$ does not diverge as $z\rightarrow\pm\infty$. As $z\rightarrow\pm\infty$, for fixed $y_i$, $x_i$ and all $\beta$
    \begin{IEEEeqnarray}{rCl}
        \left|\frac{(-z^2 + z(y_i+x_j^\top\beta) - y_ix_j^\top\beta) }{\sqrt{2\pi}\gamma^{3/2}}\exp\left(-\frac{(y_i-z)^2}{2\gamma}\right)\right| = O\left(z^2\exp\left(z^2\right)\right)\nonumber
    \end{IEEEeqnarray}
    which satisfies the requirement.
\end{proof}

To show that the term in question is bounded in the copula example, define the following functions
$$
V(z_i,z_j,\theta)=\rho(\theta)z_i+\sqrt{1-\rho(\theta)^2}z_j,\;b(\theta)=\frac{\rho(\theta)}{\sqrt{1 - \rho(\theta)}}.
$$

\begin{corollary}
    \label{cor:MMD-copula}
    Let $p_{\theta}$ be the density of a bi-variate Gaussian copula with unit marginal variances and correlation $\rho(\theta) = \frac{2}{1+\exp(-\theta)} - 1\in(-1, 1)$, and take $y_i\in\mathbb{R}^2$. Let $\hat{u}_i = \{F_{n,1}(y_{i1}), F_{n,2}(y_{i2})\}$ with $F_{n,l}(t) = \frac{1}{n}\sum_{i=1}^n\1\{y_{il}\le t\}$ for $l=1,2$. 
    For kernel $\tilde{k}: [0,1]^2\times[0,1]^2\mapsto \mathbb{R}$ so that $\tilde{k}(\hat{u}, u) = k(\Phi^{-1}(\hat{u}), \Phi^{-1}(u))$ where $k(\cdot, \cdot)$ is the RBF kernel  with hyperparameter $\gamma$, and $\Phi^{-1}$ is the inverse standard Gaussian cumulative distribution function, Lemma \ref{lem:MMD_indirect} is satisfied by taking $Q_l(y_i, \theta) = Q(\hat{u}_i, \theta)$ and $R_l(\theta) = R(\theta)$ for $l=1,2$, where
%Jack's version
%\begin{IEEEeqnarray}{rCl}
%&&Q(\hat{u}_i, \theta) =\nonumber\\
%&&\max_{z_1, z_2\in\mathbb{R}} \left|\frac{1}{2\pi\gamma^2}\exp\left(-\frac{\{\Phi^{-1}(\hat{u}_{i1}) - z_1)\}^2}{2\gamma}\right)\exp\left(-\frac{(\Phi^{-1}(\hat{u}_{i2}) - V(z_1,z_2,\theta))^2}{2\gamma}\right)\right.\nonumber\\
%&&\quad\left.\times (\Phi^{-1}(\hat{u}_{i2}) - V(z_1,z_2,\theta))\left(z_1 - b(\theta)z_2\right)\frac{2\exp(-\theta)}{(1+\exp(-\theta))^2}\right|  \nonumber \\
%&&R(\theta) = \nonumber\\
%&&\max_{z_1, z_2, z_3, z_4\in\mathbb{R}} \bigg{|}\exp\left(-\frac{\{V(z_1,z_2,\theta) - V(z_3,z_4,\theta)\}^2}{2\gamma}\right)\times \frac{1}{2\pi\gamma^2}\left\%{V(z_1,z_2,\theta) - V(z_3,z_4,\theta)\right\}\nonumber\\
%&&\quad\quad\times \left\{\left(z_1 - b(\theta)z_2\right) - \left(z_3 - b(\theta)z_4\right)\right\}\frac{2\exp(-\theta)}{(1+\exp(-\theta))^2}\exp\left(-\frac{(z_1 - z_3)^2}{2\gamma}\right)\bigg{|}.\nonumber
%\end{IEEEeqnarray}
% David's version
\begin{IEEEeqnarray}{rCl}
&&Q(\hat{u}_i, \theta) =\nonumber\\
&&\max_{z_1, z_2\in\mathbb{R}} \frac{1}{2\pi\gamma^2}\exp\left(-\frac{\{\Phi^{-1}(\hat{u}_{i1}) - z_1\}^2+(\Phi^{-1}(\hat{u}_{i2}) - V(z_1,z_2,\theta))^2}{2\gamma}\right)\frac{2\exp(-\theta)}{(1+\exp(-\theta))^2}\nonumber\\
&&\quad\times |\Phi^{-1}(\hat{u}_{i2}) - V(z_1,z_2,\theta))|\times|z_1 - b(\theta)z_2| \nonumber \\
&&R(\theta) = \nonumber\\
&&\max_{z_1, z_2, z_3, z_4\in\mathbb{R}} \frac{1}{2\pi\gamma^2}\exp\left(-\frac{\{V(z_1,z_2,\theta) - V(z_3,z_4,\theta)\}^2+(z_1 - z_3)^2}{2\gamma}\right)\frac{2\exp(-\theta)}{(1+\exp(-\theta))^2}\nonumber\\&&\quad \times|V(z_1,z_2,\theta) - V(z_3,z_4,\theta)|\times |\left(z_1 - b(\theta)z_2\right) - \left(z_3 - b(\theta)z_4\right)|.\nonumber
\end{IEEEeqnarray}

\end{corollary}

\begin{proof}
Sampling $u_{1:b} \sim p_{\theta}$ from a bivariate Gaussian copula can be achieved by first sampling $v_{1:b}$ with $v_k = \{v_{k1}, v_{k2}\}\sim p_v$ and $p_v(v) = \mathcal{N}(v_{1}; 0, 1)\mathcal{N}(v_{2}; 0, 1)$ and setting $u_k = \{u_{k1}, u_{k2}\} = \{G_{\theta,1}(v_k), G_{\theta,2}(v_k)\}$ with
\begin{align}
    G_{\theta,1}(v_k) &= \Phi(v_{k1}),~ G_{\theta,2}(v_k) = \Phi(\rho(\theta) v_{k1} + \sqrt{1-\rho(\theta)^2}v_{k2}).\nonumber
\end{align}
where $\Phi$ is the standard Gaussian cumulative distribution function.

For this example, as $\Phi^{-1}(\Phi(x)) = x$ it is convenient to write 
\begin{IEEEeqnarray}{rCl}
&&\tilde{k}(\hat{u}_i, G_{\theta}(v_k)) =\nonumber\\
&&\frac{1}{2\pi\gamma^2}\exp\left(-\frac{(\Phi^{-1}(\hat{u}_{i1}) - v_{k1}))^2}{2\gamma}\right)\exp\left(-\frac{(\Phi^{-1}(\hat{u}_{i2}) - (\rho(\theta)v_{k1} + \sqrt{1-\rho(\theta)^2}v_{k2})))^2}{2\gamma}\right)\nonumber
\end{IEEEeqnarray}
and to seek to directly bound $\frac{\partial}{\partial\theta}\tilde{k}(\hat{u}_i, G_{\theta}(v_k))$. Proceeding as such,
\begin{align*}
&\frac{\partial}{\partial\theta}\tilde{k}(\hat{u}_i, G_{\theta}(v_k))\\
=&\frac{1}{2\pi\gamma^2}\exp\left(-\frac{(\Phi^{-1}(\hat{u}_{i1}) - v_{k1})^2}{2\gamma}\right)\exp\left(-\frac{(\Phi^{-1}(\hat{u}_{i2}) - (\rho(\theta)v_{k1} + \sqrt{1-\rho(\theta)^2}v_{k2})))^2}{2\gamma}\right)\\
&\cdot(\Phi^{-1}(\hat{u}_{i2}) - (\rho(\theta)v_{k1} + \sqrt{1-\rho(\theta)^2}v_{k2})))\left(v_{k1} - b(\theta)v_{k2}\right)\frac{2\exp(-\theta)}{(1+\exp(-\theta))^2}\\
\leq&\max_{z_1, z_2\in\mathbb{R}}\frac{1}{2\pi\gamma^2}\exp\left(-\frac{(\Phi^{-1}(\hat{u}_{i1}) - z_1))^2}{2\gamma}\right)\exp\left(-\frac{(\Phi^{-1}(\hat{u}_{i2}) - V(z_1,z_2,\theta)))^2}{2\gamma}\right)\frac{2\exp(-\theta)}{(1+\exp(-\theta))^2}\nonumber\\
&\quad \times|\Phi^{-1}(\hat{u}_{i2}) - V(z_1,z_2,\theta)|\times|z_1 - b(\theta)z_2| .\nonumber
\end{align*}

Similarly,
\begin{IEEEeqnarray}{rCl}
&&\tilde{k}(G_{\theta}(v_k), G_{\theta}(v_{k^{\prime}})) = \nonumber\\
&&\frac{1}{2\pi\gamma}\exp\left(-\frac{((\rho(\theta)v_{k1} + \sqrt{1-\rho(\theta)^2}v_{k2}) - (\rho(\theta)v_{k^{\prime}1} + \sqrt{1-\rho(\theta)^2}v_{k^{\prime}2}))^2}{2\gamma}\right)\nonumber\\
&&\cdot\exp\left(-\frac{(v_{k1} - v_{k^{\prime}1})^2}{2\gamma}\right)\nonumber
\end{IEEEeqnarray}
and therefore
\begin{align*}
&\frac{\partial}{\partial\theta}\tilde{k}(G_{\theta}(v_k), G_{\theta}(v_{k^{\prime}}))\\
=&-\frac{1}{2\pi\gamma^2}\exp\left(-\frac{((\rho(\theta)v_{k1} + \sqrt{1-\rho(\theta)^2}v_{k2}) - (\rho(\theta)v_{k^{\prime}1} + \sqrt{1-\rho(\theta)^2}v_{k^{\prime}2}))^2}{2\gamma}\right)\\
&\cdot (\rho(\theta)v_{k1} + \sqrt{1-\rho(\theta)^2}v_{k2}) - (\rho(\theta)v_{k^{\prime}1} + \sqrt{1-\rho(\theta)^2}v_{k^{\prime}2})\\
&\cdot \left(\left(v_{k1} - b(\theta)v_{k2}\right) - \left(v_{k^{\prime}1} - b(\theta)v_{k^{\prime}2}\right)\right)\frac{2\exp(-\theta)}{(1+\exp(-\theta))^2}\\
&\cdot\exp\left(-\frac{(v_{k1} - v_{k^{\prime}1})^2}{2\gamma}\right)\\
\leq&\max_{z_1, z_2, z_3, z_4\in\mathbb{R}} \exp\left(-\frac{(V(z_1,z_2,\theta) - V(z_3,z_4,\theta))^2}{2\gamma}\right)\exp\left(-\frac{(z_1 - z_3)^2}{2\gamma}\right)\frac{2\exp(-\theta)}{(1+\exp(-\theta))^2}\nonumber\\
&\quad \cdot \frac{1}{2\pi\gamma^2}|V(z_1,z_2,\theta) - V(z_3,z_4,\theta)|\times  |\left(z_1 - b(\theta)z_2\right) - \left(z_3 - b(\theta)z_4\right)|
\end{align*}
as required

The final task is to ensure that $Q(\hat{u}_i, \theta)$ and $R(\theta)$ are finite for all $z_1,z_2,z_3,z_4$, so that their maximum exists and is finite. To this end, let $z=(z_1,\dots,z_4)^\top$ and define the functions 
\begin{IEEEeqnarray}{rCl}
&&\mathcal{Q}(z_1,z_2,\theta)= \frac{1}{2\pi\gamma^2}\exp\left(-\frac{\{\Phi^{-1}(\hat{u}_{i1}) - z_1\}^2+(\Phi^{-1}(\hat{u}_{i2}) - V(z_1,z_2,\theta))^2}{2\gamma}\right)\frac{2\exp(-\theta)}{(1+\exp(-\theta))^2}\nonumber\\
&&\quad\times |\Phi^{-1}(\hat{u}_{i2}) - V(z_1,z_2,\theta))|\times|z_1 - b(\theta)z_2| \nonumber \\
&&\mathcal{R}(z_1,z_2,z_3,z_4,\theta) =\frac{1}{2\pi\gamma^2}\exp\left(-\frac{\{V(z_1,z_2,\theta) - V(z_3,z_4,\theta)\}^2+(z_1 - z_3)^2}{2\gamma}\right)\frac{2\exp(-\theta)}{(1+\exp(-\theta))^2}\nonumber\\&&\quad \times|V(z_1,z_2,\theta) - V(z_3,z_4,\theta)|\times |\left(z_1 - b(\theta)z_2\right) - \left(z_3 - b(\theta)z_4\right)|.\nonumber
\end{IEEEeqnarray}

Since both $\mathcal{Q}(\cdot,\theta)$ and $\mathcal{R}(\cdot,\theta)$ are continuous functions, for $\rho(\theta)\in(-1,1)$, they are bounded for $z$ lying in $\mathbb{R}/\{-\infty,\infty\}$. Therefore, we must only show that these functions do not diverge as the components of $z$ diverge. We analyze these functions separately, starting with $\mathcal{Q}$. 

\subsubsection*{Case: $\mathcal{Q}(z_1,z_2,\theta)$.}
Note that, that the function $V(z_1,z_2,\theta)$ in $\mathcal{Q}$ is linear in $z_1,z_2$. Now, we proceed by cases.  
\begin{enumerate}
    \item[Case 1.]  For fixed $z_i$, if $z_j\rightarrow\pm\infty$, the  terms, $|V(z_1,z_2,\theta)||z_1-b(\theta)z_2|$ diverge at most like like $z_j^2$. However, since $V(z_1,z_2,\theta)$ appears as a square in the exponential, the order of the term becomes $z_j^2\exp\{-z^2_j\}$ which clearly goes to zero as $z_j\rightarrow\pm\infty$. Since this term goes to zero as $z_j\rightarrow\pm\infty$ it must achieve its maximum at a point on the interior of $\mathbb{R}$, and it therefore bounded. 
    \item[Case 2.] For $z_i,z_j\rightarrow\pm\infty$, with the same sign, the term $|V(z_1,z_2,\theta)||z_1-b(\theta)z_2|$ can be bounded, up to constants, by 
    $|z_1|^2|z_2|^2$. However, again, since $V(z_1,z_2,\theta)$ appears as a square in the exponential, the order of the term is at most  $z_j^4\exp\{-z^2_j\}$, which again decays to zero as $z_j\rightarrow\pm\infty$. 
    \item[Case 3.] Lastly, since $V(z_1,z_2,\theta)$ is linear, we must consider what occurs as $z_i\rightarrow\infty$, and $z_j\rightarrow-\infty$. In this case, the terms $|V(z_1,z_2,\theta)||z_1-b(\theta)z_2|$ can again be bounded, up to constants, by 
    $|z_1|^2|z_2|^2$. However, the $V(z_1,z_2,\theta)$ term in the exponential may no longer diverge. However, in this case the first term in the exponential, $-(\Phi^{-1}(\hat{u}_{i1}-z_1)^2$ will still diverge so that that $\mathcal{Q}(z_1,z_2,\theta)$ remains bounded as both $z_1,z_2$ diverge. 
\end{enumerate}

\subsubsection*{Case: $\mathcal{R}(z_1,z_2,z_3,z_4\theta)$.} 
Let us simplify notation and write 
$$
V(z_1,z_2,z_3,z_4)=V(z_1,z_2)-V(z_3,z_4)= \rho(\theta)(z_1-z_3)+\sqrt{1-\rho(\theta)^2}(z_2-z_4),
$$ so that we have 
\begin{flalign*}
  \mathcal{R}(z,\theta)\asymp& \exp\left\{-\frac{V(z_1,z_2,z_3,z_4)^2-(z_1-z_3)^2}{2\gamma}\right\}|V(z_1,z_2,z_3,z_4)|(|z_1-z_3|+|z_2-z_4|)
\end{flalign*}  
Again, this is a continuous function in $z$ so it is bounded on bounded domains. However, for some component of $z\rightarrow\pm\infty$, the dominant terms in the expression become
\begin{flalign}
  \mathcal{R}(z,\theta)
\lesssim   &\exp\left\{-\frac{V(z_1,z_2,z_3,z_4)^2+(z_1-z_3)^2}{2\gamma}\right\}|V(z_1,z_2,z_3,z_4)|(|z_1-z_3|+|z_2-z_4|)\nonumber\\\asymp&\exp\{-(z_1-z_3)^2-(z_2-z_4)^2+2|z_1-z_3||z_2-z_4|\}\exp\{-|z_1-z_3|^2\}\nonumber\\&\times\left(|z_1-z_3|^2+|z_2-z_4|^2+2|z_1-z_3||z_2-z_4|\right)\nonumber \\&=\exp\{-\left[|z_1-z_3|-|z_2-z_4|\right]^2\}\exp\{-|z_1-z_3|^2\}\left(|z_1-z_3|+|z_2-z_4|\right)^2\label{eq:bound_copula}.
\end{flalign}From this we see that if either $|z_1-z_3|=O(1)$ but $|z_2-z_4|\rightarrow\pm\infty$, or the other way around, then both, then the bound is finite. Thus, there are only three remaining cases to consider.
\begin{enumerate}    
    \item[Case 1.] If $|z_1-z_3|=C$ but $|z_2-z_4|\rightarrow\pm\infty$, then the dominant term is of the form 
  \begin{flalign*}
  \mathcal{R}(z,\theta)
\lesssim   &   \exp\{-(z_2-z_4)^2+2C|z_2-z_4|\}\times \left(C+|z_2-z_4|^2+2C|z_2-z_4|\right)\\\lesssim&\exp\{-(z_2-z_4)^2\}|z_2-z_4|^2.
  \end{flalign*}Clearly, this term converges to zero as $|z_2-z_4|\rightarrow\infty$.
   \item[Case 2.]  If $|z_2-z_4|=C$ but $|z_1-z_3|\rightarrow\pm\infty$, then from the symmetry of the bound in equation \eqref{eq:bound_copula} the result follows precisely as in Case 1.
     \item[Case 3.] If $|z_2-z_4|\rightarrow\infty$ and $|z_1-z_3|\rightarrow\pm\infty$, then,  the dominant term can be written as 
     $$
     \exp\{-(z_1-z_3)^2\}\exp\{-[|z_1-z_3|-|z_2-z_4|]^2\}(|z_1-z_3|+|z_2-z_4|\}^2.
     $$ Clearly, if $(|z_1-z_3|-|z_2-z_4|)^2\rightarrow\infty$, the result follows. If $|z_1-z_3|-|z_2-z_4|=o(1)$, then the dominant terms becomes 
     $$
     \exp\{-(z_1-z_3)^2\}(|z_1-z_3|+|z_2-z_4|\}^2.
     $$However, as $|z_1-z_3|,|z_2-z_4|\rightarrow \infty$ this term converges to zero. 
\end{enumerate}
\end{proof}

\subsection{Extension to sub-sampling}
\label{sec:subsampling}

Following on from the discussion in Section \ref{sec:previous}, for both the $\beta$-divergence loss Lemma \ref{lem:betaD_direct} and MMD loss Lemma \ref{lem:MMD_indirect}, the function $\varphi_{b,n}(\theta)$ satisfying \eqref{eq:needed} contains an average over the number of observations $n$. That is to say that in both cases we can write $\varphi_{b,n}(\theta) = \frac{1}{n}\sum_{i=1}^n\varphi_{b,i}(\theta)$ and as a result, $\frac{1}{\tilde{n}}\sum_{i\in B_{\tilde{n}}}\varphi_{b,i}(\theta)$, 
where $B_{\tilde{n}}\subset\{1, \ldots, n\}$ is a subset of indices $1, \ldots, n$  of size $\tilde{n}$ drawn uniformly at random, is an unbiased estimate of $\frac{\partial}{\partial\theta}\MD_n(\theta)$ and therefore also satisfies \eqref{eq:needed}. 

Defining $\Psi_{b,\tilde{n}}'(\theta) := - \frac{\partial}{\partial\theta}\log \pi(\theta) + \omega  \frac{1}{\tilde{n}}\sum_{i\in B_{\tilde{n}}}\varphi_{b,i}(\theta) $ 
we can construct a zig-zag algorithm with  Poisson switching rate functions ${\tilde{\lambda}}^{\tilde{n}}_{j}(\theta,\nu) := \max\left\{ \nu_j \cdot \left[\Psi_{b,\tilde{n}}'(\theta + \nu \cdot t)\right]_j, 0 \right\}$ that continues to sample from $\pi(\theta \mid \MD_n)$ but uses observation sub-sampling as well as pseudo-observation generation to unbiasedly estimate $\frac{\partial}{\partial\theta}\MD_n(\theta)$. This follows the original sub-sampling construction in \cite{bierkens2019zig} (their Theorem 4.1), and is presented in Algorithm \ref{Alg:ZigZag_EstimatedLossSubSampling}.

\begingroup
\singlespacing
\begin{algorithm}[H]
\caption{Zig-zag sampler for $\pi(\theta\mid\MD_n)$ using unbiasedness condition \eqref{eq:needed} and sub-sampling observations}\label{Alg:ZigZag_EstimatedLossSubSampling}
\KwIn{Initial $(\theta^{(0)},\nu^{(0)})$, $\tau^0 = 0$,  $b \in \mathbb{N}$, $\varphi_{b, n} = \frac{1}{n}\sum_{i=1}^n\varphi_{b,i}(\theta)$  satisfying \eqref{eq:needed},   $\hat{\Lambda}_j$ satisfying \eqref{eq:computational-bound-estimated} and batch size $1 \leq \tilde{n}\leq n$} 
\vspace*{0.1cm}
\For{$k=0,1,\ldots$}{
    %
    %Sample $\tau_j \sim \mathbb{P}(\tau_j \geq t) = \exp\left\{-\int_0^t\hat{\Lambda}_j(s; \: \theta^{(\tau^{k})}, \nu^{(\tau^{k})}) ds\right\}$ for $j=1,2,\dots, d$ \\
    Sample $\tau_j \sim \mathbb{P}(\tau_j \geq t) = \exp\left\{-\int_0^t\hat{\Lambda}_j(\theta^{(\tau^{k})} + \nu^{(\tau^{k})}\cdot s, \nu^{(\tau^{k})}) ds\right\}$ for $j=1,2,\dots, d$ \\
    Set $j^* = \argmin_{j}\tau_j$; update $\tau^{k+1} = \tau^k + \tau_{j^*}$ and $\theta^{(\tau^{k+1})} = \theta^{(\tau^{k})} + \nu^{(\tau^{k})} \cdot  \tau_{j^*}$\\
    Sample $u_{1:b} \sim P_{\theta^{(\tau^{k+1})}}$ and $B_{\tilde{n}}\subset \{1, \ldots, n\}$ uniformly at random; set %$\hat{p}_{k+1} =  \dfrac{\hat{\hat{\lambda}}^{\tilde{n}}_{j^*}(\tau_{j^*}; \theta^{(\tau^k)}, \nu^{(\tau^k)})}{ \hat{\Lambda}_{j^*}(\tau_{j^*}; \theta^{(\tau^k)}, \nu^{(\tau^k)})}$; 
    $\hat{p}_{k+1} =  \dfrac{{\tilde{\lambda}}^{\tilde{n}}_{j^*}(\theta^{(\tau^k)} + \nu^{(\tau^k)}\cdot \tau_{j^*}, \nu^{(\tau^k)})}{ \hat{\Lambda}_{j^*}(\theta^{(\tau^k)} + \nu^{(\tau^k)}\cdot \tau_{j^*}, \nu^{(\tau^k)})}$;
    draw 
    $Z_{k+1} \sim \operatorname{Ber}(\hat{p}_{k+1})$
    \\
    \eIf{$Z_{k+1} = 1$}{Set $\nu_{j^*}^{(\tau^{k+1})} = - \nu_{j^*}^{(\tau^{k})}$ and $\nu_{j}^{(\tau^{k+1})} =  \nu_{j}^{(\tau^{k})}$ for $j \neq j^*$}{Set $\nu^{(\tau^{k+1})} = \nu^{(\tau^{k})}$}
}
\Return $(\tau^k, \theta^{(\tau^k)}, \nu^{(\tau^k)})_{k=0}^{\infty}$
\end{algorithm}
\endgroup

\section{Supplementary Results and Figures}

\subsection{Block P-MCMC Algorithmic Details}\label{app:bPMCMC}

Denote $u_{1:m}^{1:J}$ as the random numbers required to estimate the log-loss $\MD_{m,n}(\theta, u_{1:m}^{1:J})$.  For the regression examples, we have $J=n$ and for the copula example we have $J=1$.  We employ different blocking strategies for the regression and copula examples.  For the regression examples, at each iteration of MCMC, we update the random numbers $u_{1:m}^j$ for some $j \in \{1,\ldots,J\}$, which is drawn from a discrete uniform distribution.  We do this since often $n \gg m$ and hence fewer random numbers are updated, increasing the correlation between log-loss estimates. For the copula example, we update the random numbers $u_a^{1:J}$ for some $a \in \{1,\ldots,m\}$ (recalling that $J=1$).  For the copula example, each $u_a^{1:J}$ is only 2-dimensional, and hence only 2 random numbers are changed at each MCMC iteration.  For more details on the block P-MCMC we employ, see Algorithm \ref{Alg:block_pmcmc}.

\begingroup
\singlespacing
\begin{algorithm}[H]
\caption{Block P-MCMC sampler for $\pi(\theta\mid\MD_{m,n})$}\label{Alg:block_pmcmc}
\KwIn{Initial $\theta^{(0)}$, $m \in \mathbb{N}$, random numbers for estimating the log-loss $u_{1:m}^{1:J}$, learning rate $w$, number of MCMC iterations $S$,  random walk covariance matrix $\Sigma$.} 
\vspace*{0.1cm}
\For{$s=1,2,\ldots,S$}{
    Set proposed $\tilde{u}_{1:m}^{1:J} = u_{1:m}^{1:J}$ \\ 
    For the regression examples, sample $k$ from a discrete uniform distribution $j \sim \mathcal{U}(1,J)$, draw $\tilde{u} \sim p(u_{1:m})$ and set $\tilde{u}_{1:m}^{j} = \tilde{u}$ \\ 
    For the copula example, sample $a$ from a discrete uniform distribution $a \sim \mathcal{U}(1,m)$, draw $\tilde{u} \sim p(u)$ and set $\tilde{u}_{a}^{1:J} = \tilde{u}$ \\
    Propose $\tilde{\theta} \sim \mathcal{N}(\theta^{s-1}, \Sigma)$ and compute the Metropolis-Hastings ratio:
    \begin{align*}
        \alpha &= \frac{\exp\{-w \cdot  \MD_{m,n}(\tilde{\theta},\tilde{u}_{1:m}^{1:J}) \} \pi(\tilde{\theta})}{\exp \{ -w \cdot \MD_{m,n}(\theta^{s-1},u_{1:m}^{1:J}) \} \pi(\theta^{s-1})}
    \end{align*} \\
    \If{$\mathcal{U}(0,1) < \alpha$}{
        Set $\theta^{s} = \tilde{\theta} $ \\
        Set $u_{1:m}^{1:J} = \tilde{u}_{1:m}^{1:J}$ \\
    }\Else{
        Set $\theta^{s} = \theta^{s-1} $ \\
    }
    
}
\Return $\{ \theta_s \}_{s=0}^S$
\end{algorithm}
\endgroup

\subsection{Outlier-Robust Inference in Regression}\label{Supp:FigsMMDReg}

Figure \ref{Fig:regression_misspecification} demonstrates that the tails of the data generated in Section \ref{Sec:RegressionMMD} are heavier than the assumed Gaussian error distribution. 

%PDMPs_MMD_regression2.Rmd
\begin{figure}[!ht]%[t!]
%\vskip -1cm %-0.1in
\begin{center}
\includegraphics[trim= {0.0cm 0.00cm 0.0cm 0.0cm}, clip,  width=0.49\columnwidth]{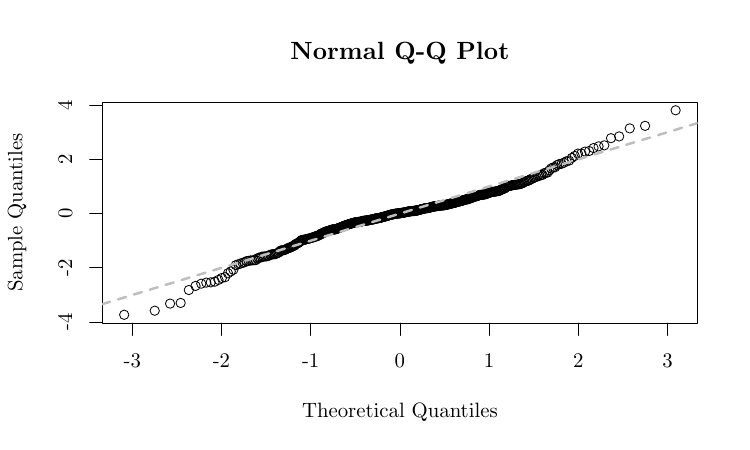}
\includegraphics[trim= {0.0cm 0.00cm 0.0cm 0.0cm}, clip,  width=0.49\columnwidth]{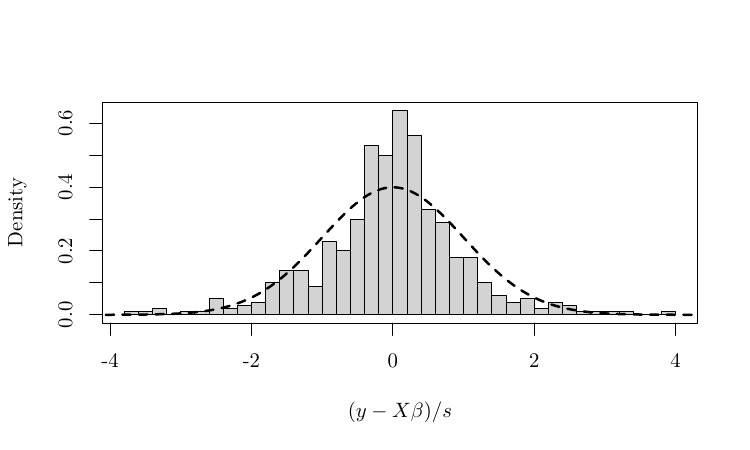}
%trim={<left> <lower> <right> <upper>}
\caption{Q-Q Normal plot (left) and histogram (right) of standardised residuals estimated by ordinary least squares compared with the standard normal distribution.
}
\label{Fig:regression_misspecification}
\end{center}
\end{figure}

\subsection{MMD regression}
\label{sec:full-MMD-regression-comparison}

\subsubsection{Outlier Robust linear regression with the MMD}{\label{Sec:RegressionMMD}}

Here, we consider Gibbs measures using MMD-based losses as constructed in \citet{alquier2024universal} using a radial basis kernel with lengthscale parameter $\gamma = 1$. 
We model $y_i \mid x_i, \beta, \sigma^2 \sim \mathcal{N}(x_i^{\top}\beta, \sigma^2)$ for $x_i \in \mathbb{R}^p$, $y_i \in \mathbb{R}$, and wish to conduct robust posterior inference on $\theta = (\beta^\top, \log \sigma)^\top$ given the priors $\beta \sim \mathcal{N}(0, 25 \cdot I_p)$ and $\sigma^2 \sim \mathcal{IG}(2,0.5)$.

To generate data, we take $x_{i1} = 1$ and sample $(x_{i2},\dots,x_{i7})^\top \overset{\text{iid}}{\sim} \mathcal{N}(0, I_7)$ for all $i=1,2,\dots n$ with $n=100$.
Setting $\beta =  (4, 4, 3, 3, 2, 2, 1, 1)^\top$, we then generate the response variables as $y_i = x_i^{\top}\beta +\varepsilon_i$, where $\varepsilon_i \overset{\text{iid}}{\sim}  \operatorname{DExp(0,1)}$ for $\operatorname{DExp(0,1)}$ denoting the double exponential distribution---making the Gaussian regression model for which inference is performed misspecified relative to the true data-generating mechanism (cf.  \Cref{Supp:FigsMMDReg}).
Our implementation of the zig-zag sampler for this example relies on the representation in \eqref{eq:candidate-indirect-sampling}, and the corresponding computational bounds are derived in \Cref{cor:regression-MMD} of \Cref{sec:MMDappendix}, which also contains additional  details.
.

Figure \ref{Fig:regression_PM_vs_ZZ_n100_beta_sigma} compares the posterior approximations for $\beta_1$ and $\log(\sigma)$ (see \Cref{sec:full-MMD-regression-comparison} for a full comparison). 
As expected, the results mirror those in Section \ref{sec:mmd-copula}: the choice of $b$ is largely inconsequential for the inferences produced by zig-zag sampling, but the bP-MCMC  sampler is quite sensitive to the choice of $m$---particularly for its inferences on $\log(\sigma)$. 
Complementing this, Figure \ref{Fig:regression_PM_vs_ZZ_sims} charts the approximation error for posterior means and standard deviations of $\beta_1$ as a function of the number of simulations from the model used.
Once again, the findings are similar as in the copula example.

%PDMPs_MMD_regression2.Rmd
\begin{figure}[!ht]%[t!]
%\vskip -1cm %-0.1in
\begin{center}
\includegraphics[trim= {0.0cm 0.00cm 0.0cm 0.0cm}, clip,  width=0.35\columnwidth]{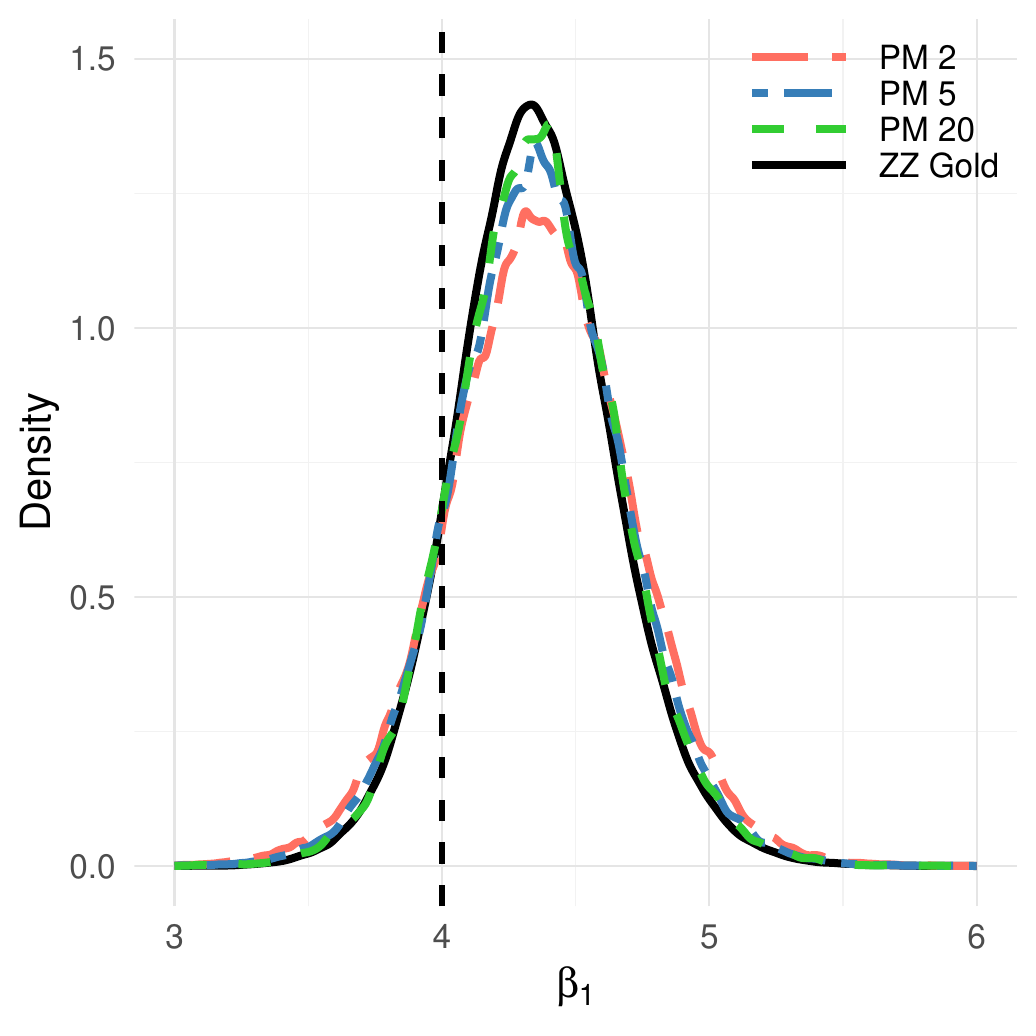}
\includegraphics[trim= {0.0cm 0.00cm 0.0cm 0.0cm}, clip,  width=0.35\columnwidth]{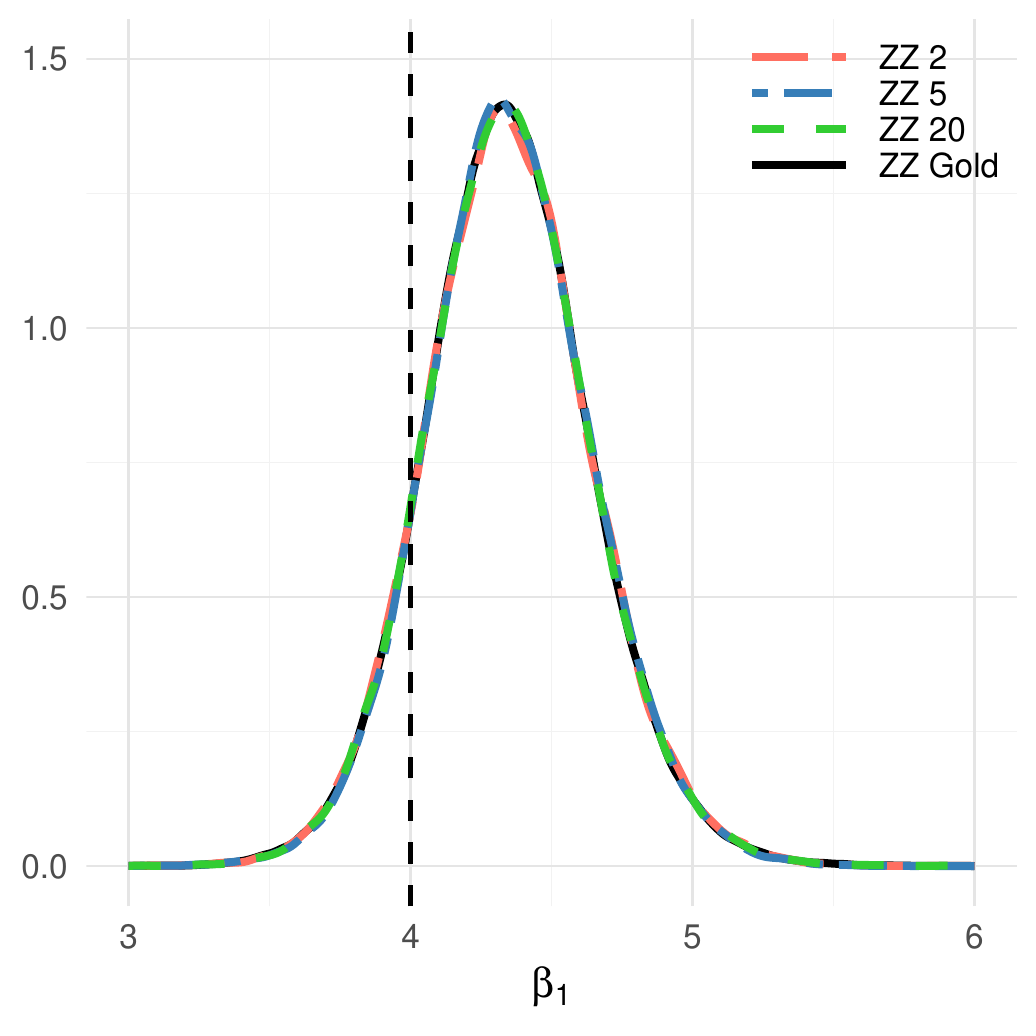}\\
\includegraphics[trim= {0.0cm 0.00cm 0.0cm 0.0cm}, clip,  width=0.35\columnwidth]{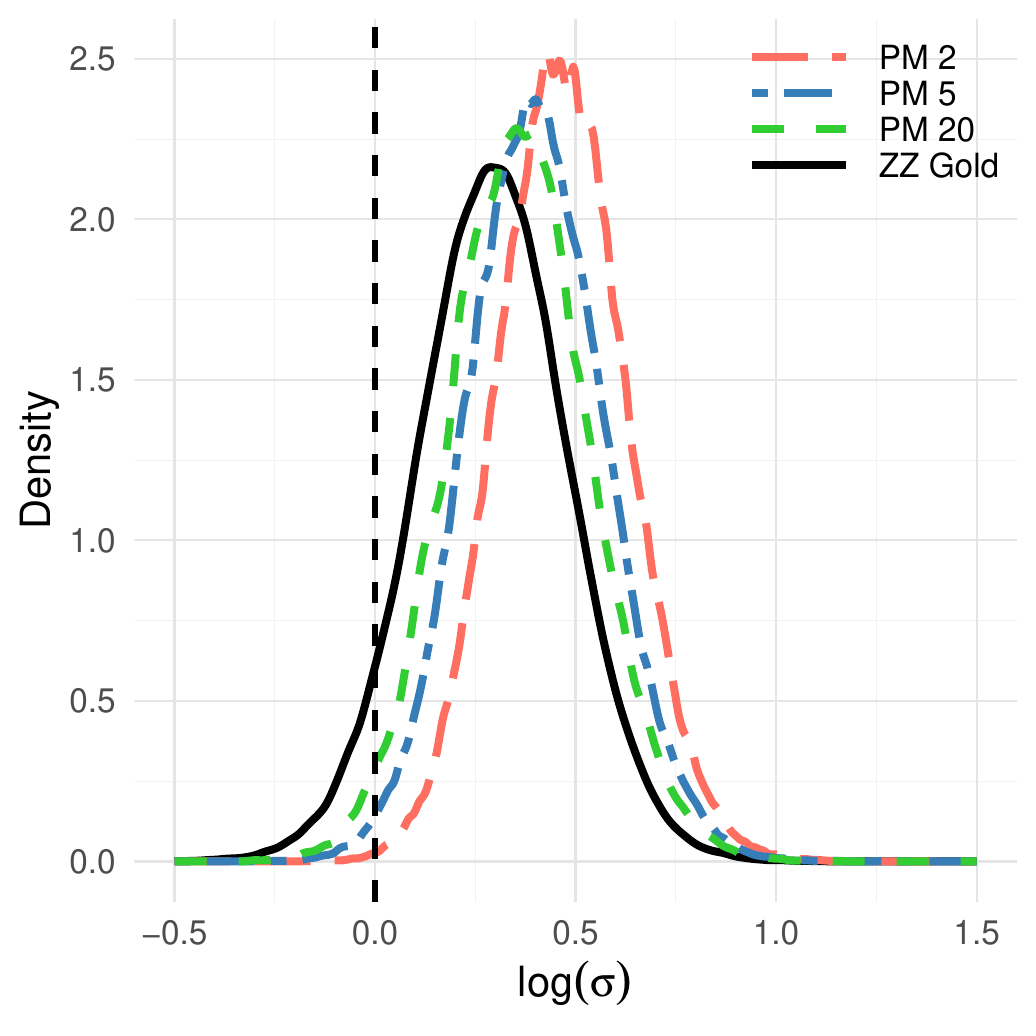}
\includegraphics[trim= {0.0cm 0.00cm 0.0cm 0.0cm}, clip,  width=0.35\columnwidth]{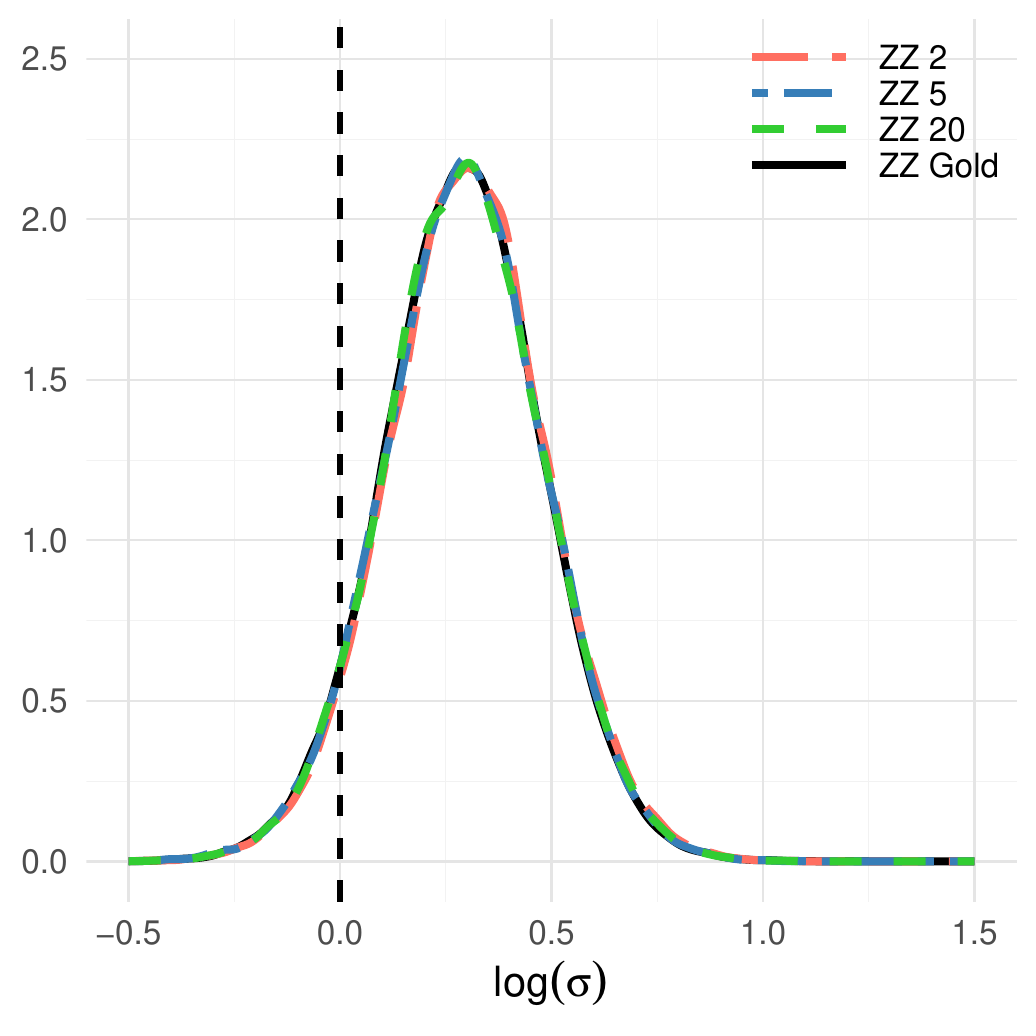}\\
%trim={<left> <lower> <right> <upper>}
\caption{Comparison of $\MMD$-Bayes posterior density for  $\beta_1$ and $\log(\sigma)$  in the Gaussian linear regression model with $n=100$ as produced by bP-MCMC (\textit{PM}) and zig-zag sampling (\textit{ZZ}) for different choices of $m$ and $b$. 
%The figures on the left contain the bP-MCMC posterior approximations and the figures on the right contain the zig-zag posterior approximations. 
Figure legend and interpretation as  in \Cref{Fig:copula_PM_vs_ZZ_n100}.
}
\label{Fig:regression_PM_vs_ZZ_n100_beta_sigma}
\end{center}
\end{figure}

\begin{figure}[!ht]%[t!]
%\vskip -1cm %-0.1in
\begin{center}
\begin{subfigure}[t]{0.45\textwidth}
    \includegraphics[scale=0.42]{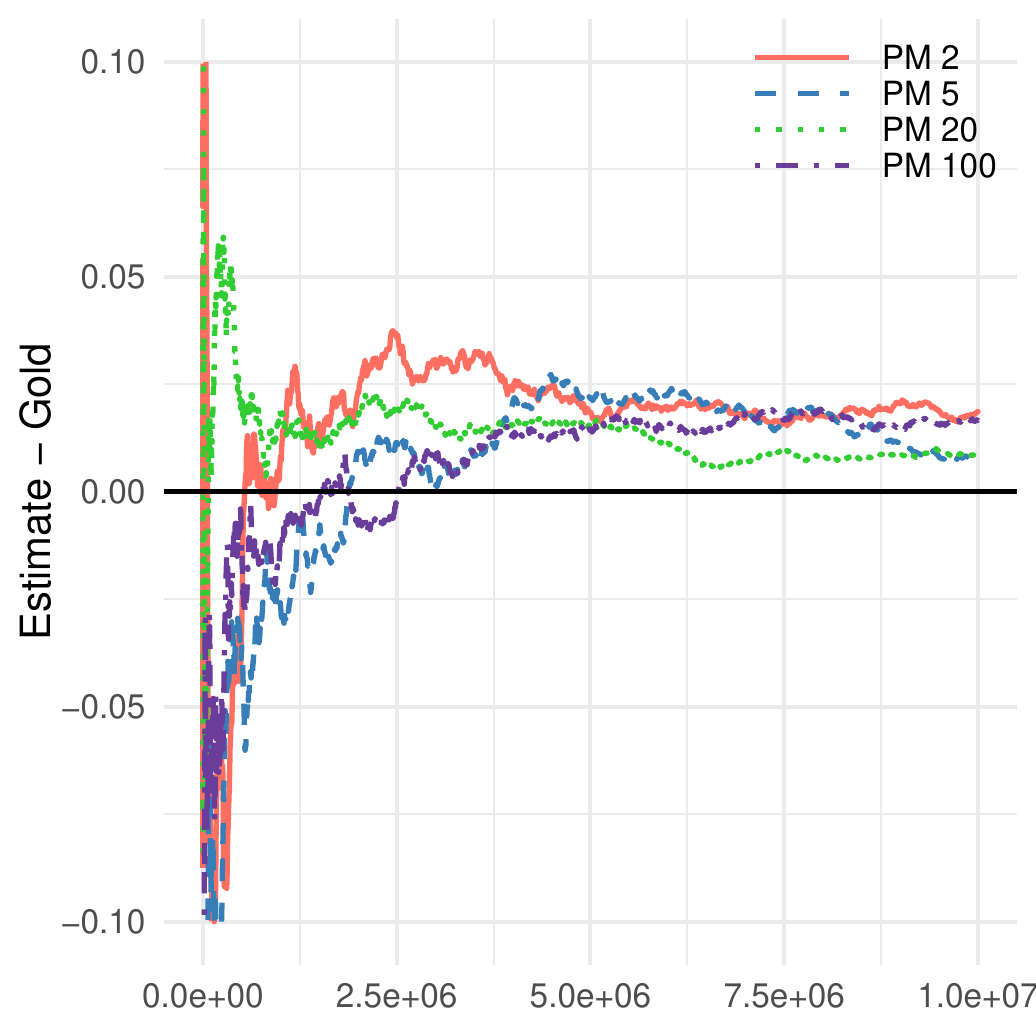}
    \caption{PM, mean}
\end{subfigure}
\begin{subfigure}[t]{0.45\textwidth}
    \includegraphics[scale=0.42]{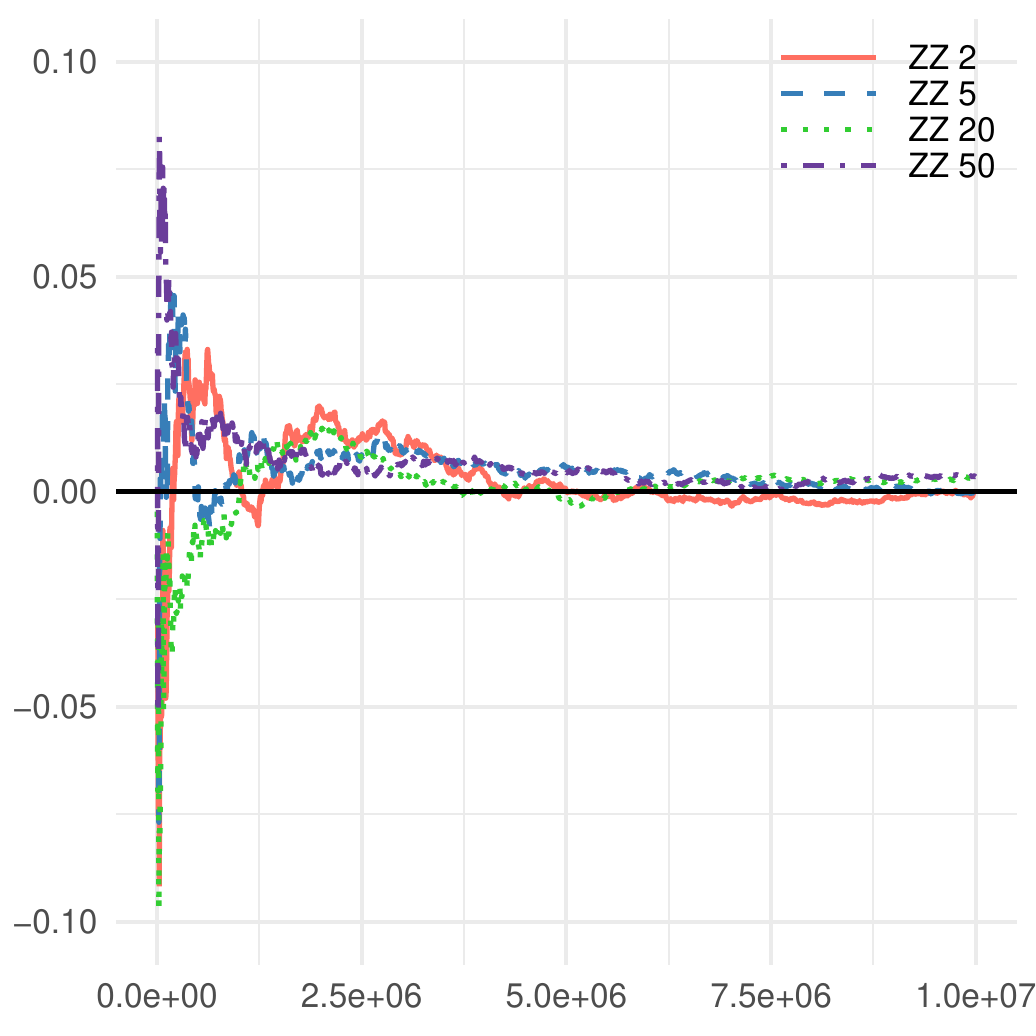}
    \caption{ZZ, mean}
\end{subfigure}
\begin{subfigure}[t]{0.45\textwidth}
    \includegraphics[scale=0.42]{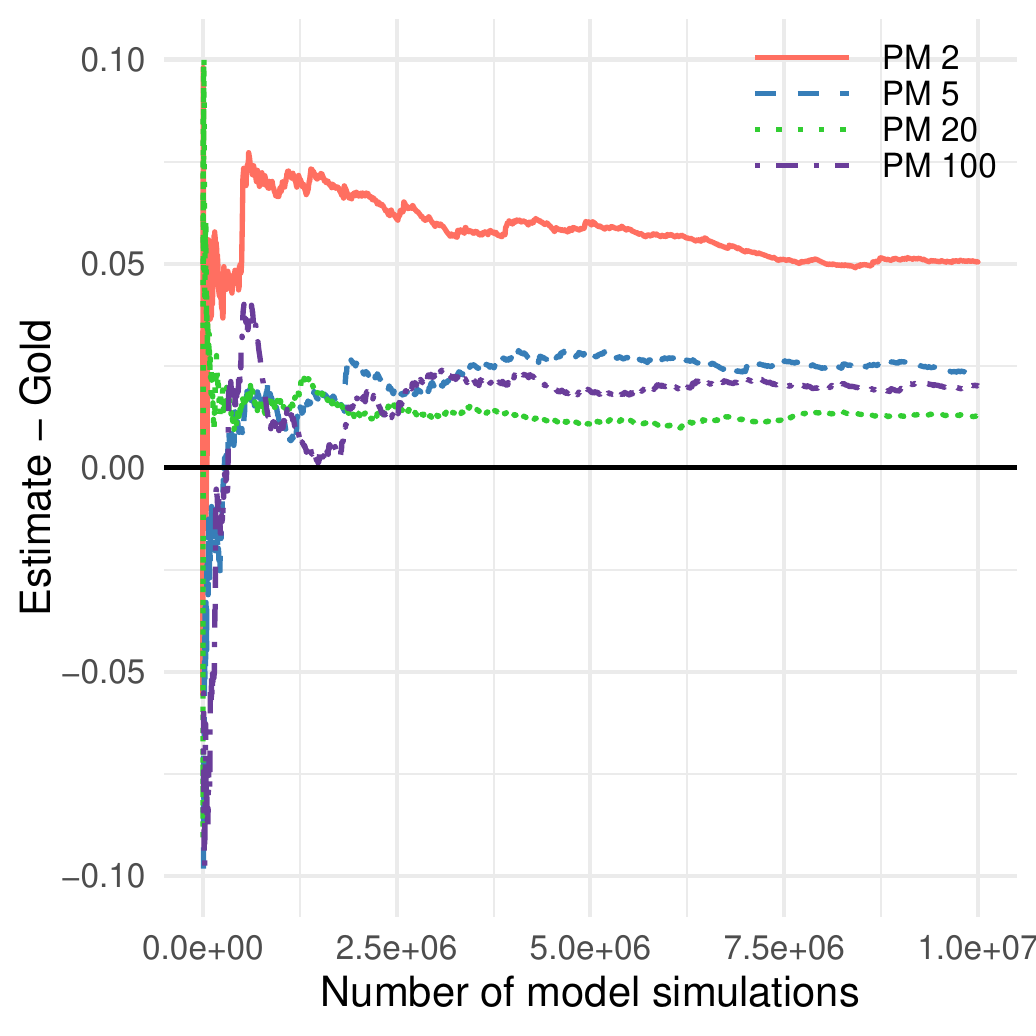}
    \caption{PM, standard deviation}
\end{subfigure}
\begin{subfigure}[t]{0.45\textwidth}
    \includegraphics[scale=0.42]{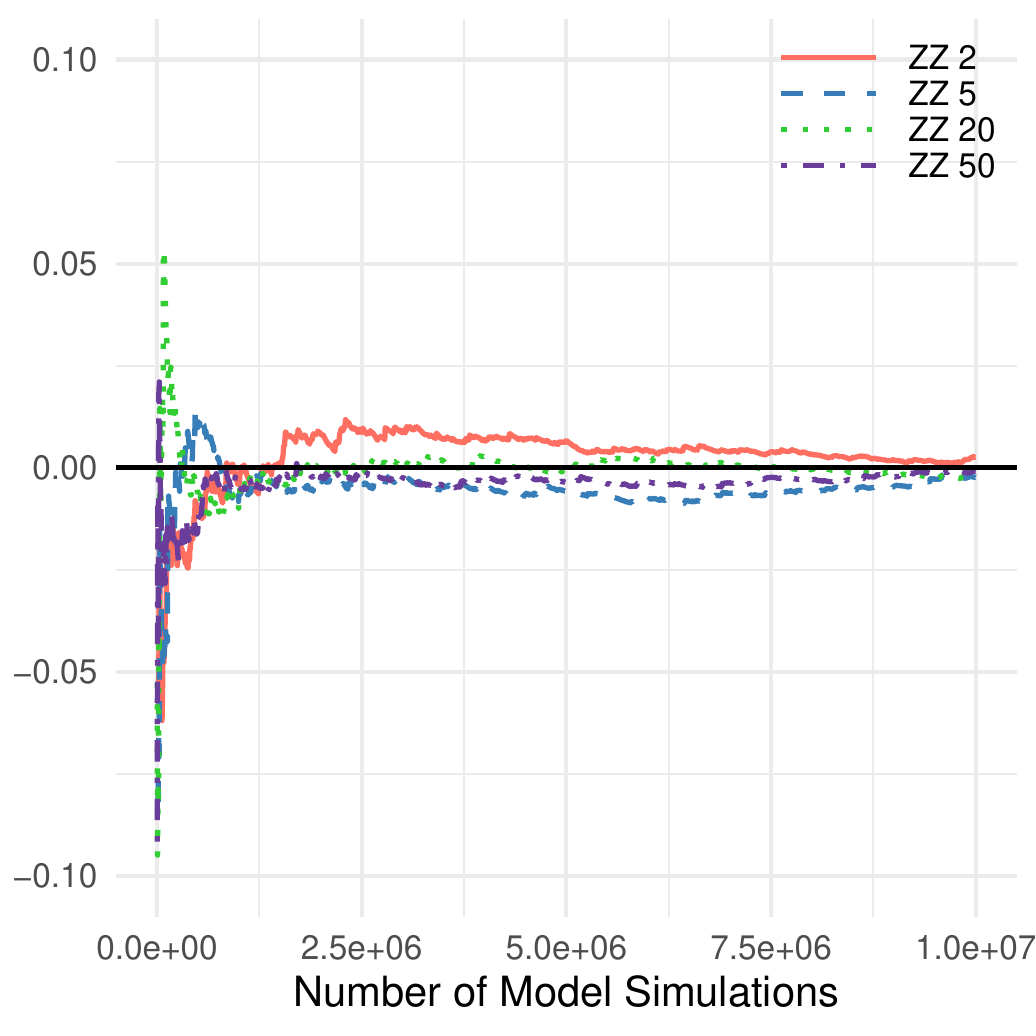}
    \caption{ZZ, standard deviation}
\end{subfigure}
\caption{
Accuracy of bP-MCMC (\textit{PM}) and zig-zag sampling (\textit{ZZ}) as the number of model simulations increases for $\beta_1$ and $n=100$ in the $\MMD$-Bayes regression example. 
Figure legend and interpretation as in \Cref{Fig:copula_PM_vs_ZZ_sims}.}
\label{Fig:regression_PM_vs_ZZ_sims}
\end{center}
\end{figure}

\begin{figure}[H]%[t!]
%\vskip -1cm %-0.1in
\begin{center}
\includegraphics[scale=0.8]{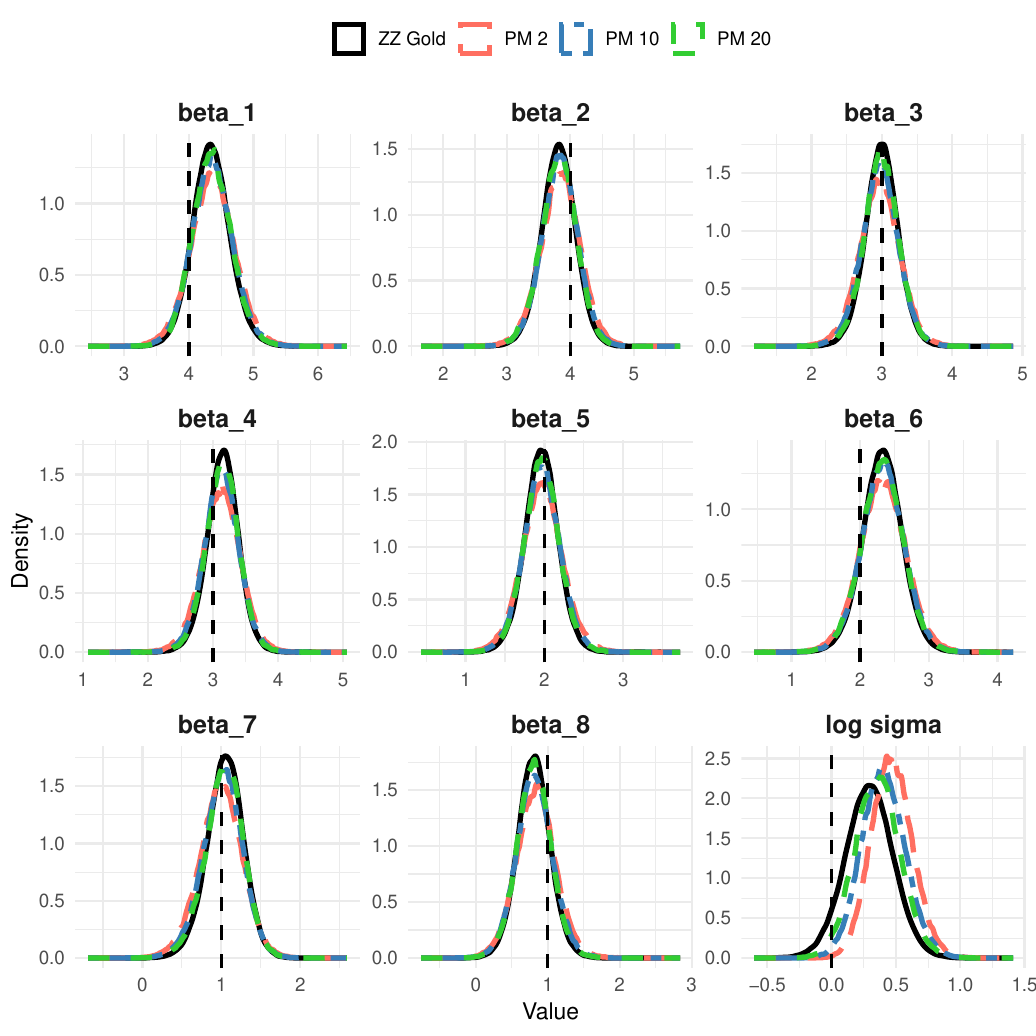}
\caption{$\MMD$-Bayes posterior density for $\beta$ and $\log(\sigma)$ of the Gaussian regression model based on sample size $n=100$ as estimated by a long run of the zig-zag sampler (ZZ Gold) and block pseudo-marginal sampler with $m=2$ (PM 2), $m=10$ (PM 10) and $m=20$ (PM 20).}
\label{Fig:regression_PM_vs_ZZ_n100}
\end{center}
\end{figure}

Figures \ref{Fig:regression_PM_vs_ZZ_n100} and \ref{Fig:regression_ZZ_n100} provide a full comparison of the target posterior densities of the bP-MCMC and zig-zag algorithm for different $m$ and $b$.

\begin{figure}[H]%[t!]
%\vskip -1cm %-0.1in
\begin{center}
\includegraphics[scale=0.8]
{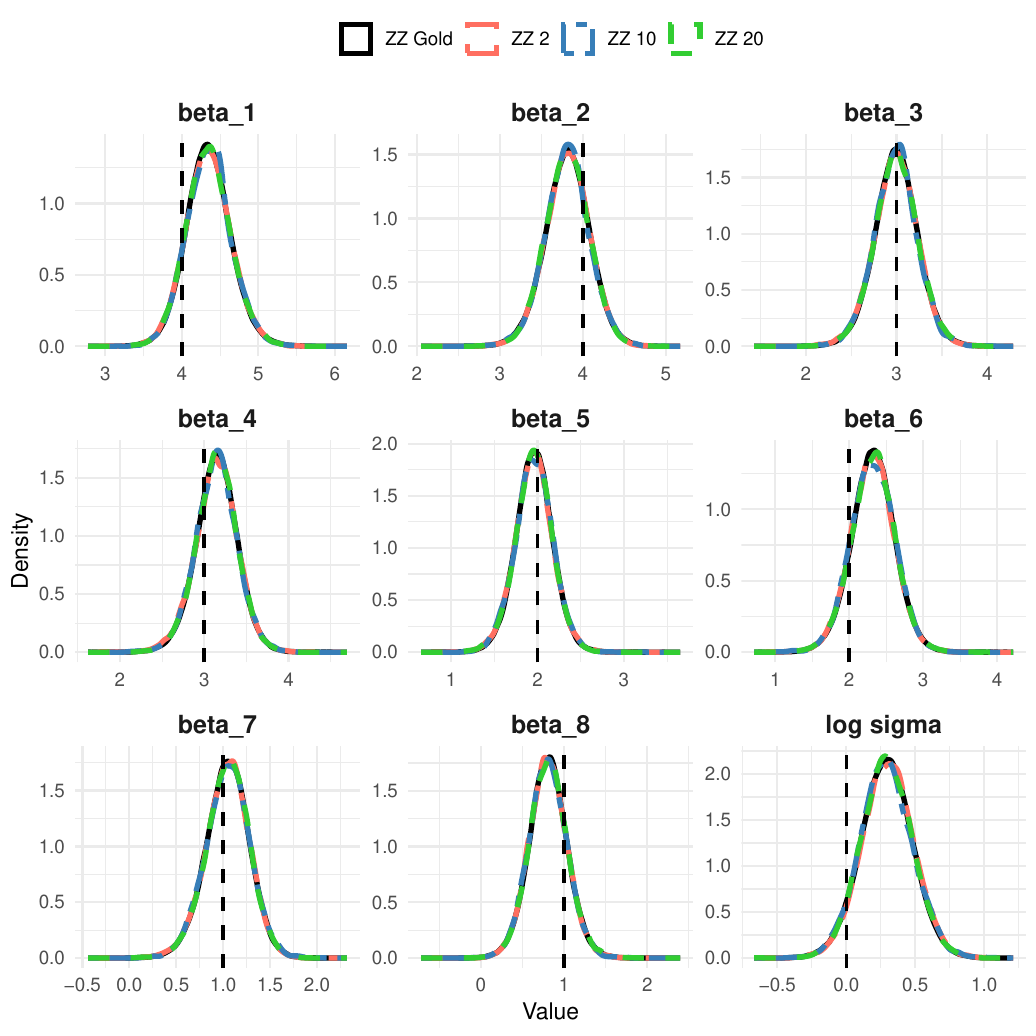}
\caption{$\MMD$-Bayes posterior density for $\beta$ and $\log(\sigma)$ of the Gaussian regression model based on sample size $n=100$ as estimated by the zig-zag sampler with $b=2$ (ZZ 2), $b= 10$ (ZZ 10) and $b = 20$ (ZZ 20). Shown also are the results for a long run of the zig-zag sampler (ZZ Gold)}
\label{Fig:regression_ZZ_n100}
\end{center}
\end{figure}

\subsection{Robust Poisson regression with the $\beta$-Divergence}\label{Sec:PoissonRegression}

In this example we investigate a setting in which the bias of $\overline{\pi}(\theta\mid\MD_{m,n})$ is practically negligible so that bP-MCMC can still compete with the proposed zig-zag sampler, and draw attention to the type of conditions under which this happens.
To this end, we 
study the posteriors based on the $\beta$-divergence loss $\MD_{m,n}^{\beta}$ presented in \Cref{subsec:applicability-of-results} for Poisson regression models, which directly draws on a long line of literature on constructing generalised posteriors using this discrepancy measure \citep[e.g.][]{ghosh2016robust, knoblauch2018doubly, jewson2018principles, futami2018variational}.
For our experiment, we fix $\beta = 0.5$, and model $y_i \mid x_i, \theta \sim \operatorname{Poiss}(e^{x_i^{\top}\theta})$ with a prior on the $p$ regression parameters given by $\theta \sim \mathcal{N}(0, I_p)$.
Artificial data is generated from this model without misspecification by taking $x_{i1} = 1$ and $(x_{i2}, \ldots, x_{i5})\stackrel{iid}{\sim}\mathcal{N}(0, 0.25^2\cdot I_4)$ for all $i=1,2,\dots n$ with $n=1000$, fixing $\theta = (1, 0.5, 1.5, 0, 0)^\top$, and then sampling $y_i \mid x_i, \theta \sim \operatorname{Poiss}(x_i^{\top}\theta)$.

Our implementation of the zig-zag sampler for this example relies on the representation in \eqref{eq:candidate-direct}, and the corresponding computational bounds are derived in \Cref{Cor:betaD_Poisson} of \Cref{sec:betaDappendix}, which also contains additional  details.
A comparison of the  obtained posteriors for the first regression coefficient is presented in \Cref{Fig:poisson_PM_vs_ZZ_n100}, and reveals a clear contrast to the two previous examples: both bP-MCMC and the proposed zig-zag sampler give nearly identical posteriors across a wide range of both  $m$ and $b$.
This behavior holds for the remaining regression coefficients, as well, and unlike in the two previous examples, both methods now also have a comparable computational overhead; see \Cref{Sec:unifGap} for details.

While this seems to fly in the face of our results thus far, there is a clear reason for this 
significant qualitative difference to our previous examples.
In particular, 
for the special case of the $\beta$-divergence loss on Poisson regression, the Jensen's Gap we identified in  \Cref{Sec:research_gap} as the source of the bias for P-MCMC methods can be bounded by $\frac{e}{4m}$ \textit{uniformly} over $\Theta$ ({see \Cref{Sec:unifGap}, and \Cref{Lem:unif_J_Gap} for details}).
As a result, even for $m=100$, the bias of P-MCMC methods in this Poisson regression model will be no larger than about $0.01$ uniformly over $\Theta$.
This is in stark contrast to the previous examples for which the difference cannot be uniformly bounded. 
Conversely, it indicates the possibility of identifying uniformity conditions to make P-MCMC algorithms reasonably practical.
While this may appear appealing at first glance, such uniformity conditions would be highly restrictive, and unlikely to apply to many other settings of interest.
Indeed, even staying within the class of $\beta$-divergence based posteriors, switching from inferences in Poisson regression to normal regression is sufficient to violate such uniformity.

%PDMPs_MMD_regression2.Rmd
\begin{figure}[!ht]%[t!]
%\vskip -1cm %-0.1in
\begin{center}
\includegraphics[trim= {0.0cm 0.00cm 0.0cm 0.0cm}, clip,  width=0.35\columnwidth]{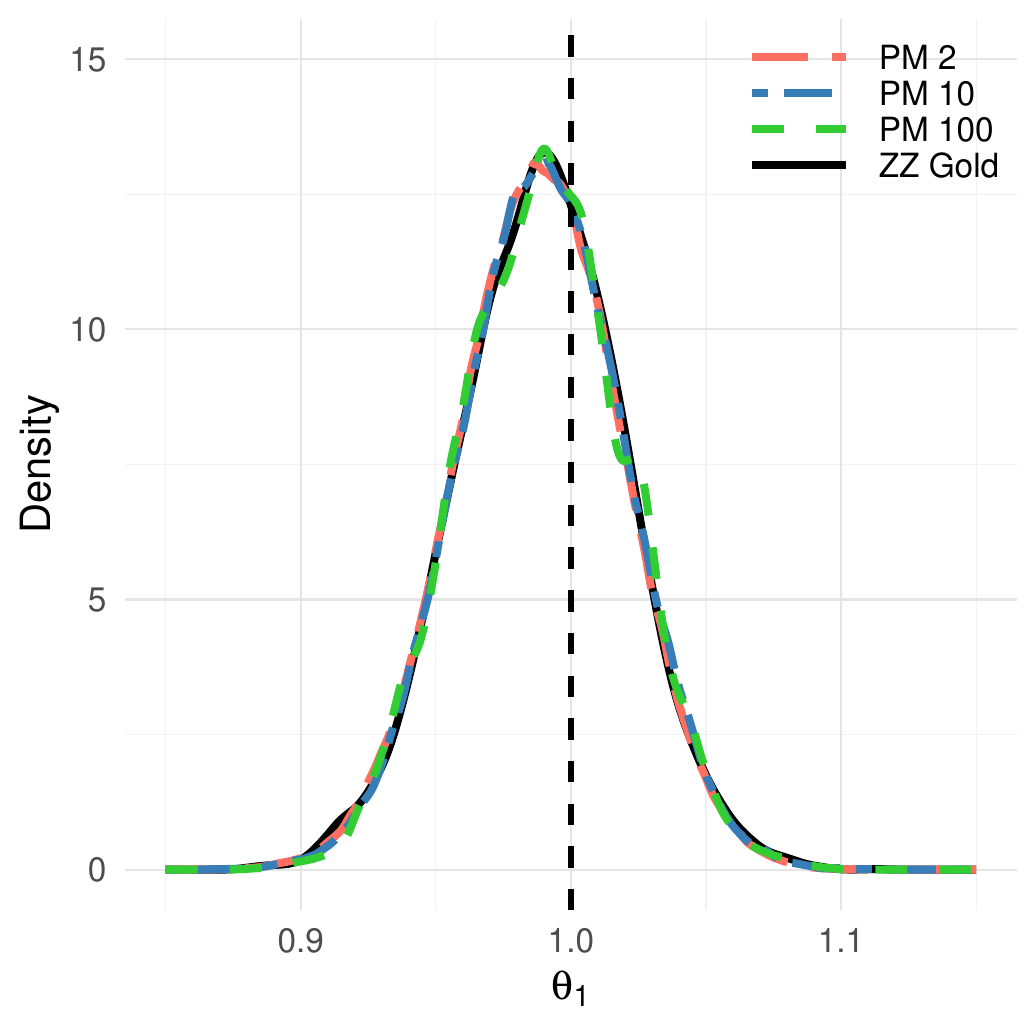}
\includegraphics[trim= {0.0cm 0.00cm 0.0cm 0.0cm}, clip,  width=0.35\columnwidth]{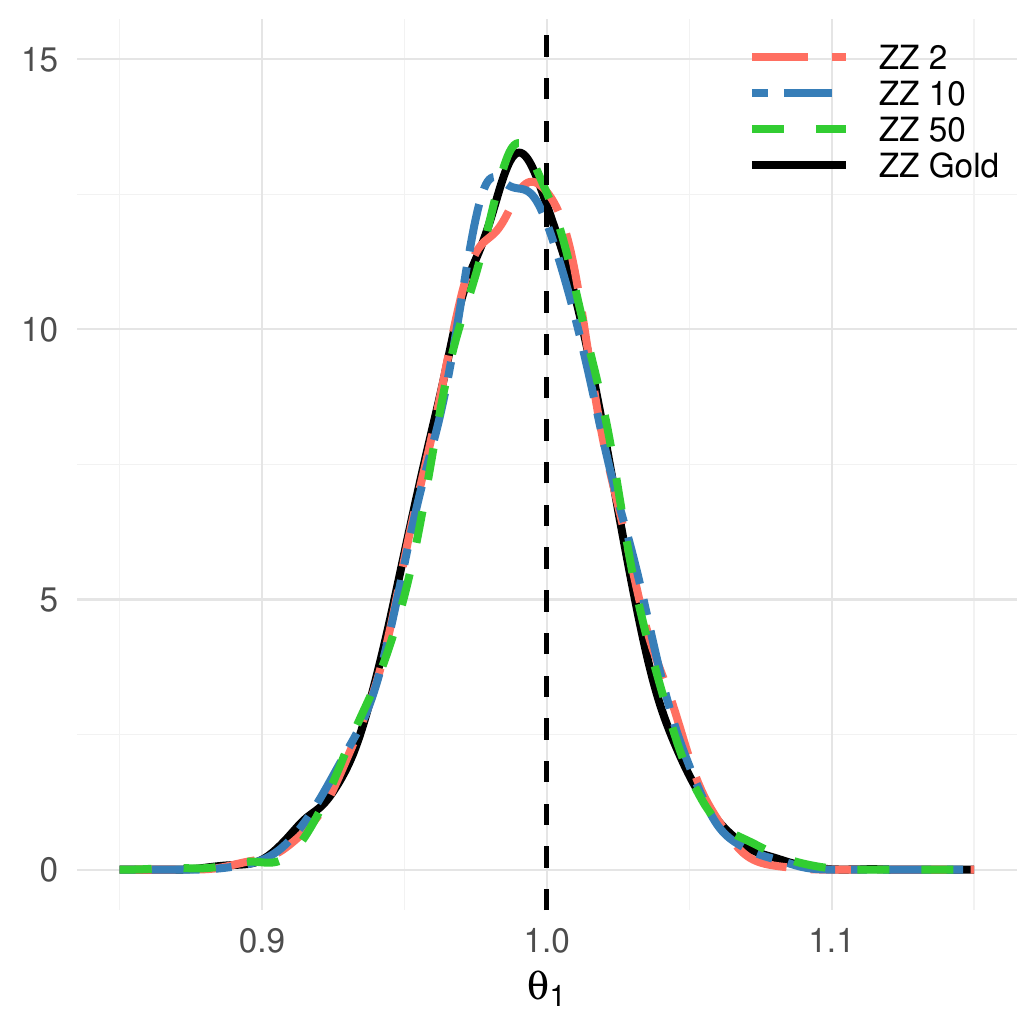}\\
\caption{$\beta$-divergence posterior density for $\theta_1$ in the Poisson regression model.
\textbf{Left} panel contains the bP-MCMC (\textit{PM}) posterior approximations, while the \textbf{right} panel contains results for the zig-zag sampler (\textit{ZZ}). 
Figure legend and interpretation as in \Cref{Fig:copula_PM_vs_ZZ_n100}.  
}
\label{Fig:poisson_PM_vs_ZZ_n100}
\end{center}
\end{figure}

Figures \ref{Fig:regression_PM_vs_ZZ_n1002} and \ref{Fig:regression_ZZ_n1002} provide a full comparison of the target posterior densities of the bP-MCMC and zig-zag algorithm for different $m$ and $b$ in the case of the poisson regression model

\begin{figure}[!ht]%[t!]
%\vskip -1cm %-0.1in
\begin{center}
\includegraphics[scale=0.8]{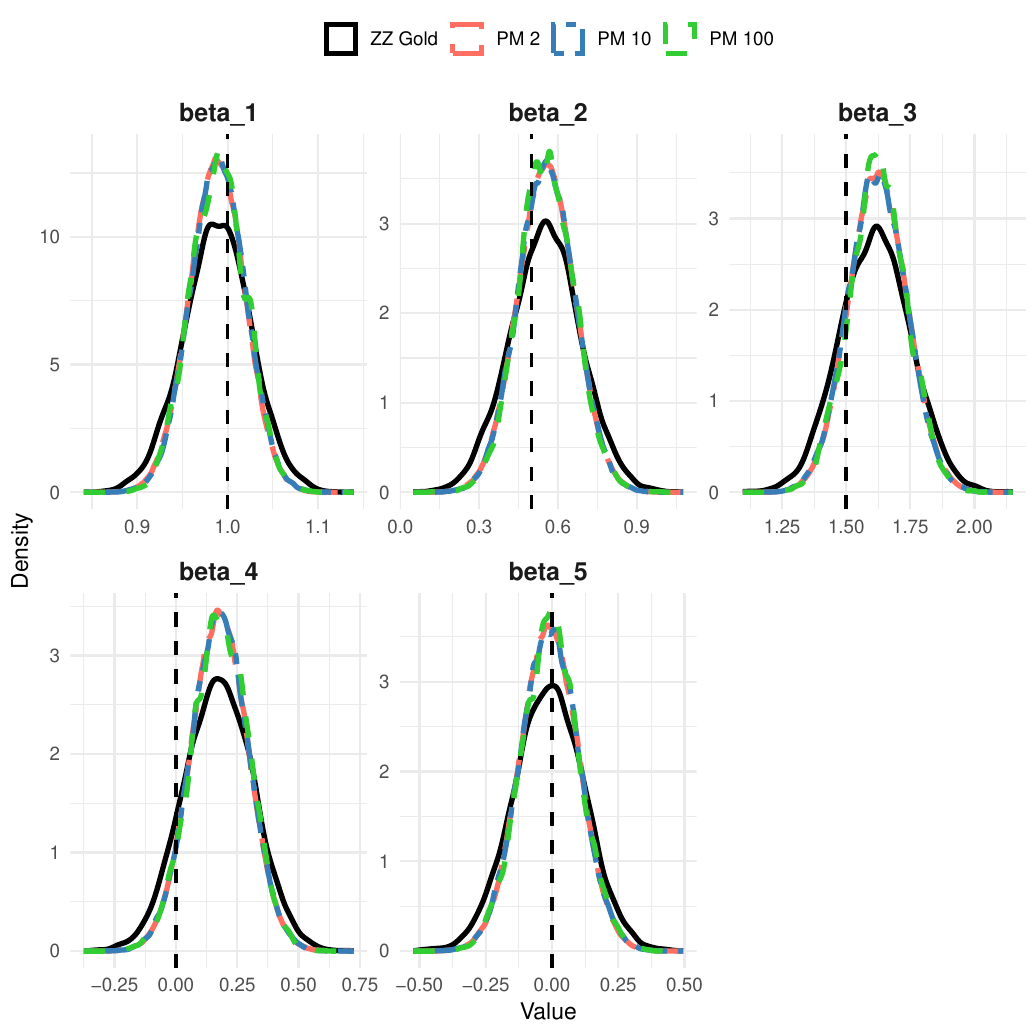}
\caption{$\beta$-divergence-based posterior density for $\theta$ in the Poisson regression model with sample size $n=1000$ as estimated by a long run of the zig-zag sampler (ZZ Gold) and block pseudo-marginal sampler with $m=2$ (PM 2), $m=10$ (PM 10) and $m=20$ (PM 20).}
\label{Fig:regression_PM_vs_ZZ_n1002}
\end{center}
\end{figure}

\begin{figure}[!ht]%[t!]
%\vskip -1cm %-0.1in
\begin{center}
\includegraphics[scale=0.8]
{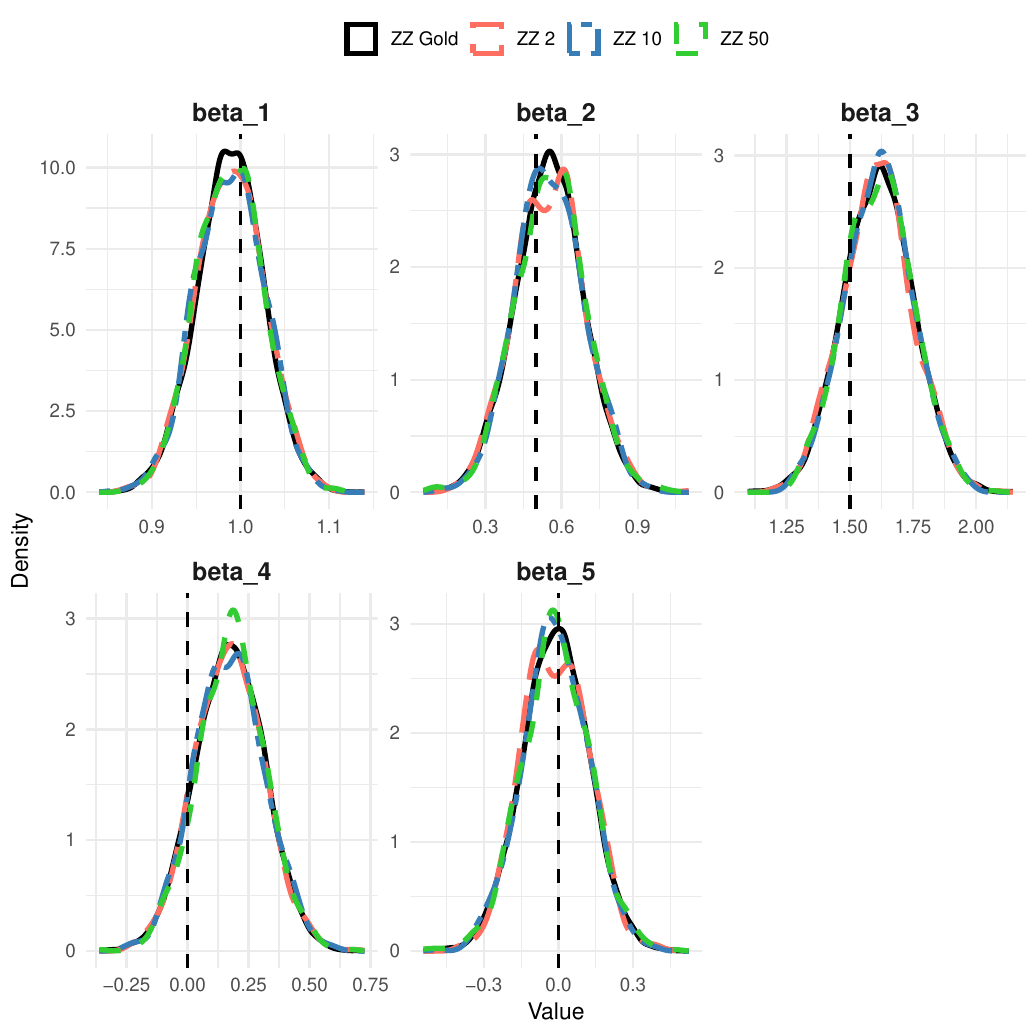}
\caption{$\beta$-divergence-based posterior density for $\theta$ in the Poisson regression model with sample size $n=1000$ as estimated by the zig-zag sampler with $b=2$ (ZZ 2), $b= 10$ (ZZ 10) and $b = 20$ (ZZ 20). Shown also are the results for a long run of the zig-zag sampler (ZZ Gold)}
\label{Fig:regression_ZZ_n1002}
\end{center}
\end{figure}

\subsubsection{Uniform Accuracy of Certain Posterior Approximations}\label{Sec:unifGap}
In \Cref{Sec:PoissonRegression}, we showed empirically that the difference between $\overline\pi(\theta\mid\MD_{m,n})$ and $\pi(\theta\mid \MD_n)$ was negligible in the specific case of Poisson regression. In this section, we show that this behavior is due to Jensen's gap between $\overline\pi(\theta\mid\MD_{m,n})$ and $\pi(\theta\mid \MD_n)$ being quite small in this example.

In Poisson regression the mass function for $y_i$, conditional on predictor variables $x_i$, is given by $p_{\theta}(y_i; x_i)=[e^{x_i^\top\theta}]^{y_i}e^{-e^{x_i^\top\theta}}/y_i!$, and the $\beta$D-loss involves the intractable expectation:
$$
%\sum_{i=1}^{n}\int p_\theta (y;X_i)^\beta p_\theta(y;X_i)\dt y=
\sum_{i=1}^{n}\E_{u\sim p_\theta(\cdot;x_i)}[p_\theta (u;x_i)^\beta]=\sum_{i=1}^{n}\sum_{y=0}^{\infty}p_\theta (y;x_i)^{1+\beta},
$$
which for $\beta > 0$ is not available in closed form. However, it is simple to estimate this term via Monte Carlo sampling by drawing $U_j\stackrel{iid}{\sim} p_\theta(\cdot;x_i)$, for $j=1,\dots,m$, and then using the estimated loss function 
$$
\MD_{m, n}^\beta\left(\theta, y_{1: n}\right)=\frac{1}{n}\sum_{i=1}^{n}\frac{1}{m} \sum_{j=1}^m p_\theta\left(u_j;x_i\right)^\beta-\left(1+\frac{1}{\beta}\right) \frac{1}{n}\sum_{i=1}^n p_\theta\left(y_i;x_i\right)^\beta
$$
to conduct inference on $\theta$.

The accuracy of bP-MCMC in this example is related to the specific nature of the Poisson regression model and the $\beta$-divergence loss. To see why, recall that the difference between the exact and P-MCMC posteriors is the result of Jensen's inequality, see Section \ref{Sec:research_gap}; also, recall that only the first term in the $\beta$-divergence loss is estimated (i.e., the term $\sum_{y=0}^{\infty}p_\theta(y,x_i)^{1+\beta}$). Thus, the difference between $\pi(\theta\mid\MD_n)$ and $\overline\pi(\theta\mid \MD_{m,n})$ only depends on this single term, and the difference between the kernel of the posteriors is described by Jensen's gap:
\begin{flalign*}
\text{JGap}_m(\theta):=& \E_{u_{1:m}\sim P_\theta}\exp\left\{\sum_{i=1}^{n}\frac{1}{m}\sum_{j=1}^{m}p_\theta(u_j;x_i)^\beta\right\}-\exp\left\{\sum_{i=1}^{n}\E_{u_{1:m}\sim P_\theta}p_\theta(u_j;x_i)^\beta\right\}\\&\ge0.
\end{flalign*}
The following result shows that a uniform bound on $\text{JGap}_m(\theta)$ can be derived that depends only on $m$ and is quite small in our examples. 

\begin{lemma}\label{Lem:unif_J_Gap}
Under the setup described in \Cref{Sec:PoissonRegression}, $
\sup_{\theta\in\Theta}\mathrm{JGap}_m(\theta)
\times\pi(\theta)\le 1.6\times \frac{e}{4m}.
$    
\end{lemma}
\begin{proof}
We upper bound $\text{JGap}_m(\theta)$ using Theorem 1 of \cite{liao2018sharpening}. To this end, set $Z_m=\frac{1}{m}\sum_{j=1}^{m}p_\theta(u_j;x_i)^\beta$ and note that since $p_\theta(u_j,x_i)^\beta\in(0,1)$ for all $u_j$ and $x_i$ - since $p_\theta(u_j,x_i)$ is a probability mass function - we have that $Z_m\in(0,1)$. Further, note that $\text{Var}(Z_m)=\text{Var}\{p_\theta(u;x_i)^\beta\}/m$ since each $u_j$ is iid. Now, apply Theorem 1 of \cite{liao2018sharpening} and Popoviciu's inequality for bounded random variables to see that
$$
\text{JGap}_m(\theta)\le \sup_{Z_m\in(0,1)}\exp\{Z_m\}\text{Var}(Z_m)\le \frac{e}{4m}.
$$    Since $\pi(\theta)=N(\theta;0,0.25^2)$, we have that 
$$
\sup_{\theta\in\Theta}\mathrm{JGap}_m(\theta)
\times\pi(\theta)\le \sup_{\theta\in\Theta}\pi(\theta)\times \frac{e}{4m}\le 1.6 \frac{e}{4m}
$$
\end{proof}
To see the benefit of this result, consider that we take $m=100$ random numbers to approximate the loss, as in our experiments, we see that $\sup_{\theta\in\Theta}\text{JGap}_m(\theta)\le 0.01$. Hence, in this example Jensen's gap is close to zero uniformly over $\theta$ and the two posteriors will be very similar.

\end{document}